\documentclass[12pt]{amsart}

\usepackage{amsmath,amsthm,amscd,amsfonts,amssymb,epic,eepic,bbm,tikz-cd}
\usepackage[pagebackref,colorlinks=true,linkcolor=blue,citecolor=blue]{hyperref}
\usepackage[capitalise]{cleveref}
\usepackage{mathtools}

\allowdisplaybreaks
\newtheorem{theorem}{Theorem}[section]
\newtheorem{lemma}[theorem]{Lemma}

\theoremstyle{definition}

\newtheorem*{remark}{Remark}
\numberwithin{equation}{section}

\newcounter{remarkscounter}
\newenvironment{remarks}
{\medskip\noindent{\it
Remarks.}\begin{list}{{\rm(\arabic{remarkscounter})}
}{\usecounter{remarkscounter}

\setlength{\labelsep}{\fill} \setlength{\leftmargin}{0pt}
\setlength{\itemindent}{\fill}
\setlength{\labelwidth}{\fill}\setlength{\topsep}{0pt}
\setlength{\listparindent}{0pt}}} {\end{list}}

\numberwithin{equation}{section}
\newcommand{\A}{\mathbb{A}}
\newcommand{\C}{{\mathbb C}}

\newcommand{\GL}{\mathrm{GL}}

\newcommand{\SL}{\mathrm{SL}}
\newcommand{\Sp}{\mathrm{Sp}}
\renewcommand{\S}{{\mathcal S}}

\newcommand{\Z}{{\mathbb Z}}

\newcommand{\SO}{\mathrm{SO}}
\newcommand{\lto}{\longrightarrow}

\newcommand{\OO}{\mathcal{O}}
\newcommand{\CC}{\mathbb{C}}
\newcommand{\RR}{\mathbb{R}}
\newcommand{\GG}{\mathbb{G}}

\newcommand{\quash}[1]{}

\theoremstyle{definition}

\newenvironment{psmatrix}
  {\left(\begin{smallmatrix}}
  {\end{smallmatrix}\right)}

\renewcommand{\bar}{\overline}

\newcommand{\JS}[1]{}

\renewcommand{\hat}{\widehat}

\newcommand{\one}{\mathbbm{1}}

\allowdisplaybreaks

\usepackage{microtype}

\makeatletter
\@namedef{subjclassname@2020}{\textup{2020} Mathematics Subject Classification}
\makeatother

\begin{document}

\title{Automorphic-twisted summation formulae for pairs of quadratic spaces}
\author{Miao (Pam) Gu}
\address{Department of Mathematics\\
University of Michigan\\
Ann Arbor, MI 48109-1043}
\email{pmgu@umich.edu}
\subjclass[2020]{Primary: 11F70, Seconday: 11F66.}
\keywords{Poisson summation conjecture, Eisenstein series, theta series}

\begin{abstract}
Motivated by the conjectures of Braverman-Kazhdan, Lafforgue, Ng\^{o}, and Sakellaridis, we prove a summation formula for certain spaces of test functions on the zero locus of a quadratic form. The functions are built from the Whittaker coefficients of automorphic representations on $\GL_n$. We also give an expression of the local factors where all the data is unramified. 
\end{abstract}

\maketitle

\tableofcontents

\section{Introduction} \label{1}

In this paper, we prove a summation formula for a family of test functions on the zero locus of a quadratic form.  The test functions are built out of Whittaker coefficients of an automorphic representation of $\GL_n(\A_F)$.  We first put our work in context and then state our results precisely.

\subsection{Generalized Poisson summation formulae for spherical varieties}

Conjectures of Braverman-Kazhdan \cite{BK}, Lafforgue \cite{Laf}, Ng\^{o} \cite{Ngo, Ngo:Hankel}, and Sakellaridis \cite{Sak} suggest that every affine spherical variety admits a generalized Poisson summation formula. We refer to this conjecture as the Poisson summation conjecture. The Poisson summation conjecture implies the functional equation and meromorphic continuation for fairly general Langlands $L$-functions which, by the converse theorem, implies Langlands functoriality in great generality. 

In \cite{GL}, such a summation formula is proved where the underlying scheme is built out of a triple of quadratic spaces.  This setting is of particular interest because it is the first case in which the Poisson summation conjecture is known where the underlying affine spherical variety is not a torus bundle over a flag variety.  
 Their method of proof involves replacing the cuspidal representation of $\SL_2^3(\A_F)$ appearing in  Garrett's integral representation of the Rankin triple product $L$-function \cite{Gar2, PSR} with a $\theta$-function.

This suggests that new summation formulae can be obtained by replacing cusp forms on symplectic groups appearing in known integral representations with restrictions of $\theta$-functions on metaplectic groups.  
In this paper, we take another step towards this general program. In more detail, we use the exceptional isogeny $\mathrm{SL}_2 \times \mathrm{SL}_2 \to \mathrm{SO}_4$ to substitute two $\theta$-functions into the Rankin-Selberg integral for $\mathrm{SO}_{2\ell} \times \GL_n$ constructed in \cite{Kap} in the special case $\ell=2$. In the next subsection, we state our formula precisely and then give a representation-theoretic interpretation. 
 
\begin{remark}
More generally, one may consider substituting restrictions of minimal representations (in the sense of \cite{GanSavin}) into integral representations of $L$-functions. 
\end{remark}
 
The Rankin-Selberg integral in \cite{Kap} represents a Langlands-Shahidi $L$-function, and it is illuminating to consider our procedure from the point of view of the Langlands-Shahidi method.  It is well-known that Langlands-Shahidi $L$-functions can be roughly enumerated by root systems together with a simple root.  The Dynkin diagram that remains after deleting the simple root is the Dynkin diagram of a Levi subgroup. Our construction corresponds to the Dynkin diagram $D_{n+2}$ with the unique simple root such that the complement of the root is the Dynkin diagram for $\SL_2 \times \SL_2 \times \SL_n$. 

The summation formula we prove in this paper enlarges the collection of cases in which we know the Poisson summation conjecture. At present the set of cases is very small (\cite{BK:normalized, GL, GH}). This provides crucial test cases to examine for insight into the general picture. Moreover, since the ultimate goal is to study higher-rank automorphic $L$-functions that are currently not understood, incorporating a cusp form of arbitrarily high rank from the outset is a step in the right direction. We also point out that there are methods of building up new summation formulae from old ones modeled on the manner that the summation formula for the basic affine space of \cite{BK:basic:affine} is built up out of a family of Poisson summation formulae for a two-dimensional symplectic vector space. Thus it is of interest to have various families of summation formulae to serve as building blocks.

\subsection{Main Theorem}

Let us make the summation formula precise.
Let $F$ be a number field and $\A_F$ be its Adele ring. Let $d_1,d_2$ be two even positive integers, and $V_1=\mathbb{G}_{a}^{d_1},V_2=\mathbb{G}_{a}^{d_2}$ be a pair of affine spaces over $F$ equipped with non-degenerate quadratic forms $\mathcal{Q}$ and $\mathcal{Q'}$ respectively.  Let $V:=V_1 \oplus V_2$.

Let $Y\subset V$ be the closed subscheme whose points in an $F$-algebra $R$ are 
\begin{align}\label{Y}
Y(R):= \{ y=(y_1,y_2)\in V(R): \mathcal{Q}(y_1)=2\mathcal{Q'}(y_2) \}.
\end{align}
Below we will use $R$ to denote a ``test'' $F$-algebra, sometimes without further comment.

Let 
\begin{align*}
V':= \{ \gamma=(\gamma_1,\gamma_2)\in V(R): \gamma_i\neq 0 \}.
\end{align*}
We let $\mathbb{P}Y' \subset \mathbb{P}V$ be the corresponding quasi-projective scheme.  This is the scheme attached to the pair of quadratic spaces mentioned above. 

Our summation formula will involve functions on $Y$ twisted by Whittaker functions attached to a higher-rank cuspidal automorphic representation.  Let $n$ be a positive integer. Let $\tau$ be a cuspidal automorphic representation of $\GL_n(\A_F)$.
Let $H$ be the split orthogonal group $\mathrm{SO}_{2n+1}$. Let
$$
G(R):=\{g=(g_1,g_2) \in \GL_2^2(R): \det g_1=\det g_2^{-1}\},
$$
and let $\xi_s$ be a smooth holomorphic section from the space
\[
\mathrm{Ind}_{Q_n(\A)}^{H(\A)} (\tau \otimes |\det|^{s-\frac{1}{2}}).
\]
Here $Q_n \leq H$ is a parabolic subgroup with Levi $\GL_n$. In \cref{3}, following \cite{Sou,Kap}, we construct a family of inductions of Whittaker functions (lying in $\mathrm{Ind}_{Q_n}^{H}(\mathcal{W}_{\tau,s})$, where $\mathcal{W}_{\tau,s}=\mathcal{W}_{\tau}\otimes |\mathrm{det}|^{s-\frac{1}{2}}$ and $\mathcal{W}_\tau$ is the Whittaker model for $\tau$)
\begin{align*}
H(\A_F) \times \CC &\lto \CC\\
(h,s) &\longmapsto W_{\xi_s}(h,1).
\end{align*}
Here $W$ is an indication that this is a Whittaker function on $\GL_n(\A_F)$ for $\tau$ when restricted to an appropriate Levi subgroup of $H$. 

\begin{remark}
The spaces $V, V_1, V_2$ have no relationship with the groups $G, H$, and the split $2n+1$-dimensional space $H$ is acting on.  We merely require that $n\geq \ell=2$. 
\end{remark}

We extend the Weil representation of $\SL_2^2(\A_F)$ on $\mathcal{S}(V(\A_F))=\mathcal{S}(V_1(\A_F) \times V_2(\A_F))$ to a representation $\rho$ of $G(\A_F)$ on $\mathcal{S}(V_1(\A_F) \times V_2(\A_F) \times \A_F^\times)$ in \cref{3} via a standard procedure.  

We then define for $y \in Y'(\A_F)$ the global integral 
\begin{align} \label{Ifxis}
I(f,W_{\xi_s})(y)=\int_{U_{2}(\mathbb{A}_F)\backslash G(\mathbb{A}_F)} \rho(g)f(y,1)\int_{N^{\circ}(\mathbb{A}_F)} W_{\xi_s}(w u\iota(g),a_y)\psi_{1}(u)dudg.
\end{align}
Here\begin{align} \label{U_2}
U_{2} \subset G
\end{align}
is a maximal unipotent subgroup, $N^{\circ}$ is a certain unipotent subgroup of $H$ (see (\ref{Ncirc})), and $\iota$ is a embedding map from $G$ to $H$ (see (\ref{iota})). Also, 
\begin{align} \label{a_y}
a_y = \begin{psmatrix}
-4Q'(y_2) & \\ & I_{n-1}
\end{psmatrix}\in \GL_n(\A_F)
\end{align} 
encodes the value of the quadratic form.  We point out that the integral $I(f,W_{\xi})(y)$ is only well-defined for $y \in Y(\A_F),$ not for $y \in V(\A_F)$, due to the invariance properties of the Weil representation.  

Thus we have a map 
\begin{align} \label{forimage}
\mathcal{S}(V(\A_F)\times \A^\times) \otimes \mathrm{Ind}_{Q_n(\A_F)}^{H(\A_F)}(\mathcal{W}_{\tau,s}) \lto C^\infty(Y'(\A_F)).
\end{align}
We define $\mathcal{S}(Y(\A_F),\tau,s)$ to be the image of \eqref{forimage}.  We view this as a Schwartz space of functions on $Y'(\A_F)$ twisted by $\tau \otimes |\det|^{s-\frac{1}{2}}.$  
One has a restricted direct product decomposition
\begin{align}
\mathcal{S}(V(\A_F)) \otimes \mathrm{Ind}_{Q_n(\A_F)}^{H(\A_F)}(\mathcal{W}_{\tau,s})=\otimes_v' \mathcal{S}(V(F_v)) \otimes \mathrm{Ind}_{Q_n(F_v)}^{H(F_v)}(\mathcal{W}_{\tau_v,s}).
\end{align}
Here the restricted direct product is with respect to the basic vectors $\one_{V(\OO)\times \OO^{\times}} \otimes W^{\circ}_{\rho_{\tau,s}}$, where $\OO$ denotes the ring of integers of $F$, and $W^{\circ}_{\rho_{\tau,s}}$ is the unique normalized spherical vector in $\mathrm{Ind}_{Q_n(F_v)}^{H(F_v)}(\mathcal{W}_{\tau_v,s})$ at unramified places. Thus one has a restricted direct product decomposition
$$
\mathcal{S}(Y'(\A_F),\tau,s)=\otimes_v' \mathcal{S}(Y'(F_v),\tau_v,s)
$$
with respect to the vectors $I(\one_{V(\OO)\times \OO^{\times}},W^{\circ}_{\rho_{\tau,s}})$ at unramified places. 

We refer to $I(\one_{V(\OO)\times \OO^{\times}},W^{\circ}_{\rho_{\tau,s}})$ as the basic function.  We give an expression for it in \cref{theorem2} below.

\begin{remark}
The function $I(\one_{V(\OO)\times \OO^{\times}},W^{\circ}_{\rho_{\tau,s}})$ is not the basic function of a spherical variety in the usual sense.  We explain its expected relation to the usual basic function in \eqref{basic:rel} below.
\end{remark}

The space $\mathcal{S}(Y'(\A_F),\tau,s)$ comes equipped with a correspondence on the last horizontal line of the following diagram:
\begin{center}
\begin{tikzcd}
 \mathcal{S}(V(\A_F)) \otimes \mathrm{Ind}_{Q_n}^{H}(\mathcal{W}_{\tau,s}) \arrow[r, "{M(\tau,s)}"] \arrow[d, "I"] & \mathcal{S}(V(\A_F)) \otimes \mathrm{Ind}_{Q_n}^{H}(\mathcal{W}_{\tau^{\vee},1-s}) \arrow[d, "I"] \\
\mathcal{S}(Y'(\A_F),\tau,s) \arrow[r, dotted] & \mathcal{S}(Y'(\A_F),\tau^\vee,1-s)
\end{tikzcd}.
\end{center}
Here $M(\tau,s)$ is the usual intertwining operator from $\mathrm{Ind}_{Q_n}^{H}(\mathcal{W}_{\tau, s})$ to $\mathrm{Ind}_{Q_n}^{H}(\mathcal{W}_{\tau^{\vee},1-s})$.
We do not give a definition of the dotted line in this work; it is nontrivial to prove that it exists.  We expect one can prove it using an analogue of the arguments of  \cite{GH}.

Let
\begin{align}
   V''(R)=\{(\gamma_1,\gamma_2)\in V(R): \mathcal{Q}(\gamma_1)=\mathcal{Q'}(\gamma_2)=0\}.
\end{align}
Our summation formula follows:

\begin{theorem}\label{main thm}
For $g\in \mathrm{O}(V_1)(\A_F) \times \mathrm{O}(V_2)(\A_F)$, the sum $\sum_{y\in \mathbb{P}Y'(F)}I(f,W_{\xi_s})(gy)$ admits a meromorphic continuation to the whole $s$-plane. It satisfies a functional equation
\[
\sum_{y\in \mathbb{P}Y'(F)} I(f,W_{\xi_s})(gy) = \sum_{y\in \mathbb{P}Y'(F)} I(f,M(\tau,s)W_{\xi_s})(gy).
 \]
Here $f(y,1)\in \S(V(\A_F)\times \A_F^{\times})$ is a Schwartz-Bruhat function such that $\rho(g)f(\gamma,u)=0$ for all $(g,\gamma,u) \in G(F) \times  V''(F) \times F^{\times}.$ 
\end{theorem}

\begin{remark}
If one directly integrates the sum $\sum_{y\in \mathbb{P}Y'(F)}I(f,W_{\xi_s})(gy)$ against a pair of cusp forms on $ \mathrm{O}(V_1)(\A_F) \times \mathrm{O}(V_2)(\A_F)$ the integral will vanish if the Witt index of one of the $V_i$ is sufficiently large by \cite{AGS}.  Thus one cannot directly apply Theorem \ref{main thm} to the study of $L$-functions on orthogonal groups whose underlying quadratic space has large Witt index.  
\end{remark}

Theorem \ref{main thm} has a different form than predicted by the papers mentioned at the beginning of the paper. It seems reasonable to expect that there exists a spherical variety $Z$ for $ \mathrm{O}(V_1) \times \mathrm{O}(V_2) \times \GL_n$ equipped with a Schwartz space $\mathcal{S}(Z(\A_F))$ with Fourier transform $\mathcal{F}$ and Poisson summation formula 
\[ 
\sum_{z \in Z(F)} f(z)= \sum_{z \in Z(F)} \mathcal{F}(f)(z),
\]
such that (roughly)
$$
\int_{[\GL_n]}\xi_s(1,g')\sum_{z \in Z(F)} f(z(g,g')) dg'=\sum_{y \in \mathbb{P}Y'(F)} I(f,W_{\xi_s})(gy)
$$
and
$$
\int_{[\GL_n]}\xi_s(1,g')\sum_{z \in Z(F)} \mathcal{F}(f)(z(g,g')) dg'=\sum_{y \in \mathbb{P}Y'(F)} I(f,M(\tau,s)W_{\xi_s})(gy)
$$
for $g\in \mathrm{O}(V_1)(\A_F) \times \mathrm{O}(V_2)(\A_F)$ and $g'\in \GL_n(\A_F)$.
After unfolding, at the unramified places, we expect that (again roughly)
\begin{align} \label{basic:rel}
\int_{\GL_n(F)} W_{\xi_s}(1,g') f_{Z}(z(g,g'))dg' = I(\one_{V(\OO)\times \OO^{\times}},W^{\circ}_{\rho_{\tau,s}})
\end{align}
Here $f_{Z} \in \mathcal{S}(Z(F))$ is the expected basic function in the sense of \cite{YW}.  The expectations above are rough because the functions $I(f,W_{\xi_s})$ may need to be renormalized.

The integral $I(f,W_{\xi_s})(y)$ mixes the arithmetic of the quadratic forms $\mathcal{Q}$ and $\mathcal{Q'}$ and the cuspidal automorphic representation $\tau$. It is Eulerian for each $y$ (see the discussion around \eqref{Eulerian}). Ideally, one would like an expression for the unramified local factors in terms of a suitable local model for $\tau$ and the point $y.$  We achieve this in \cref{theorem2} below. This is far more difficult than the corresponding calculation in \cite{GL}. To execute it, we adapt an argument appearing in \cite{Kap}, which ultimately relates the integral to the Bessel model of $\mathrm{Ind}_{Q_n}^{H}(\mathcal{W}_{\tau, s})$ attached to a character on a unipotent subgroup of $H$ and a character on $\SO_2$.  

\begin{theorem}\label{theorem2}
For all the data unramified, $\Re(s)$ large, and $d_2 > d_1$, we have
\begin{align*}
I(\one_{V(\OO)\times \OO^{\times}},W^{\circ}_{\rho_{\tau,s}})(y)  &= \alpha(y_1, y_2)|4Q'(y_2)|^{-\Re(s)+\frac{1}{2}-\frac{n}{2}} \sum_{k=\mathrm{val}(4\mathcal{Q'}(y_2))}^{\infty}  q^{(n-2+\frac{d_2}{2})k} C_{k,s}(y).
\end{align*}
Here $\alpha(y_1, y_2)$ is as defined in \cref{c_k}, and
\begin{align*}
C_{k,s}(y)&=\int_{iI_F +\sigma_1}\int_{iI_F+\sigma_2} \frac{(1-q^{s_2}+(q-1)q^{-\mathrm{val}(y_2)s_2})\zeta_v(-s_2-\frac{d_1}{2}+2)^2\zeta_v(-s_2)^2(\log q)^2}{4\pi^2\gamma(s-s_1+s_2, \chi'\otimes \tau)} \\
& \times 
q^{(-s_1+s_2)k}B_{W^{\circ}_{\rho_{\tau,s}}}\begin{psmatrix}
-4\mathcal{Q}'(y_2)\varpi^k&& \\ &I_{2n-1}& \\ &&(-4\mathcal{Q}'(y_2)\varpi^k)^{-1}
\end{psmatrix}ds_1 ds_2,
\end{align*}
which is the product of the $-k$-th coefficient in $q^{s_1}$ and the $k$-th coefficient in  $q^{s_2}$ of a product of Laurent series in $q^{s_1}$ and $q^{s_2}$, where $\gamma$ represents gamma factor for $\GL_1 \times \GL_n$, and $B_{W^{\circ}_{\rho_{\tau,s}}}$ is the normalized unramified Bessel function defined in \ref{Bessel}.
\end{theorem}

\begin{remarks} 
    \item Let $b_Z$ be the basic function at unramified places in the conjectural Schwartz space $\mathcal{S}(Z(\A_F))$. Then the quantity computed in \cref{theorem2} should correspond to
    \[
    \int_{\GL_n(F)} b_Z(z,g')\xi_s(1,g')dg.
    \]
    \item As mentioned before, the key difference between the results in this paper with the work of Kaplan and Soudry is that we substitute the theta functions on $G(\A)$ for cusp forms in their work. In the unramified setting, we must relate $\rho(g)f(y,1)$ to functions lying in an appropriate Whittaker model in order to apply Kaplan and Soudry's methods. \\
\end{remarks}
    
Let us indicate how our constructions are related to the $\theta$-correspondence and Rankin-Selberg $L$-functions.  
Let $\pi_i$ be a cuspidal automorphic representation of $O(V_i)(F) \backslash O(V_i)(\A_F)$ for $i=1,2.$  Assume that $\pi_i$ is the $\theta$-lift of a cuspidal automorphic representation $\sigma_i$ of $\SL_2(\A_F).$  Assume moreover that the central character of $\sigma_1 \otimes \sigma_2$ is trivial when restricted to the diagonal copy of $\pm I_2.$  Then using the isomorphism
$$
\pm (I_2, I_2) \backslash \SL_2 \times \SL_2 \lto \mathrm{SO}_4
$$
the representation $\sigma_1 \otimes \sigma_2$ defines a cuspidal automorphic representation $\sigma$ of $\mathrm{SO}_4.$
Let 
$$
r:{}^{L}(\mathrm{SO}_4 \times \mathrm{GL}_n) \lto \GG_a^{4n}
$$
be the tensor product of the two standard representations.  
Finally, $\phi_i$ is a cusp form in the space of $\pi_i.$

    Then it follows from \cite{Kap} that the integral 
    \begin{align*}
    &\int_{O(V)(F) \backslash O(V)(\A_F)}\int_{U_2(\A_F) \backslash G(\A_F)}  \sum_{y\in \mathbb{P}Y'(F)} \rho(g,h) f(h^{-1}y,1)\phi_1(h_1)\phi_2(h_2)\\
    & \times \int_{N^{\circ}(\mathbb{A}_F)} W_{\xi_s}(w u\iota(g),a_y)\psi_{1}(u)dudgdh
    \end{align*}
   is Eulerian, with unramified local factors equal to 
    $$
    \frac{L(s,\sigma \times \tau,r)}{L(2s,\pi,\mathrm{Sym}^2)}.
    $$

\subsection{Representation theoretic interpretation} \label{sec:rep}

We provide an interpretation of our summation formula from a representation-theoretic perspective.

In (\ref{Ncirc}) and (\ref{psibeta}) we define a unipotent subgroup $N^{\circ}  \subset \SO_{2n+1}$ and a character $\psi_1: N^{\circ}(F) \backslash N^{\circ}(\A_F) \to \CC$. In particular, $N^{\circ}(\A_F)$ acts on 
$$
\mathrm{Ind}_{Q_n(\A_F)}^{H(\A_F)}(\mathcal{W}_{\tau,s})
$$
via $\psi_1$ (where $\mathcal{W}$ denotes the Whittaker model) and we can consider the coinvariants 
$$
\mathrm{Ind}_{Q_n(\A_F)}^{H(\A_F)}(\mathcal{W}_{\tau,s})_{N^{\circ}(\A_F),\psi_{1}}.
$$
There is an embedding $\iota:G \to \mathrm{SO}_{2n+1}$.  The image normalizes $N^\circ$ and stabilizes $\psi_1$, and we can consider the coinvariants
$$
(\mathrm{Ind}_{Q_n(\A_F)}^{H(\A_F)}(\mathcal{W}_{\tau,s})_{N^{\circ}(\A_F),\psi_{1}} \otimes \mathcal{S}(V(\A_F)))_{G(\A_F)}.
$$
The integral $I(f,\xi_s,s)$ may be viewed as a functional
\begin{align}
(\mathrm{Ind}_{Q_n(\A_F)}^{H(\A_F)}(\mathcal{W}_{\tau,s})_{N^{\circ}(\A_F),\psi_{1}} \otimes \mathcal{S}(V(\A_F)))_{G(\A_F)} \lto C^\infty(Y'(\A_F)).
\end{align}

The functional equation in the main theorem (\cref{main thm}) ultimately is a consequence of the existence of the intertwining operator on $\mathrm{Ind}_{Q_n}^{H}(\mathcal{W}_{\tau,s})$ associated with the longest Weyl element and the functional equation of the corresponding Eisenstein series defined in \cref{Eis} with respect to the intertwining operator.

\subsection{Outline of the paper}
We set up the notation for the various algebraic groups in \cref{2}. In \cref{3}, we establish our summation formula assuming various quantities converge. The main theorem is made rigorous by showing the absolute convergence of the sum of the global integrals in \cref{5}, \cref{7}, and \cref{8}. We give a list of symbols at the end of the paper. 

In \cref{4}, we give the computation of the local integral when all the data are unramified. We justify in \cref{6} the absolute convergence of various integrals and the final result in \cref{4}.

\section*{Acknowledgements}
I would like to thank my advisor Jayce Getz for suggesting this problem and for providing relentless support and valuable advice. I am also grateful to Eyal Kaplan for answering several questions related to his thesis and pointing out several typos in an earlier version of this paper, and to Spencer Leslie for helpful discussions and comments. The interpretation of \cref{sec:rep} is my understanding of comments of Yiannis Sakellaridis, which are greatly appreciated. I also want to thank Orsola Capovilla-Searle, Huajie Li, Stephen Mckean, Aaron Pollack, and Jiandi Zou for helpful comments. This work was partially supported by Jayce Getz's NSF grant DMS-1901883. Finally, I want to thank the anonymous referee for a thorough reading and helpful comments.

\section{Preliminaries} \label{2}

\subsection{Groups}\label{2.1}
For this section we let $F$ be a field of characteristic zero. Let $\OO$ be the ring of integers of $F$ and $\varpi$ be the uniformizer of $\OO$. To define points of $F$-schemes we let $R$ denote an $F$-algebra. All algebraic groups we define below are affine algebraic groups over $F$.

Let 
\[
J_{k}=\begin{pmatrix}
0 & & 1 \\ &\reflectbox{$\ddots$}& \\ 1&&0
\end{pmatrix} \in \GL_k(F)
\]
for $k$ a positive integer. Let $\SO_k$ be the special orthogonal group with respect to $J_k$.

We say that a parabolic subgroup of $\SO_{k}$ is standard if it contains the Borel subgroup of upper triangular matrices.  
Let
\begin{align}\label{G,H}
G(R):=\{(g_1, g_2) \in \GL_{2}^2(R): \mathrm{det}g_1=\mathrm{det}g_2^{-1}\} \quad \textrm{ and } \quad H:=\SO_{2n+1}.
\end{align}

Let
\begin{align} \label{G'}
G':=\SO_4.
\end{align}

We denote by $T_{G'}$ and $T_H$ the corresponding maximal split tori consisting of diagonal matrices. 

\subsubsection{Subgroups of $H$}\label{2.1.2}

Let $Q_{H}$ be the standard parabolic subgroup with Levi subgroup whose points in an $F$-algebra $R$ are
\[
M_H(R):=\left\{(x,c)^{\wedge}=\begin{psmatrix}  x & & \\ & c & \\ & & x^*\end{psmatrix} \in H: (x,c) \in \GL_{n-2}(R) \times \SO_{5}(R), x^{*}=J_{n-2}(^tx^{-1})J_{n-2}\right\}. 
\]

Let $N_H$ be the unipotent subgroup whose points in an $F$-algebra $R$ are
\[
N_{H}(R):=\left\{\begin{pmatrix} z & x & y \\ & I_5 & x' \\ & & z^{*}\end{pmatrix}: x\in M_{(n-2)\times 5}(R), y\in M_{n-2}(R), z\in Z_H(R) \right\}.
\]
Here $z^{*}=J_{n-2}(^tz^{-1})J_{n-2}$, $x'= -J_{5}(\,^tx)J_{n-2}z^{\ast}$, and $Z_{H}$ is the unipotent radical of the Borel subgroup of upper triangular matrices of $\GL_{n-2}$.

We let $Y_{H}$ be the subgroup of $N_{H}$ whose points in an $F$-algebra $R$ are 
\[
Y_{H}(R)= \left\{\begin{pmatrix} z&0&0&x&0 \\ &I_2&0&0&x' \\ &&1&0&0 \\ &&&I_2&0 \\ &&&&z^{*}\end{pmatrix}: z \in Z_{H}(R), z^{*}=J_{n-2}(^tz^{-1})J_{n-2}, x'= -J_{2}(\,^tx)J_{n-2}z^{\ast}\right\}, 
\]
and we denote $N^{\circ}$ the subgroup of $N_{H}$ whose points in an $F$-algebra $R$ are 
\begin{align} \label{Ncirc}
N^{\circ}(R)= \left\{\begin{pmatrix} I_{n-2}&x&y&0&z \\ &I_2&0&0&0 \\ &&1&0&y' \\ &&&I_2&x' \\ &&&&I_{n-2}\end{pmatrix}: x'=-J_2(^tx)J_{n-2}, y'=-J_1(^ty)J_{n-2} \right\}
\end{align}
such that $N^{\circ}$ is isomorphic to $Y_H \backslash N_H$.

For $x \in \GL_n(R)$
 let
 \begin{align} \label{vx}
 v(x):= \begin{psmatrix}  x & & \\ & 1 & \\ & & J_{n}(^tx^{-1})J_{n}\end{psmatrix} \in \GL_{2n+1}(R).
 \end{align}

Let $Q_n$ be the standard parabolic subgroup with Levi subgroup $M_n$ whose points in an $F$-algebra $R$ are 
\begin{align} \label{M_n}
M_n(R):=\left\{v(x): x\in \GL_{n}(R) \right\}.
\end{align}

Let $N_n$ be the unipotent subgroup whose points in an $F$-algebra $R$ are 
\begin{align} \label{N_n}
N_{n}(R):=\left\{\begin{pmatrix} z & x & y \\ & 1 & x' \\ & & z^{*}\end{pmatrix}: z\in Z_n, z^{*}=J_{n}(^tz^{-1})J_{n}, x'=-J_{1}(\,^tx)J_nz^{\ast} \right\}, 
\end{align}
where $Z_n$ is the unipotent radical of the Borel subgroup of upper triangular matrices of $\GL_n$. 

Accordingly, we denote $\overline{Q}_n \subset H$ as the opposite parabolic subgroup with Levi subgroup $M_n$, and we let $\overline{N}_n$ be the corresponding unipotent radical of $\overline{Q}_n$.  
 
Let 
\begin{align} \label{omega}
w  = \begin{pmatrix} &\frac{1}{2} I_2&&& \\ &&&&I_{n-2} \\ &&(-1)^{n-2}&& \\ I_{n-2}&&&& \\ &&&2I_2
& \end{pmatrix} 
\end{align}
be a Weyl group element in $H$.

Let $Q_{G'}$ be a subgroup of $H$ whose points in an $F$-algebra $R$ are 
\begin{align*}
    &Q_{G'}(R)= \\
    & \left\{\begin{pmatrix} a&b&c&-2b&d \\ &1&0&0&-2b' \\ &&1&0&c' \\ &&&1&b' \\ &&&&a^{-1}\end{pmatrix} : a \in R^\times, c \in R, b'=-ba^{-1}, c'=-ca^{-1}, d=\frac{-4b'^2+c'^2}{2a^{-1}}\right\}.
\end{align*}

\subsection{Embedding of the groups} \label{2.2}

For the construction of the global integral, we use two embeddings of groups. Here we give the explicit maps we use in our integral. 

We have a sporadic isogeny between the algebraic groups $\SL_2\times \SL_2 $ and $G'=\mathrm{SO}_4$.  It induces a surjection $G \to G'$ given on points in an $F$-algebra by
\begin{align} \begin{split}
\iota_1: G(R) &\lto G'(R)\\
\left(\begin{pmatrix} a & b \\ c & d \end{pmatrix},\begin{pmatrix} a' & b'\\ c' & d' \end{pmatrix}\right) &\longmapsto \begin{pmatrix}  
aa' & -ab' & ba' & bb' \\ -ac' & ad' & -bc' & -bd' \\ ca' & -cb' & da' & db' \\ cc' & -cd' & dc'  & dd' \label{map}
\end{pmatrix}. \end{split}
\end{align}

In the construction in \cite[Section 2.1]{Kap}, the embedding of $G'$ in $H$ is given by $\mathrm{diag}(I_{n-2},G'',I_{n-2})$, where we denote $G'' \subset \mathrm{SO}_5$ as the image of the embedding of $G'$ in $\mathrm{SO}_5 \subset H$. The map $\iota_2: G' \to \SO_5$ is 
\begin{align*}
&\begin{pmatrix}
a_1 & a_2 & b_1 & b_2 \\  a_3 & a_4 & b_3 & b_4 \\ c_1 & c_2 & d_1 & d_2 \\ c_3 & c_4 & d_3 & d_4
\end{pmatrix} \longmapsto \\
&\resizebox{\textwidth}{!}{$
\begin{pmatrix}
a_1 & \frac{1}{4}a_2-\frac{1}{2}b_1 & \frac{1}{2}a_2+ b_1 & -\frac{1}{2}a_2+b_1 & b_2 \\ a_3-\frac{1}{2}c_1 & \frac{1}{4}a_4-\frac{1}{2}b_3+\frac{1}{2}-\frac{1}{8}c_2+\frac{1}{4}d_1 &  \frac{1}{2}a_4- b_3-\frac{1}{4}c_2-\frac{1}{2}d_1 & -\frac{1}{2}a_4- b_3+1\frac{1}{4}c_2-\frac{1}{2}d_1 & b_4-\frac{1}{2}d_2 \\
a_3+\frac{1}{2}c_1 & \frac{1}{4} a_4-\frac{1}{2}b_3+\frac{1}{8}c_2-\frac{1}{4} d_1 & \frac{1}{2}a_4+b_3+\frac{1}{4}c_2+\frac{1}{2}d_1 & -\frac{1}{2}a_4- b_3-\frac{1}{4}c_2+\frac{1}{2}d_1 & b_4+\frac{1}{2}d_2 \\ -\frac{1}{2}a_3+\frac{1}{4}c_1 &  -\frac{1}{8}a_4-\frac{1}{4} b_3+\frac{1}{4}+\frac{1}{16}c_2-\frac{1}{8}d_1 & -\frac{1}{4} a_4-\frac{1}{2} b_3+\frac{1}{8}c_2+\frac{1}{4} d_1 & \frac{1}{4}a_4-\frac{1}{2} b_3+\frac{1}{2}-\frac{1}{8}c_2+\frac{1}{4}d_1 & -\frac{1}{2}b_4+\frac{1}{4}d_2 \\ c_3 & \frac{1}{4}c_4-\frac{1}{2}ds_3 & \frac{1}{2}c_4+ d_3 & -\frac{1}{2}c_4+d_3 & d_4 
\end{pmatrix}$}.
\end{align*}

We define the composite map
\begin{align}\label{iota}
    \iota:=\iota_2 \circ \iota_1:G \lto \SO_5.
\end{align}

\subsection{Image of the maps} \label{2.3}

Using the map from $G$ to $G'=\SO_4$ and $G'$ in $\SO_5$ (which naturally embeds in $H$), we make the image of subgroups of $G$ in $\SO_5$ precise. 

Let $M_1$ be the subgroup of the  maximal torus of $G$  whose points in an $F$-algebra $R$ are
\begin{align} \label{M_1}
M_1(R):= \left\{\left(\begin{psmatrix}
1 & 0 \\ 0 & m^{-1} 
\end{psmatrix},
\begin{psmatrix}
m & 0 \\ 0 & 1 
\end{psmatrix}
\right): m \in R^{\times}\right\}. 
\end{align}

\begin{lemma} \label{lem1}
Let $M'_1=\iota(M_1)\subset \mathrm{SO}_5$.  Then
\begin{align} \label{M_2'}
M_1'(R)=\left\{\begin{psmatrix} m & & \\ & I_3 & \\ & & m^{-1} \end{psmatrix}: m \in R^{\times}\right\}. 
\end{align}
\end{lemma} 

Let $G_1$ be a subgroup of the  maximal torus of G whose points in an $F$-algebra $R$ are
\begin{align} \label{G_1}
G_1(R) := \left\{\left(\begin{psmatrix}
1 & 0 \\ 0 & b^{-1} 
\end{psmatrix},
\begin{psmatrix}
1 & 0 \\ 0 & b 
\end{psmatrix}
\right): b \in R^{\times}\right\}.
\end{align}

\begin{lemma}\label{lem2}
Let $G_1'=\iota(G_1)<\SO_5$. Then 
\[
G_1'(R)=\left\{\begin{pmatrix} 1&&&& \\ &\frac{1}{2}+\frac{1}{4}(b+b^{-1}) & \frac{1}{2}(b-b^{-1}) & 2(b-b^{-1})& \\ &(b-b^{-1})& \frac{1}{2}(b+b^{-1}) & -\frac{1}{2}(b-b^{-1})& \\ &\frac{1}{2}(\frac{1}{2}-\frac{1}{4}(b+b^{-1})) & -\frac{1}{4}(b-b^{-1}) & \frac{1}{2}+\frac{1}{4}(b+b^{-1})& \\&&&&1 \end{pmatrix}: b\in R^{\times}\right\}.
\]
\end{lemma} 

Let $A_1$ be a subgroup of $T_G$ whose points in $F$-algebra $R$ are 
\begin{align} \label{A_1}
A_1(R) := \left\{\left(\begin{pmatrix}
a_1 & 0 \\ 0 & {a_1}^{-1} 
\end{pmatrix},
\begin{pmatrix}
1 & 0 \\ 0 & 1 
\end{pmatrix}
\right): a_1 \in R^{\times}\right\}.
\end{align}

\begin{lemma}
Let $A_1'=\iota(A_1)$. Then 
\[
A_1'(R)=\left\{\begin{psmatrix} a_1&&&& \\ &\frac{1}{2}+\frac{1}{4}(a_1^2+a_1^{-2}) & \frac{1}{2}(a_1^2-a_1^{-2}) & 2(a_1^2-a_1^{-2})& \\ &(a_1^2-a_1^{-2})& \frac{1}{2}(a_1^2+a_1^{-2}) & -\frac{1}{2}(a_1^2-a_1^{-2})& \\ &\frac{1}{2}(\frac{1}{2}-\frac{1}{4}(a_1^2+a_1^{-2})) & -\frac{1}{4}(a_1^2-a_1^{-2}) & \frac{1}{2}+\frac{1}{4}(a_1^2+a_1^{-2})& \\&&&&a_1^{-1} \end{psmatrix}: a_1\in R^{\times}\right\}.
\]
\end{lemma}

Let $A_2$ be the subgroup of $T_G$ whose points in $F$-algebra $R$ are 
\begin{align} \label{A_2}
A_2(R) := \left\{\left(\begin{pmatrix}
1 & 0 \\ 0 & 1
\end{pmatrix},
\begin{pmatrix}
a_2 & 0 \\ 0 & a_2^{-1} 
\end{pmatrix}
\right): a_2 \in R^{\times}\right\}.
\end{align}

\begin{lemma}
Let $A_2'=\iota(A_2)$. Then 
\[
A_2'(R)=\left\{\begin{psmatrix} a_2&&&& \\ &\frac{1}{2}+\frac{1}{4}(a_2^2+a_2^{-2}) & \frac{1}{2}(-a_2^2+a_2^{-2}) & 2(-a_2^2+a_2^{-2})& \\ &(-a_2^2+a_2^{-2})& \frac{1}{2}(a_2^2+a_2^{-2}) & -\frac{1}{2}(a_2^2-a_2^{-2})& \\ &\frac{1}{2}(\frac{1}{2}-\frac{1}{4}(a_2^2+a_2^{-2})) & -\frac{1}{4}(-a_2^2+a_2^{-2}) & \frac{1}{2}+\frac{1}{4}(a_2^2+a_2^{-2})& \\&&&&a_2^{-1} \end{psmatrix}: a_2\in R^{\times}\right\}.
\]
\end{lemma}

Note that we have $T_G=A_1A_2G_1$. 

Let $U_2$ be the maximal unipotent radical of the Borel subgroup of upper triangular matrices of $G$. Let $N_2$ be a subgroup of the $U_2$ whose points in an $F$-algebra $R$ are
\begin{align} \label{N1}
N_1(R):=\left\{\left(\begin{pmatrix}
1 & \frac{c}{2} \\ 0 & 1
\end{pmatrix},
\begin{pmatrix}
1 & - c \\ 0 & 1
\end{pmatrix}\right): c\in R \right\}.
\end{align}

\begin{lemma} \label{lem3} 
Let $N_1'=\iota(N_1) <\SO_5$. Then
$$
N_1'(R)= \left\{\begin{pmatrix} 1&0&c&0&-\frac{1}{2}c^2 \\ &1&0&0&0 \\ &&1&0&-c \\ &&&1&0 \\ &&&&1\end{pmatrix}: c\in R \right\}.
$$
\end{lemma} 

Let $N_2$ be a subgroup of $U_2$ whose points in an $F$-algebra $R$ are
\begin{align} \label{N2}
N_{2}(R):=\left\{\left(\begin{pmatrix}
1 & b \\ 0 & 1
\end{pmatrix},
\begin{pmatrix}
1 & 2b \\ 0 & 1
\end{pmatrix}\right): b\in R \right\}.
\end{align}

\begin{lemma} \label{lem4}
Let $N_2'= \iota(N_2)<\SO_5$
$$
N_2'(R)=\left\{\begin{pmatrix} 1&b&0&-2b&0 \\ &1&0&0&2b \\ &&1&0&0 \\ &&&1&-b \\ &&&&1\end{pmatrix}: b\in R\right\}.
$$
\end{lemma} 

Let $M_{SL_2^2}$ be a subgroup of $\SL_2 \times \SL_2$ whose points in an $F$-algebra $R$ are
\begin{align} \label{M_SL_2^2}
M_{SL_2^2}(R):=\left\{\left(\begin{pmatrix}
m & 0 \\ 0 & m^{-1}
\end{pmatrix},
\begin{pmatrix}
m & 0 \\ 0 & m^{-1}
\end{pmatrix}\right): m\in R^{\times} \right\}.
\end{align}

\begin{lemma} 
Let $Q_G=M_{SL_2^2}N_1N_2$. Then $\iota(Q_{G})(R)$ is  
$$
\left\{\begin{psmatrix} a^2&b&c&-2b&d \\ &1&0&0&-2b' \\ &&1&0&c' \\ &&&1&b' \\ &&&&a^{-2}\end{psmatrix} : a \in R^\times, c \in R, b'=-ba^{-1}, c'=-ca^{-1}, d=\frac{-4b'^2+c'^2}{2a^{-1}}\right\}. 
$$
\end{lemma}

\subsection{Summary}

We have given a Levi decomposition 
$$
A_1A_2G_1N_1N_2=A_1G_1M_1N_1N_2
$$
of the Borel of upper triangular matrices in $G(R):=\{(g_1, g_2) \in \GL_{2}^2(R): \mathrm{det}g_1=\mathrm{det}g_2^{-1}\}$.  We let $G':=\mathrm{SO}_4$. Moreover, we have a commutative diagram
$$
\begin{tikzcd}
A_1A_2G_1N_1N_2 \arrow[hookrightarrow]{r} \arrow[rd] &G \arrow[d,"\iota_1"] \arrow[dd, bend left=90, "\iota", looseness=1]\\
&G':=\mathrm{SO}_4 \arrow[d,"\iota_2"]\\
&\mathrm{SO}_5
\end{tikzcd}.
$$

Our main theorems are stated without the use of $G'$, but we require it in the proofs.  

\subsection{Notations for local fields} 

Let $F$ be a global field and $v$ a place of $F$.  We denote by $\OO$ the ring of integers of $F$ and $\OO_v$ the ring of integers of $F_v$ for nonarchimedean $v$.  We denote by $\varpi_v$ a uniformizer for $\OO_v$ and $q_v:=|\OO_v/\varpi_v|$ the residual characteristic.  The idelic norm is denoted by $|\cdot|$ and the local norm on $F_v$ (normalized in the usual manner) is denoted by $|\cdot|_v$.  

\subsection{Measures}
We fix a nontrivial character $\psi:F \backslash \A_F \to \CC^\times$ and then choose Haar measures on $F_v$ for all places $v$ that are self-dual with respect to $\psi_i.$  This yields a measure on $\A_F.$
We let 
\[
d^{\times}x_v:= \zeta_v(1)\frac{dx_v}{|x|_v}.
\]
This is a Haar measure on $F_v^\times.$  

For every split reductive group $G$ we fix a maximal compact subgroup $K \leq G(\A_F)$ that is hyperspecial at all finite places and we normalize the Haar measure on $G(F_v)$ so that $K_v$ has measure $1.$  For the $F_v$-points of unipotent subgroups we normalize the Haar measure by transporting measures from $F_v$ to the root subgroups in the usual manner.

\section{The Summation Formula} \label{3}

In this section, we use the Rankin-Selberg integral for $\SO_{2\ell}\times \GL_n$ developed in \cite{Kap} to deduce the expression of our global integral when we take $\ell=2$. We state and prove our main theorem of this paper in \cref{thm} assuming the absolute convergence statement. 

The main theorem will be made rigorous by showing the absolute convergence of the sum of the global integrals in \cref{7}. 

Let $F$ be a number field. We first briefly recall the construction of the Rankin-Selberg integral in \cite[Section 3]{Kap}. Let $\tau$ be an irreducible automorphic representation for $\GL_n(\A_F)$. 

Let $\xi_{s}$ be a smooth holomorphic section from the (normalized induction) space 
\[
\mathrm{Ind}_{Q_n(\A_F)}^{H(\A_F)}(\tau \otimes |\det|^{s-\frac{1}{2}}).
\]
We have the Eisenstein series 
\begin{align}\label{Eis}
E_{\xi_s}(h):=\sum_{y\in Q_n(F) \backslash H(F)} \xi_s(yh,1),
\end{align}
where the first variable of $\xi_s$ is on $H$, and the second variable of $\xi_s$ is on $\GL_n$.

Let $\psi$ be a non-trivial additive character of $F\backslash \A_F$. For $u\in N_{H}(\A_F)$, let
\begin{align} \label{psibeta}
\psi_{1}(u)=\psi\left(\sum_{i=1}^{n-3}u_{i,i+1}+u_{n-2,n}+\frac{1}{2} u_{n-2,n+2}\right)
\end{align}
be a character of $N_{H}(\A_F)$, trivial on $N_H(F)$.

Then the $\psi_1$-coefficient of $E_{\xi_s}$ with respect to $N_{H}(\A_F)$ is
\[
E_{\xi_s}^{\psi_1}(h)=\int_{N_{H}(F)\backslash N_{H}(\A_F)} E_{\xi_s}(uh)\psi_1(u)du.
\]

Let $\varphi$ be a cusp form on $G(\A_F)$. The Rankin-Selberg integral in this case is
\[
I(\varphi,\xi_s)=\int_{G'(F)\backslash G'(\A_F)} \varphi(g) E_{\xi_s}^{\psi_1}(g)dg.
\]

This global integral converges absolutely in the whole $s$-plane except at the poles of the Eisenstein series $E_{\xi_s}(h)$, and the absolute convergence follows from the rapid decay of the cusp form $\varphi$ and the moderate growth of the Eisenstein series $E_{\xi_s}(h)$. 

Let $d_1,d_2$ be two even positive integers, and $V_1=\mathbb{G}_{a}^{d_1},V_2=\mathbb{G}_{a}^{d_2}$ be a pair of affine spaces over $F$ equipped with non-degenerate quadratic forms $\mathcal{Q}$ and $\mathcal{Q'}$ respectively.  Let $V:=V_1 \oplus V_2$. Let 
\begin{align}
Y(R):= \{ y=(y_1,y_2)\in V(R): \mathcal{Q}(y_1)=2\mathcal{Q'}(y_2) \}, 
\end{align}
and let
\begin{align*}
V':= \{ \gamma=(\gamma_1,\gamma_2)\in V(R): \gamma_i\neq 0 \}.
\end{align*}
We let $\mathbb{P}Y' \subset \mathbb{P}V$ be the corresponding quasi-projective scheme.  This is the scheme attached to the pair of quadratic spaces mentioned above. 

For a fixed $u\in F^{\times}$, let $\rho_u$ be the usual Weil representation on $\SL_2^2(\A_F)$ with quadratic forms $u\mathcal{Q}(\gamma)$. Let $f\in \S(V(\A_F)\times \A_F^{\times})$. 

We construct the Weil representation $\rho$ for $G(\A_F)$ following \cite[Section 2.1]{YYZ}, which extends the usual Weil representation of $\SL_2^2(\A_F)$ as follows:
\begin{align*}
&\rho(g)f(\gamma,u)  = \rho_u(g)f(\gamma,u), \quad g\in \SL_2^2(\A_F)\\
&\rho(\begin{psmatrix}
1&\\&a 
\end{psmatrix},\begin{psmatrix}
1&\\&a^{-1}
\end{psmatrix})f(\gamma,u)  = f(\gamma,a^{-1}u)|a|^{-\frac{\mathrm{dim}V_1}{4}}, \quad a\in\A_F^{\times}.
\end{align*} 
Here $\S(V(\A_F)\times \A_F^{\times})$ is the usual Schwartz space for vector spaces.

Let $f \in \S(V(\A_F)\times \A_F^{\times})$  be a Schwartz-Bruhat function such that 
\begin{align} \label{condition}
\rho(g)f(\gamma,u)=0 \text{ for all } g\in G(\A_F) \text{ and } (\gamma,u) \in V''(F) \times F^{\times},
\end{align}
where $V''(R)=\{(\gamma_1,\gamma_2)\in V(R): \mathcal{Q}(\gamma_1)=\mathcal{Q}'(\gamma_2)=0\}$. 

We let
\[
\Theta_{f}(g)=\sum_{(\gamma,u) \in V(F) \times F^{\times}} \rho(g)f(\gamma,u)
\]
be the theta function on $G(\A_F)$. 

Using the formula of $I(\varphi,\xi_s)$, we define the global integral as 
\[
I(\Theta_f,\xi_s) = \int_{G(F)\backslash G(\A_F)} \Theta_{f}(g) E_{\xi_s}^{\psi_1}(g)dg.
\]

Since we take $\rho(g)f(0)=0$ for all $g\in G(\A_F)$, $\theta_{f}(g)$ is cuspidal. Thus $I(\Theta_f, \xi_s)$ converges absolutely in the whole $s$-plane except at the poles of the Eisenstein series $E_{\xi_s}^{\psi_1}(h)$ similar as the Rankin-Selberg integral $I(\varphi,\xi_s)$.

By the action of Weil representation, we have 
\[
I(\Theta_f,\xi_s) = \int_{\SL_2^2(F)\backslash G(\A_F)} \sum_{\gamma \in V(F)} \rho(g)f(\gamma,1) E_{\xi_s}^{\psi_1}(g)dg.
\]

We first unfold the Eisenstein series $E_{\xi_s}$ for $\Re(s)$ large. 

\begin{lemma}
For $\Re(s)$ large, we have 
\begin{align} \label{unfold}
I(\Theta_f,\xi_s) = \int_{Q_{G}(F)\backslash G(\A_F)} \sum_{\gamma \in V(F)} \rho(g)f(\gamma,1) \int_{Y_H(F)\backslash N_H(\mathbb{A})}  \xi_s(w_0 u\iota(g),1)\psi_1(u)dudg.
\end{align}
Here
\[
w_0 = \begin{pmatrix} &&&&I_{n-2} \\ &I_2&&& \\ &&(-1)^{n-2}&& \\ &&&I_2& \\ I_{n-2} &&&&\end{pmatrix} \in H.
\]
\end{lemma}

\begin{proof}
By the embedding map $\iota_1$ from $G$ to $G'$ (see \cref{map}), we have a long exact sequence
\[
1 \to \{\pm I_2\} \to \SL_2^2(F) \xrightarrow{\iota_1} G'(F) \xrightarrow{\mathrm{sn}} H^1(F,\pm I_2)\cong F^{\times}/(F^{\times})^2  \to 1,
\]
where the map sn denotes the spinor norm. 

It follows that 
\begin{align*}
\iota_2(G'(F)) &= \bigcup_{\epsilon \in F^{\times}/(F^{\times})^2}\begin{psmatrix}
I_{n-2}&&&&\\&\epsilon&&&\\&&I_3&&\\&&&\epsilon^{-1}&\\&&&&I_{n-2}
\end{psmatrix}\iota(\SL_2^2)(F),
\\
Q_{G'}(F) &=  \bigcup_{\epsilon \in F^{\times}/(F^{\times})^2}\begin{psmatrix}
I_{n-2}&&&&\\&\epsilon&&&\\&&I_3&&\\&&&\epsilon^{-1}&\\&&&&I_{n-2}
\end{psmatrix} \iota(Q_G)(F).
\end{align*}

Then 
\[
 Q_G(F) \backslash \SL_2^2(F) \cong \iota(Q_G)(F) \backslash \iota(\SL_2^2)(F) \cong Q_{G'}(F) \backslash \iota_2(G'(F)).
\]

By \cite[Proof of Proposition 3.1, Page 151-154]{Kap}, after unfolding the Eisenstein series, the only non-vanishing contribution of $E_{\xi_s}^{\psi_1}$ in $I(\Theta_f, \xi, s)$ for $\Re(s)$ large is
\[
\sum_{y \in Q_{G'}(F) \backslash \iota_2(G'(F))} \int_{Y_H(F)\backslash N_H(\A_F)} \xi_s(w_0yu\iota(g),1)\psi_1(u)du,
\]
which is equivalent to 
\[
\sum_{y \in Q_G(F) \backslash  \SL_2^2(F)} \int_{Y_H(F)\backslash N_H(\A_F)} \xi_s(w_0yu\iota(g),1)\psi_1(u)du.
\]

Then we have
\begin{align*}
I(\Theta_f,\xi_s)
&=\int_{\SL_2^2(F) \backslash G(\A_F)} \sum_{\gamma \in V(F)} \rho(g)f(\gamma,1) \\
&\times \sum_{y\in Q_{G}(F)\backslash \SL_2^2(F)} \int_{Y_H(F)\backslash N_H(\A_F)}  \xi_s(w_0 yu\iota(g),1)\psi_1(u)dudg\\
&=\int_{Q_{G}(F)\backslash G(\A_F)} \sum_{\gamma \in V(F)} \rho(g)f(\gamma,1) \int_{Y_H(F)\backslash N_H(\A_F)}  \xi_s(w_0 u\iota(g),1)\psi_1(u)dudg.
\end{align*}
\end{proof}

Let 
\[
a_y = \begin{psmatrix}
-4Q'(y_2) & \\ & I_{n-1}
\end{psmatrix}\in \GL_n(\A_F).
\]
We define 
\begin{align} \label{Whittaker}
W_{\xi_s}(w u\iota(g),a_y) = \int_{Z_n(F) \backslash Z_n(\A)} \xi_s(w u\iota(g),a_y z)\psi_0^{-1}(z)dz, 
\end{align}
where $w$ is defined in \cref{omega}, f satisfies \cref{condition} ,and $\psi_0$ is the standard character on $Z_n(\A)$.

The main theorem of the paper is:
\begin{theorem}\label{thm}
Let 
\[
I(f,W_{\xi_s})(y)=\int_{U_2(\mathbb{A}_F)\backslash G(\mathbb{A}_F)} \rho(g)f(y,1)\int_{N^{\circ}(\mathbb{A}_F)}  W_{\xi_s}(w u\iota(g),a_y)\psi_1(u)dudg.
\]

For $g\in \mathrm{O}(V_1)(\A_F) \times \mathrm{O}(V_2)(\A_F)$, the sum $\sum_{y\in \mathbb{P}Y'(F)}I(f,W_{\xi_s})(gy)$ admits a meromorphic continuation to the whole $s$-plane which satisfies a functional equation
\[
\sum_{y\in \mathbb{P}Y'(F)} I(f,W_{\xi_s})(gy) = \sum_{y\in \mathbb{P}Y'(F)} I(f,M(\tau,s)W_{\xi_s})(gy).
\]
\end{theorem}

\begin{proof}
We use the defining property of the action of the Weil representation on $f$, and \cref{lem1} through \cref{lem4} to unfold the integral.

Firstly, using \cref{lem1}, \cref{lem3}, \cref{lem4}, and the action of $N_1$ on $f$ we have
\begin{align*}
I(\Theta_f,\xi_s) &=\int_{Q_{G}(F)\backslash G(\A_F)} \sum_{\gamma \in V(F)} \rho(g)f(\gamma,1) \int_{Y_H(F)\backslash N_H(\A_F)}  \xi_s(w_0 u\iota(g),1)\psi_1(u)dudg \\
&= \int_{N_1(\A_F)N_2(F)M_1(F) \backslash G(\A_F)} \int_{N_1(F) \backslash N_1(\A_F)} \sum_{\gamma \in V(F)} \rho(r'g)f(\gamma,1) \\& \times \int_{Y_H(F)\backslash N_H(\A_F)}  \xi_s(w_0 ur'\iota(g),1)\psi_1(u)dudr'dg \\
&=  \int_{N_1(\A)N_2(F)M_1(F) \backslash G(\A_F)} \int_{F\backslash \A_F} \sum_{\gamma \in V(F)} \rho(g)\psi(\frac{c}{2}Q(\gamma_1)- cQ'(\gamma_2))f(\gamma,1) \\ & \times \int_{Y_H(F)\backslash N_H(\A_F)}  \xi_s(w_0 uc\iota(g),1)\psi_1(u)dudcdg \\
&= \int_{N_1(\A_F)N_2(F)M_1(F) \backslash G(\A_F)} \sum_{\substack{ \gamma\in V'(F) \\ Q(\gamma_1)=2Q'(\gamma_2)}} \rho(g)f(\gamma,1) \\
& \times \left(\int_{Y_H(F)\backslash N_H(\A_F)}  \xi_s(w_0 u\iota(g),1)\psi_1(u)du\right)dg.
\end{align*}
Here the last line holds since by \cite[Proof of Propostion~ 3.1]{Kap}, the function
\[
g \mapsto \int_{Y_H(F)\backslash N_H(\A_F)}  \xi_s(w_0 u\iota(g),1)\psi_1(u)du
\]
is invariant on the left for $r'\in N_1(\A_F)$. 

Using the action of $N_2$ on $f$ we have
\begin{align*}
I(\Theta_f,\xi_s) &= \int_{N_1(\A_F)N_2(\A_F)M_1(F) \backslash G(\A_F)} \int_{N_2(F) \backslash N_2(\A_F)} \sum_{ \substack{ \gamma\in V'(F) \\ Q(\gamma_1)=2Q'(\gamma_2)}} \rho(zg)f(\gamma,1) \\ & \times \int_{Y_H(F)\backslash N_H(\A_F)}  \xi_s(w_0 uz\iota(g),1)\psi_1(u)dudzdg\\ 
&= \int_{N_1(\A_F)N_2(\A_F)M_1(F) \backslash G(\A_F)} \int_{N_2(F) \backslash N_2(\A_F)} \sum_{ \substack{ \gamma\in V'(F) \\ Q(\gamma_1)=2Q'(\gamma_2)}} \rho(g) \\ & \times \psi(bQ(\gamma_1)+2bQ'(\gamma_2))f(\gamma,1) \int_{Y_H(F)\backslash N_H(\A_F)}  \xi_s(w_0 uz\iota(g),1)\psi_1(u)dudbdg \\
&=  \int_{N_1(\A)N_2(\A_F)M_1(F) \backslash G(\A_F)}  \sum_{ \substack{ \gamma\in V'(F) \\  Q(\gamma_1)=2Q'(\gamma_2)}} \rho(g)f(\gamma,1)  \\ &\times \int_{N_2(F) \backslash N_2(\A_F)} \int_{Y_H(F)\backslash N_H(\mathbb{A})}  \xi_s(w_0 uz\iota(g),1)\psi(4bQ'(\gamma_2))\psi_1(u)dudzdg. \\
\end{align*}
Using the action of $M_1$ on $f$ we have
\begin{align*}
I(\Theta_f,\xi_s)
&=  \int_{N_1(\A_F)N_2(\A_F) \backslash G(\A_F)}  \sum_{\substack {\gamma \in \mathbb{P}V'(F) \\  Q(\gamma_1)=2Q'(\gamma_2)}} \rho(g)f(\gamma,1) \\ & \times \int_{N_2(F) \backslash N_2(\A_F)} \int_{Y_H(F)\backslash N_H(\A_F)}  \xi_s(w_0 uz\iota(g),1)\psi_1(u)\psi(4bQ'(\gamma_2))dudzdg \\
&=  \int_{N_1(\A_F)N_2(\A_F) \backslash G(\A_F)}  \sum_{y \in \mathbb{P}Y'(F)} \rho(g)f(y,1) \\ & \times \int_{N_2(F) \backslash N_2(\A_F)}\int_{Y_H(\A_F)\backslash N_H(\A_F)} \int_{Y_H(F) \backslash Y_H(\A_F)}  \xi_s(w_0 yuz\iota(g),1) \\
& \times \psi_1(nu)\psi(4b\mathcal{Q'}(y_2))dndudzdg.
\end{align*}

As in \cite[Page 42]{Kap2}, for fixed $u$ and $g$, the function on $N_2(\A_F)$ 
\[
z \mapsto \int_{Y_H(F) \backslash Y_H(\A_F)} \xi_s(w_0 nuzg,1)\psi_1(nu)\psi(4b\mathcal{Q'}(y_2))dn
\]
is well-defined since the elements of $N_2(\A_F)$ and $Y_H(\A_F)$ commute. Also, since $z$ normalizes $Y_H(\A_F)$, stabilizes $\psi_1(y)$ and $\xi_s(w_0z,1)=\xi_s(w_0,1)$, the function is left-invariant on $N_2(F)$. The mapping on $N_H(\A_F)$ 
\[
u \mapsto \int_{N_2(F) \backslash N_2(\A_F)} \int_{Y_H(F) \backslash Y_H(\A_F)} \xi_s(w_0 nuzg,1)\psi_1(yu)\psi(4b\mathcal{Q'}(y_2))dndz
\]
is left-invariant by $Y_H(\A_F)$. Also, $z$ in the integral normalizes $N_H(\A_F)$ and stabilizes $\psi_1$. Thus we can interchange $uz$ to $zu$ in the integral. We have
\begin{align*}
I(\Theta_f,\xi_s)
&= \int_{N_2(\A_F)N_2'(\A_F) \backslash G(\A_F)}  \sum_{y \in \mathbb{P}Y'(F)} \rho(g)f(y,1) \\ & \times \int_{Y_H(\A_F)\backslash N_H(\A_F)} \int_{N_2'(F) \backslash N_2'(\A_F)} \int_{Y_H(F) \backslash Y_H(\A_F)}  \xi_s((w_0 nzw_0^{-1})w_0u\iota(g),1)\\
&\times \psi_1(yu)\psi(4b\mathcal{Q'}(y_2))dndudzdg.
\end{align*}

Let 
\[
w' = \begin{pmatrix} 0&I_2 \\ I_{n-2}&0 \end{pmatrix}.
\]

As in \cite[Page 43]{Kap2}, the double integral $\int_{N_2'(F) \backslash N_2'(\A_F)} \int_{Y_H(F) \backslash Y_H(\A_F)}$ can be written as $\int_{\tilde{Z}_n(F) \backslash \tilde{Z}_n(\A_F)}$, where $w' \tilde{Z}_n w'^{-1}=Z_n$.

We note that now the character on the group $Z_n$ is 
\[
\psi_1'(z) = \psi(-4\mathcal{Q'}(y_2)z_{1,2}+\frac{1}{2}z_{2,3}+\sum_{i=3}^{n-1}z_{i,i+1})
\]
for $z\in Z_n(\A_F)$.

We use a conjugation by
\[
\begin{pmatrix} \frac{1}{2} I_2&0 \\ 0&I_{n-2} \end{pmatrix}
\]
to replace the character $\psi_1'$ to a character $\psi_{y, \mathcal{Q'}}$, where $\psi_{y, \mathcal{Q'}}$ is the generic character of $Z_n$ such that 
\begin{align} \label{psiQ}
\psi_{y, \mathcal{Q'}}(z) = \psi(-4\mathcal{Q'}(y_2)z_{1,2}+z_{2,3}+\cdots+z_{n-1,n}). 
\end{align}

Then we get the sum of integrals
\begin{align*}
& I(\Theta_f,\xi_s) \\
&= \int_{N_2(\A_F)N_2'(\A_F) \backslash G(\A_F)}  \sum_{y \in \mathbb{P}Y'(F)} \rho(g)f(y,1) \int_{Y_H(\A_F)\backslash N_H(\A_F)} W_{\xi_s}^{\psi_{y, \mathcal{Q'}}}(w u\iota(g),1)\psi_1(u)dudg \\
&= \sum_{y \in \mathbb{P}Y'(F)} \int_{U_{2}(\A_F) \backslash G(\A_F)}\rho(g)f(y,1) \int_{Y_H(\A_F)\backslash N_H(\A_F)} W_{\xi_s}^{\psi_{y, \mathcal{Q'}}}(w u\iota(g),1)\psi_1(u)dudg,
\end{align*}
where 
\begin{align*}
W_{\xi_s}^{\psi_{y, \mathcal{Q'}}}(w u\iota(g),1) = \int_{Z_n(F) \backslash Z_n(\A)} \xi_s(w u\iota(g),z)\psi_{y, \mathcal{Q'}}^{-1}(z)dz. 
\end{align*}
We recall that $W_{\xi_s}^{\psi_{y, \mathcal{Q'}}}$ may be factored into a product of local Whittaker functions if $\xi_s$ is a pure tensor (due to the uniqueness of the Whittaker model \cite{GK}), and hence the integral
\begin{align} \label{Eulerian}
\int_{U_{2}(\A_F) \backslash G(\A_F)}\rho(g)f(y,1) \int_{Y_H(\A_F)\backslash N_H(\A_F)} W_{\xi_s}^{\psi_{y, \mathcal{Q'}}}(w u\iota(g),1)\psi_1(u)dudg
\end{align}
is Eulerian. Thus $I(\Theta_f,\xi_s)$ is a sum of Eulerian integrals. To finish the proof of the theorem we need to use the lemma below.

\begin{lemma}
We have 
\begin{align*}
W_{\xi_s}^{\psi_{y,Q'}}(w u\iota(g),1) = W_{\xi_s}(w u\iota(g), a_y).
\end{align*}
\end{lemma}

\begin{proof}
We have
\[
\psi_0(a_yza_y^{-1}) = \psi_{y,Q'}(z)
\]
for $z\in Z_n$. 

Thus we have the function
\[
g'=w u\iota(g) \mapsto W_{\xi_s}(g', a_y) 
\]
lies in $\mathrm{Ind}_{Q_n}^{H}(\mathcal{W}_{\tau, s,\psi_{y,Q'}})$ since
\[
W_{\xi_s}(g', a_yz') = W_{\xi_s}(g', (a_yz'a_y^{-1})a_y) = \psi_0(a_yz'a_y^{-1})W_{\xi_s}(g', a_y)=  \psi_{y,Q'}(z')W_{\xi_s}(g', a_y)
\]
for $z'\in Z_n(\A)$.
\end{proof}

Thus we have
\begin{align*}
I(\Theta_f,\xi_s) = \sum_{y \in \mathbb{P}Y'(F)} \int_{U_{2}(\A_F) \backslash G(\A_F)}\rho(g)f(y,1) \int_{Y_H(\A_F)\backslash N_H(\A_F)} W_{\xi_s}(w u\iota(g),a_y)\psi_1(u)dudg. 
\end{align*}

The manipulations of the integral will be justified in \cref{8} by showing the sum converges absolutely for $\Re(s)$ large. Then for $\Re(s)$ large, we have 
\begin{align}
I_{\Theta_f} = \sum_{y\in \mathbb{P}Y'(F)} I(f,W_{\xi_s})(y).    
\end{align}
Therefore, $\sum_{y\in \mathbb{P}Y'(F)} I(f,W_{\xi_s})(y)$ admits a meromorphic continuation to all $s$-plane. 

By the functional equation of the Eisenstein series $E_{\xi_s}$, we obtain the desired functional equation for the sum of the global integral. Also, the poles of our sum of integrals come from the poles of $E_{\xi_s}$ (see \cite{GPSR, BG}). 
\end{proof}

\section{Bounds of the local integrals in the non-Archimedean case }\label{5}

Let $F$ be a non-Archimedean local field of characteristic zero. Let $K_G=G(\OO)$, and let $K_H=\SO_{2n+1}(\OO)$. In this section we give
bounds for the local factors of the global integral $I(f,W_{\xi_s})$ in the non-Archimedean case.

In \cref{5.1} and \cref{5.2} we first bound the inner integral of our local integral using some techniques from the proof of convergence of non-Archimedean Rankin-Selberg integral for $\SO_{2\ell+1} \times \GL_n$ in \cite[Section 4]{Sou}.  

We give a bound for our local integral in the general case in \cref{5.3}, and then in the unramified case in \cref{5.5}.

Let 
\begin{align*} 
\rho_{\tau,s}=\mathrm{Ind}_{Q_n(F)}^{H(F)}(\mathcal{W}(\tau, \psi_0) \otimes |\det |^{s-\tfrac{1}{2}}),
\end{align*}
and let $W_{\rho_{\tau,s}} \in \rho_{\tau,s}$. Here, as the notation indicates, we are viewing the induced representation as a space of smooth functions taking values in the Whittaker model.

The local integral is
\[
I_s(y):= I(f,W_{\rho_{\tau,s}})(y)=\int_{U_2(F)\backslash G(F)} \rho(g)f(y,1)\int_{N^{\circ}(F)} W_{\rho_{\tau,s}}(w ug,a_y)\psi_1(u)dudg. 
\]

For  $t=\mathrm{diag}(t_1,\dots,t_n) \in \GL_n(F)$, let
\begin{align}
t':=\begin{psmatrix} t & & \\ & 1 & \\ & & w_0t^{-1}w_0\end{psmatrix}\in T_H(F),
\end{align}
where $w_0 \in \GL_n(F)$ is the antidiagonal matrix.  

For a quasi-character $\eta:F^\times \to \CC^\times$ there is a unique real number $\Re(\eta)$ such that $\eta |\cdot|^{-\Re(\eta)}$ is unitary.  We say $\eta$ is \textbf{positive} if $\Re(\eta)>0$.

One can prove the following lemma using the same argument as that proving \cite[Lemma 4.4]{Sou}:

\begin{lemma}\label{5.1}
Let $(n,t',k) \in N_n(F) \times T_H(F) \times K_H$, we have
\[
|W_{\rho_{\tau,s}}(nt'k,a_y)|\leq |\det t|^{\Re(s)+\frac{n-1}{2}}\sum_{j=1}^{l}c_{j,s}\eta_{j}(a_yt).
\]
Here $c_{j,s} \in \CC$ and $\eta_{j}$ are positive quasi-characters of $T_H(F)$ which depend on $\tau$. \qed
\end{lemma}

The points of the group $w N^{\circ}w^{-1}$ in an $F$-algebra $R$ are 
\[
w N^{\circ}(R)w^{-1} = \left\{ 
\begin{psmatrix}
1 & &  &&&& \\ &1&&&&& \\ &&I_{n-2}&&&&& \\ 0&0&v_3'&1&&& \\ v_1 & v_2 & T & v_3 & I_{n-2} & & \\   0&0&v_2'&0&&1& \\ 0&0&v_1'&0&&&1
\end{psmatrix}: v_1, v_2, v_3 \in R^{n-2}, v_i'=-\,^{t}v_iJ_{n-2}, T\in M_{n-2}(R)
\right\}. 
\]

Let $v\in H(F)$ be a unipotent element of $H$ of the form 

\begin{align}\label{v}
v = \begin{psmatrix}
1&&&&&& \\ &1&&c&&-\frac{1}{2}c^2& \\ &&I_{n-2}&&&& \\ &&&1&&-c& \\ &&&&I_{n-2}&& \\ &&&&&1& \\ &&&&&&1 
\end{psmatrix}
\end{align}
for some  $c\in F$. 

\begin{lemma}\label{5.2} 
For $\Re(s)$ large, the integral 
\[
\int_{w N^{\circ}(F)w^{-1}} |W_{\rho_{\tau,s}}(uv,1)|du
\]
converges.
\end{lemma}

\begin{proof}
For $u\in w N^{\circ}(F)w^{-1}$, we have
\begin{align}\label{x}
uv = \begin{psmatrix}
1 & &  &&&& \\ &1&&c&&-\frac{1}{2}c^2& \\ &&I_{n-2}&&&&& \\ 0&0&v_3'&1&&-c& \\ v_1 & v_2 & T & v_2c+v_3 & I_{n-2} &-\frac{1}{2}v_2c^2-v_3c & \\   0&0&v_2'&0&&1& \\ 0&0&v_1'&0&&&1
\end{psmatrix}.
\end{align}

We denote the Iwasawa decomposition of $uv$ as $uv=nt'k$ where $(n,t',k) \in N_n(F) \times T_H(F) \times K_H$, and we denote the $i$-th line of $uv$ as $(uv)_i$.

By \cref{5.1}, the integral is majorized by
\begin{align} \label{bound}
\sum_{j=1}^{\nu} c_{j,s}\int_{w N^{\circ}(F)w^{-1}}[D(uv)]^{\Re(s)+\frac{n-1}{2}}E_j(uv)du,
\end{align}
where 
\begin{align*}
D(nt'k)=|\det t|,\\
E_j(nt'k)=\eta_j(t).
\end{align*}

We use techniques following \cite[Lemma 1, Section 11.15]{Sou}. Let $\{e_1,\dots, e_{2n+1}\}$ be the standard basis of $F^{2n+1}$. We take the sup-norm on $\wedge^pF^{2n+1}$ according to the basis $\{e_{i_1}\wedge e_{i_2} \wedge \dots \wedge e_{i_p}| 1\leq i_1 < \dots < i_p \leq n\}$. $K_H$ preserves this norm. We have 

\[
\lVert v_1 \wedge v_2 \wedge \dots \wedge v_p \rVert  \leq \lVert v_1\rVert \cdot \lVert v_2\rVert \cdots \lVert v_p \rVert, v_j\in F^{2n+1}.
\]

Let $e_{n+1+j}=e_{-n+j-1}$ for $j=1,\dots,n$, we have 
\begin{align*}
|t_{j+1}\cdots t_1|&=\lVert (e_{-(j+1)}uv) \wedge \cdots (e_{-1}uv)\rVert \\
&= \lVert (e_{-(j+1)}+(uv)_{2n+1-j})\wedge \cdots (e_{-1}+(uv)_{2n+1})\rVert \\
&\leq \prod_{i=0}^{j}\mathrm{max}\{1,\lVert (uv)_{2n+1-j}\rVert\} = \prod_{i=0}^{j}[(uv)_{2n+1-j}].
\end{align*}
Here $uv$ is a matrix, $e_{-{j+1}}uv$ is a vector, $[(uv)_{2n+1-j}]=\max \{1, \lVert (uv)_{2n+1-j} \rVert \}$ and $\lVert \cdot \rVert$ denotes the sup-norm.
For 
$$
h=\begin{psmatrix} h_1\\ \vdots \\ h_{2n+1}\end{psmatrix} \in H(F),
$$
we let
\begin{align} \label{B}
\mathcal{L}(h)=\begin{psmatrix} h_{n+2}\\ \vdots \\ h_{2n+1}\end{psmatrix}
\end{align}
be the bottom $n$ rows of $h$.

Since the coordinates of $\mathcal{L}(uv)$ appear in the coefficients of 
\[
t_{j+1}\cdots t_1 = e_{-(j+1)}\mathcal{L}(uv)\wedge \cdots e_{-1}\mathcal{L}(uv),
\]
we have
\[
|t_{j+1} \cdots t_1|^{-1} \geq \mathrm{max}\{1,\lVert \overline{x}_{2n+1-j}' \rVert,\dots,\lVert \overline{x}_{2n+1}' \rVert\}.
\]

We denote 
\begin{align}\label{symbol-[]}
[\mathcal{L}(uv)]=\max \{1, \lVert \mathcal{L}(uv) \rVert \},
\end{align}
where $\lVert \cdot \rVert$ is the sup-norm.

Then we have
\begin{align}\label{eta bound}
[\mathcal{L}(uv)]^{-2j}\leq |\frac{t_j}{t_{j+1}}|\leq [\mathcal{L}(uv)]^{2j}, j=1,\dots,n-1,
\end{align}
and 
\begin{align}\label{det bound}
[\mathcal{L}(uv)]^{-n}\leq D(u'v)\leq [\mathcal{L}(uv)]^{-1}.
\end{align}

Since $[\mathcal{L}(uv)] \leq [u][v]$, we have 

\begin{align} \label{E}
E_j(uv) \leq [\mathcal{L}(uv)]^{C} \leq [u]^{C}[v]^C
\end{align}
for some positive constant $C$ which depends only on $\tau$. 

By the structure of $\mathcal{L}(uv)$, we have 
\begin{align*}
[\mathcal{L}(uv)]^{-1} &= \max \{1, \lVert v_1 \rVert, \lVert v_2 \rVert, \lVert v_3 \rVert, \lVert T \rVert,  \lVert -\frac{1}{2}v_2c^2 \rVert, \lVert v_2c\rVert, \lVert v_3c \rVert\}^{-1} \\
&\leq \max \{1, \lVert v_1 \rVert, \lVert v_2 \rVert, \lVert v_3 \rVert, \lVert T \rVert \}^{-1}\\
&= [u]^{-1}.
\end{align*} 
Here the middle line holds since if $\max\{1, \lVert -\frac{1}{2}v_2c^2 \rVert, \lVert v_2c\rVert, \lVert v_3c \rVert\}=1$, we have equality, and we have 
$$\max \{1, \lVert v_1 \rVert, \lVert v_2 \rVert, \lVert v_3 \rVert, \lVert T \rVert,  \lVert -\frac{1}{2}v_2c^2 \rVert, \lVert v_2c\rVert, \lVert v_3c \rVert\} > \max \{1, \lVert v_1 \rVert, \lVert v_2 \rVert, \lVert v_3 \rVert, \lVert T \rVert \}$$
otherwise.

Thus we have 
\begin{align} \label{D}
D(uv) \leq [\mathcal{L}(uv)]^{-1} \leq [u]^{-1}. 
\end{align}

By \cref{E} and \cref{D}, for $\Re(s)+\frac{n-1}{2}-C>0$, \cref{bound} is bounded by

\begin{align}\label{Unram bound}
\sum_{j=1}^{\nu} c_{j,s}[v]^C \int_{w N^{\circ}(F)w^{-1}}[u]^{-\Re(s)-\frac{n-1}{2}+C}du,
\end{align}
which converges absolutely. 

\end{proof}

Now we proceed to bound our local integral in the general case. 

\begin{lemma}\label{5.3}
For $\Re(s)$ large enough, we have
\[
I_s(y) \ll_{\tau,s}  \int_{F^{\times}}\int_{F^{\times}} f'(a_1y_1, a_2y_2) |a_1|^{\Re(s)+\frac{d_1-n-1}{2}}|a_2|^{\Re(s)+\frac{d_2-n-1}{2}} d^{\times}a_1d^{\times}a_2.
\]
Here $f' \in \S(V(F))$ depends on $f$ and $\ll$ is synonymous with the big O notation. 

\end{lemma}

\begin{proof}
We apply the Iwasawa decomposition of $U_2(F)\backslash G(F)$ with respect to the usual Borel subgroup of $G(F)$. Since $W_{\rho_{\tau,s}}$ is smooth, it suffices to bound
\[
\int_{T_G(F)} \int_{K_G} |\rho(ak)f(y,1)| \delta_{B_G}^{-1}(a)\int_{N^{\circ}(F)}|W_{\rho_{\tau,s}}(w u\iota(ak),a_y)|dudadk.
\]

By the defining property of the Weil representation, the above integral is 
\begin{align*}
& \int_{G_1} \int_{A_1(F)} \int_{A_2(F)} \int_{K_G} |\rho(xa_1a_2k)f(y,1)| \delta_{B_G}^{-1}(a_1a_2) \\
& \times \int_{N^{\circ}(F)}|W_{\rho_{\tau,s}}|(w u\iota(xa_1a_2k),a_y)duda_1da_2dk  \\
= & \int_{F^{\times}} \int_{F^{\times}} \int_{F^{\times}} \int_{K_G} |\rho(k)f(a_1y_1,a_2y_2,b)||a_1|^{\frac{d_1}{2}}|a_2|^{\frac{d_2}{2}} |a_1a_2|^{-2}  \\
& \times \int_{N^{\circ}(F)} |W_{s, \phi_Q}|(w u\iota\left(\begin{psmatrix}
a_1& \\ &a_1^{-1}b^{-1}
\end{psmatrix},\begin{psmatrix}
a_2&\\&a_2^{-1}b
\end{psmatrix}\right)\iota(k),a_y)dud^{\times}bd^{\times}a_1d^{\times}a_2dk.
\end{align*}
Here $\iota\left(\begin{psmatrix}
a_1& \\ &a_1^{-1}b^{-1}
\end{psmatrix},\begin{psmatrix}
a_2&\\&a_2^{-1}b
\end{psmatrix}\right)$ is 
\[
\begin{psmatrix}
I_{n-1}&&&&&&\\&a_1a_2&&&&& \\ &&\frac{1}{2}+\frac{1}{4}(b^{-1}a_1a_2^{-1}+ba_1^{-1}a_2) & \frac{1}{2}(b^{-1}a_1a_2^{-1}-ba_1^{-1}a_2) & 2(b^{-1}a_1a_2^{-1}-ba_1^{-1}a_2)&& \\ &&(b^{-1}a_1a_2^{-1}-ba_1^{-1}a_2)& \frac{1}{2}(b^{-1}a_1a_2^{-1}+ba_1^{-1}a_2) & -\frac{1}{2}(b^{-1}a_1a_2^{-1}-ba_1^{-1}a_2)&& \\ &&\frac{1}{2}(\frac{1}{2}-\frac{1}{4}(b^{-1}a_1a_2^{-1}+ba_1^{-1}a_2)) & -\frac{1}{4}(b^{-1}a_1a_2^{-1}-ba_1^{-1}a_2) & \frac{1}{2}+\frac{1}{4}(b^{-1}a_1a_2^{-1}+ba_1^{-1}a_2)&&\\  &&&&&(a_1a_2)^{-1}& \\&&&&&&I_{n-1}
\end{psmatrix}. 
\]

Apply the Iwasawa decomposition to $\SO_3(F)$ we have for $b\in F^{\times}$,
\[
\begin{psmatrix} \frac{1}{2}+\frac{1}{4}(b+b^{-1}) & \frac{1}{2}(b-b^{-1}) & 2(b-b^{-1}) \\ (b-b^{-1})& \frac{1}{2}(b+b^{-1}) & -\frac{1}{2}(b-b^{-1}) \\ \frac{1}{2}(\frac{1}{2}-\frac{1}{4}(b+b^{-1})) & -\frac{1}{4}(b-b^{-1}) & \frac{1}{2}+\frac{1}{4}(b+b^{-1}) \end{psmatrix}=\begin{psmatrix}
1 & c & -\frac{1}{2}c^2  \\
0 & 1 & -c \\
0 & 0 & 1 
\end{psmatrix} \begin{psmatrix}
\lfloor b \rfloor  & 0 & 0  \\
0 & 1 & 0 \\
0 & 0 & \lfloor b \rfloor ^{-1} 
\end{psmatrix} k'.
\]
Here $k'\in \SO_3(\OO)\subset K_H$, 
and 
\begin{align}\label{floornotation}
\lfloor b \rfloor = \begin{cases} b \textrm{ if } |b|\leq 1 \\
b^{-1} \textrm{ if } |b|> 1
\end{cases},
\end{align}
\[
c =\begin{cases}-2 & \textrm{ if }\lfloor b \rfloor =b\\
2   &  \textrm{ if } \lfloor b \rfloor =b^{-1}
\end{cases}.
\]

We also notice that
\[
\begin{psmatrix}
1 & c & -\frac{1}{2}c^2  \\
0 & 1 & -c \\
0 & 0 & 1 
\end{psmatrix} \begin{psmatrix}
\lfloor b^2 \rfloor  & 0 & 0  \\
0 & 1 & 0 \\
0 & 0 & \lfloor b \rfloor ^{-1}
\end{psmatrix} = \begin{psmatrix}
\lfloor b \rfloor  & 0 & 0  \\
0 & 1 & 0 \\
0 & 0 & \lfloor b \rfloor ^{-1}
\end{psmatrix} \begin{psmatrix}
1 & c\lfloor b \rfloor ^{-1} & -\frac{1}{2}c^2\lfloor b \rfloor ^{-2}  \\
0 & 1 & -c\lfloor b \rfloor ^{-1} \\
0 & 0 & 1 
\end{psmatrix}.
\]

Thus our local integral is majorized by 
\begin{align}\begin{split}\label{eq4.0.13}
&\int_{F^{\times}} \int_{F^{\times}} \int_{F^{\times}} \int_{K_G} |\rho(k)f(a_1y_1,a_2y_2,b)| |a_1|^{\frac{d_1}{2}}|a_2|^{\frac{d_2}{2}} |a_1a_2|^{-2}  \\
&\times \int_{N^{\circ}(F)}|W_{\rho_{\tau,s}}(wu \; \mathrm{diag}(I_{n-2},a_1a_2,I_3,(a_1a_2)^{-1},I_{n-2}) tn'k''\iota(k), a_y)|dud^{\times}bd^{\times}a_1d^{\times}a_2dk.
\end{split}
\end{align}
Here
\begin{align} 
t &= \begin{psmatrix}
I_{n-1}&&&& \\ &\lfloor b^{-1}a_1a_2^{-1} \rfloor &&&  \\
&& 1 &&  \\
&&& \lfloor ba_1^{-1}a_2 \rfloor  & \\
&&&&I_{n-1}
\end{psmatrix}, \\
n'&=\begin{psmatrix}
I_{n-1}&&&& \\
&1 & c\lfloor ba_1^{-1}a_2 \rfloor & -\frac{1}{2}c^2\lfloor ba_1^{-1}a_2 \rfloor^2&  \\
&& 1 & -c\lfloor ba_1^{-1}a_2 \rfloor& \\
&&&1& \\
&&&&I_{n-1}\\
\end{psmatrix}, \label{n'} \\
 k''&= \begin{psmatrix}
I_{n-1} && \\ 
&k'& \\
&&  I_{n-1}
\end{psmatrix}.
\end{align}

We have
\begin{align*}
&w \begin{psmatrix}
I_{n-2}&&&&&& \\ 
&a_1a_2&&&&& \\
&&\lfloor b^{-1}a_1a_2^{-1} \rfloor &&&&\\
&&&1&&& \\
&&&&\lfloor ba_1^{-1}a_2 \rfloor&&\\
&&&&&(a_1a_2)^{-1}& \\
&&&&&&I_{n-2}
\end{psmatrix} w^{-1} \\
=& \begin{psmatrix}
a_1a_2&&&&&& \\ 
&\lfloor b^{-1}a_1a_2^{-1} \rfloor &&&&& \\
&&I_{n-2}&&&&\\
&&&1&&& \\
&&&&I_{n-2}&&\\
&&&&&\lfloor ba_1^{-1}a_2 \rfloor& \\
&&&&&&(a_1a_2)^{-1}
\end{psmatrix}.
\end{align*}
Then by the property of $W_{\rho_{\tau,s}}$ \cref{eq4.0.13} is
\begin{multline*}
\int_{F^{\times}} \int_{F^{\times}} \int_{F^{\times}} \int_{K_G} |\rho(k)f(a_1y_1,a_2y_2,b)||a_1|^{\frac{d_1}{2}}|a_2|^{\frac{d_2}{2}} |a_1a_2|^{\Re(s)-\frac{n+1}{2}} |\lfloor b^{-1}a_1a_2^{-1} \rfloor|^{\Re(s)-\frac{n-3}{2}} \\
\times \int_{N^{\circ}(F)} |W_{\rho_{\tau,s}}(w un'k''\iota(k), a_y\mathrm{diag}(a_1a_2, \lfloor b^{-1}a_1a_2^{-1} \rfloor , I_{n-2}))|dud^{\times}bd^{\times}a_1d^{\times}a_2dk.
\end{multline*}
This is 
\begin{align*}
& \int_{F^{\times}} \int_{F^{\times}} \int_{F^{\times}} \int_{K_G} |\rho(k)f(a_1y_1,a_2y_2,b)||a_1|^{\frac{d_1}{2}}|a_2|^{\frac{d_2}{2}} |a_1a_2|^{\Re(s)-\frac{n+1}{2}} |\lfloor b^{-1}a_1a_2^{-1} \rfloor|^{\Re(s)-\frac{n-3}{2}}\\
& \times \int_{N^{\circ}(F)} |W_{\rho_{\tau,s}}((w uw^{-1})(w n'w^{-1})w k''\iota(k), a_y\mathrm{diag}(a_1a_2, \lfloor b^{-1}a_1a_2^{-1} \rfloor , I_{n-2}))|\\
& \times dud^{\times}bd^{\times}a_1d^{\times}a_2dk.
\end{align*}

Since  
\[w n'w^{-1} = \begin{psmatrix}
1&&&&&& \\ &1&&-\frac{1}{2}c\lfloor ba_1^{-1}a_2 \rfloor ^{-1}&&-\frac{1}{4}c^2\lfloor ba_1^{-1}a_2 \rfloor^2& \\ &&I_{n-2}&&&& \\ &&&1&&\frac{1}{2}c\lfloor ba_1^{-1}a_2 \rfloor& \\ &&&&I_{n-2}&& \\ &&&&&1& \\ &&&&&&1 
\end{psmatrix},\]we have
\[
[w n'w^{-1}] \ll |\lfloor b^{-1}a_1a_2^{-1} \rfloor|^{-2}.
\]
Then the above integral is majorized by 
\begin{multline*}
\int_{F^{\times}} \int_{F^{\times}} \int_{F^{\times}}\int_{K_G}|\rho(k)f(a_1y_1,a_2y_2,b)||a_1|^{\frac{d_1}{2}}|a_2|^{\frac{d_2}{2}}|a_1a_2|^{\Re(s)-\frac{n+1}{2}} |\lfloor b^{-1}a_1a_2^{-1} \rfloor|^{\Re(s)-\frac{n-3}{2}}\\
\times  \int_{w N^{\circ}(F)w^{-1}} |W_{\rho_{\tau,s}}(u(w n'w^{-1})w k''\iota(k), a_y\mathrm{diag}(a_1a_2, \lfloor b^{-1}a_1a_2^{-1} \rfloor , I_{n-2}))|dud^{\times}bd^{\times}a_1d^{\times}a_2dk.
\end{multline*}

Since $w k''k \in K_H$, we apply \cref{5.1} and \cref{5.2}. Then the local integral is majorized by
\begin{align}\label{sum of integrals}\begin{split}
&\sum_{j=1}^{\nu} c_{j,s} \int_{F^{\times}} \int_{F^{\times}} \int_{F^{\times}} \int_{K_G}|\rho(k)f(a_1y_1,a_2y_2,b)||a_1|^{\frac{d_1}{2}}|a_2|^{\frac{d_2}{2}} |a_1a_2|^{\Re(s)-\frac{n+1}{2}}  \\  \times &  |\lfloor b^{-1}a_1a_2^{-1} \rfloor|^{\Re(s)-\frac{n-3}{2}-2C} \eta_j(\mathrm{diag}(-4\mathcal{Q}'(y_2)a_1a_2, \lfloor b^{-1}a_1a_2^{-1} \rfloor , I_{n-2})) \\
\times & \left(\int_{w N^{\circ}(F)w^{-1}}[u]^{-\Re(s)-\frac{n-1}{2}+C}du\right)d^{\times}bd^{\times}a_1d^{\times}a_2dk.
\end{split}
\end{align}
Here $\eta_j$ are positive quasi-characters that depend only on $\tau$. 

Also, for  $u\in w N^{\circ}(F)w^{-1}$, if we denote the Iwasawa decomposition of $u(w n'w^{-1})$ by $u(w n'w^{-1})=nt'k$ (using notations as in \cref{5.1}), then $a_y\mathrm{diag}(a_1a_2, \lfloor b^{-1}a_1a_2^{-1} \rfloor , I_{n-2})t$ lies in the support of a gauge on $\GL_n(F)$. Also, there are constants $c_1$ that depends only on $\tau$ such that
\[
\left|\frac{-4Q'(y_2)a_1a_2}{|\lfloor b^{-1}a_1a_2^{-1} \rfloor| }\frac{t_1}{t_2}\right| \leq c_1. 
\]

Thus by \cref{5.2} we have
\begin{align}\label{m bound}
|-4Q'(y_2)a_1a_2| \leq c_1[u]|\lfloor b^{-1}a_1a_2^{-1} \rfloor|^{-2}.
\end{align}

Therefore, 
\[
|\eta_j|( a_y\mathrm{diag}(a_1a_2, \lfloor b^{-1}a_1a_2^{-1} \rfloor , I_{n-2})) \leq c_1[u]^{c_2}|\lfloor b^{-1}a_1a_2^{-1} \rfloor|^{-2c_2-c_3}
\] 
for some positive integers $c_2$, $c_3$ that depend only on $\tau$ for $j=1,\dots,\nu$. 

Thus the integral is majorized by 
\begin{multline*}
\sum_{j=1}^{\nu} c_{j,s} \int_{F^{\times}} \int_{F^{\times}} \int_{F^{\times}} \int_{K_G}|\rho(k)f(a_1y_1,a_2y_2,b)| |a_1|^{\frac{d_1}{2}}|a_2|^{\frac{d_2}{2}}|a_1a_2|^{\Re(s)-\frac{n+1}{2}} \\ \times  |\lfloor b^{-1}a_1a_2^{-1} \rfloor|^{\Re(s)-\frac{n-3}{2}-2C-4c_2-c_3} \left(\int_{w N^{\circ}(F)w^{-1}}[u]^{-\Re(s)-\frac{n-1}{2}+C+2c_2}du\right)d^{\times}bd^{\times}a_1d^{\times}a_2dk.
\end{multline*}

Then by the definition of $[\cdot ]$ (as in \cref{symbol-[]}), $[u] = \mathrm{max}\{1, \lVert u \rVert\} \geq 1$ for $u\in w N^{\circ}(F)w^{-1}$, then for $\Re(s)$ large the above is majorized by 
\begin{align*}
\sum_{j=1}^{\nu} c_{j,s} \int_{F^{\times}} \int_{F^{\times}} \int_{F^{\times}} \int_{K_G}|\rho(k)f(a_1y_1,a_2y_2,b)| |a_1|^{\frac{d_1}{2}}|a_2|^{\frac{d_2}{2}}|a_1a_2|^{\Re(s)-\frac{n+1}{2}} \\ \times  |\lfloor b^{-1}a_1a_2^{-1} \rfloor|^{\Re(s)-\frac{n-3}{2}-2C-4c_2-c_3} d^{\times}bd^{\times}a_1d^{\times}a_2dk.
\end{align*}

By definition of the symbol $\lfloor \cdot \rfloor$ (see \cref{floornotation}), $|\lfloor b^{-1}a_1a_2^{-1} \rfloor|\leq 1$ for $a_1,a_2,b \in F^{\times}$. Thus for $\Re(s)$ large, the above sum of integrals is majorized by a constant depends on $\tau,s$ times
\[
\int_{F^{\times}} \int_{F^{\times}} \int_{K_G}|\rho(k)\tilde{f}(a_1y_1,a_2y_2)| |a_1|^{\frac{d_1}{2}}|a_2|^{\frac{d_2}{2}} |a_1a_2|^{\Re(s)-\frac{n+1}{2}} d^{\times}a_1d^{\times}a_2dk,
\]
where $\tilde{f} \in \S(V(F))$ depends on $f$. 

Let $f'(v)=\int_{K_G}\rho(k)\tilde{f}(v)$ for $v\in V(F)$, we have $f' \in \S(V(F))$. The integral above is equal to
\[
\int_{F^{\times}} \int_{F^{\times}} |f'(a_1y_1,a_2y_2)||a_1|^{\frac{d_1}{2}}|a_2|^{\frac{d_2}{2}} |a_1a_2|^{\Re(s)-\frac{n+1}{2}} d^{\times}a_1d^{\times}a_2,
\]
which converges for $\Re(s)$ large enough.
\end{proof}

Now we give a bound for the local integral in the unramified case. Suppose the data are unramified as in \cref{4}. Suppose $F$ is such that $q\geq n$, and $|Q'(y_2)|=1$. Let $f=\one_{V(\OO)\times \OO^{\times}}$.

In the unramified case, $\rho(g)f(y,1)$ and $W_{\rho_{\tau,s}}$ are right invariant by $K_G$ and $K_H$, so after applying the Iwasawa decomposition, we have that
\[
I_s(y) = 
\int_{M_1(F)}\int_{G_1(F)} \rho(mx)f(y) \delta_{B_G}^{-1}(mx) \int_{N^{\circ}(F)}W_{\rho_{\tau,s}}(w u\iota(mx),a_y)\psi(u)dudmdx.
\]

\begin{lemma} \label{5.4}
Let $a=(a_1,\dots,a_n)\in \GL_n(F)$ be such that 
\[|a_1|=q^{-k_{y_1}}\geq |a_2|=q^{-k_{y_2}} \geq \cdots \geq |a_n|=q^{-k_n}, \]
where $k_n\leq 0$ (i.e. the positive cone). Let $W_{\tau}$ be the unramified Whittaker function for $\tau$. There exists a positive integer $c_0$ which depends on $\tau$ such that 
\[
|W_{\tau}(a)| \leq \delta_{B_{\GL_n}}^{\frac{1}{2}}(a) q^{-k_nnc_0}|\mathrm{det}a|^{-c_0}.
\]
\end{lemma}

\begin{proof}
The result follows from arguments in \cite[Section 2.4]{JIS} for $k_n=0$ by twisting the corresponding rational representation of $\GL_n(\C)$ in the explicit formula of $W_{\tau}(a)$ (see \cite{CS}).  
\end{proof}

\begin{lemma} \label{5.5}
For $\Re(s)$ large enough, we have
\begin{align*}
    & |I_s(y)| \\
    &\leq \zeta_v(\Re(s)+c_{\tau}) \int_{F^{\times}} \int_{F^{\times}} |\one_{V(\OO_F)}(a_1y_1, a_2y_2)| |a_1|^{\Re(s)+\frac{d_1-n-1}{2}-c_0}|a_2|^{\Re(s)+\frac{d_2-n-1}{2}-c_0} d^{\times}a_1d^{\times}a_2,
\end{align*}
where $c_0>0$, $c_{\tau}$ are integers depend only on $\tau$.
\end{lemma}

\begin{proof}
As in \cref{5.3}, and since in the unramified case $|2|=1$, $\rho(k)f(y)=f(y)$ for $k\in K_G$, we have 
\begin{multline*}
|I_s(y)| \leq \int_{F^{\times}} \int_{F^{\times}} \int_{F^{\times}}  |f(a_1y_1,a_2y_2,b)||a_1|^{\frac{d_1}{2}}|a_2|^{\frac{d_2}{2}} |a_1a_2|^{\Re(s)-\frac{n+1}{2}} |\lfloor b^{-1}a_1a_2^{-1} \rfloor|^{\Re(s)-\frac{n-3}{2}} \\
\times \int_{w N^{\circ}(F)w^{-1}} |W_{\rho_{\tau,s}}|(u(w n'w^{-1}), a_y\mathrm{diag}(a_1a_2, \lfloor b^{-1}a_1a_2^{-1} \rfloor , I_{n-2}))dud^{\times}bd^{\times}a_1d^{\times}a_2.
\end{multline*}

Using the notation as in \cref{5.1} and \cref{5.2}, we write
 $u(w n'w^{-1})=nt'k$, by \cref{eta bound} and \cref{det bound} we have 
\[
|t_n| \leq [\mathcal{B}(u(w n'w^{-1}))]^{n-2} \leq ([u]|\lfloor b^{-1}a_1a_2^{-1} \rfloor|^{-2})^{n-2}. 
\]

Also, by the property of $W_{\rho_{\tau,s}}$, we have 
\begin{align*}
&|W_{\rho_{\tau,s}}|(u(w n'w^{-1}), a_y\mathrm{diag}(a_1a_2, \lfloor b^{-1}a_1a_2^{-1} \rfloor , I_{n-2})) \\=& |D(u(w n'w^{-1}))|^{\Re(s)+\frac{n-1}{2}}W_{\tau}(a_y\mathrm{diag}(a_1a_2, \lfloor b^{-1}a_1a_2^{-1} \rfloor , I_{n-2})t). 
\end{align*}

Thus by \cref{5.4} and \cref{det bound}, for $\Re(s)$ large we have 
\begin{align*}
&|W_{\rho_{\tau,s}}|(u(w n'w^{-1}), a_y\mathrm{diag}(a_1a_2, \lfloor b^{-1}a_1a_2^{-1} \rfloor , I_{n-2})) \\ \leq& [u]^{-\Re(s)-n-\frac{1}{2}}([u]|\lfloor b^{-1}a_1a_2^{-1} \rfloor|^{-2})^{(n-2)nc_0}|\mathrm{det} (a_y\mathrm{diag}(a_1a_2, \lfloor b^{-1}a_1a_2^{-1} \rfloor , I_{n-2})t)|^{-c_0} \\
=& [u]^{-\Re(s)-n-\frac{1}{2}}([u]|\lfloor b^{-1}a_1a_2^{-1} \rfloor|^{-2})^{(n-2)nc_0}|\mathrm{det} (\mathrm{diag}(a_1a_2, \lfloor b^{-1}a_1a_2^{-1} \rfloor , I_{n-2})t)|^{-c_0}.
\end{align*}

By \cref{det bound} again, we have 
\[
|\mathrm{det} t|^{-c_0} \leq ([u]|\lfloor b^{-1}a_1a_2^{-1} \rfloor|^{-2})^{c_0n}.
\]

Thus, we obtain 
\begin{align*}
& |W_{\rho_{\tau,s}}|(u(w n'w^{-1}), \mathrm{diag}(a_1a_2, \lfloor b^{-1}a_1a_2^{-1} \rfloor , I_{n-2}))\\ \leq& |a_1a_2|^{-c_0} [u]^{-\Re(s)-n-\frac{1}{2}+c_0n^2-c_0n}|\lfloor b^{-1}a_1a_2^{-1} \rfloor|^{-c_0(2n^2-2n+1)}.
\end{align*}

Therefore, 
\begin{multline*}
|I_s(y)| \leq \int_{F^{\times}} \int_{F^{\times}} \int_{F^{\times}} |f(a_1y_1,a_2y_2,b)||a_1|^{\frac{d_1}{2}}|a_2|^{\frac{d_2}{2}} |a_1a_2|^{\Re(s)-\frac{n-1}{2}-c_0} \\
\times |\lfloor b^{-1}a_1a_2^{-1} \rfloor|^{\Re(s)-\frac{n-3}{2}-c_0(2n^2-2n+1)} \int_{w N^{\circ}(F)w^{-1}} [u]^{-\Re(s)-n-\frac{1}{2}+c_0n^2-c_0n}du d^{\times}a_1d^{\times}a_2.
\end{multline*}

Then, for $\Re(s)$ large, we have 
\begin{align*}
|I_s(y)| &\leq \int_{F^{\times}} \int_{F^{\times}} \int_{F^{\times}}  |\one_{V(\OO_F)\times \OO_F^{\times}}(a_1y_1,a_2y_2,b)||a_1|^{\frac{d_1}{2}}|a_2|^{\frac{d_2}{2}} |a_1a_2|^{\Re(s)-\frac{n-1}{2}-c_0} \\
&\times \int_{w N^{\circ}(F)w^{-1}} [u]^{-\Re(s)-n-\frac{1}{2}+c_0n^2-c_0n}dud^{\times}bd^{\times}a_1d^{\times}a_2\\
&= \int_{F^{\times}} \int_{F^{\times}}|\one_{V(\OO_F)}(a_1y_1,a_2y_2)||a_1|^{\frac{d_1}{2}}|a_2|^{\frac{d_2}{2}} |a_1a_2|^{\Re(s)-\frac{n-1}{2}-c_0} \\
&\times \int_{w N^{\circ}(F)w^{-1}} [u]^{-\Re(s)-n-\frac{1}{2}+c_0n^2-c_0n}dud^{\times}a_1d^{\times}a_2. 
\end{align*}

We have 
\begin{align*}
\int_{w N^{\circ}(F)w^{-1}} [u]^{-\Re(s)-n-\frac{1}{2}+2c_0n(n-2)-c_0}du &\leq 1+\sum_{k=1}^{\infty}\int_{[u]=q^k} q^{-(\Re(s)+n+\frac{1}{2}-2c_0n(n-2)+c_0)k}du \\
&\leq  1+\frac{q^{-\Re(s)+(c_0+1)n^2-(c_0+2)n-c_0-\frac{5}{2}}}{1-q^{-\Re(s)+(c_0+1)n^2-(c_0+2)n-c_0-\frac{5}{2}}} \\
& = \zeta_v(\Re(s)-(c_0+1)n^2+(c_0+2)n+c_0+\tfrac{5}{2}).
\end{align*}

Therefore, we obtain 
\begin{align*}
|I_s(y)| \leq&  \zeta_v(\Re(s)-(c_0+1)n^2+(c_0+2)n+c_0+\tfrac{5}{2}) \\
\times& \int_{F^{\times}} \int_{F^{\times}} |\one_{V(\OO_F)}(a_1y_1,a_2y_2)| |a_1|^{\frac{d_1}{2}}|a_2|^{\frac{d_2}{2}} |a_1a_2|^{\Re(s)-\frac{n-1}{2}-c_0}d^{\times}a_1d^{\times}a_2 \\
=& \zeta_v(\Re(s)+c_{\tau}) \int_{F^{\times}}\int_{F^{\times}} |\one_{V(\OO_F)}(a_1y_1, a_2y_2)||a_1|^{\frac{d_1}{2}}|a_2|^{\frac{d_2}{2}} |a_1a_2|^{\Re(s)-\frac{n-1}{2}-c_0}d^{\times}a_1d^{\times}a_2,
\end{align*}
where $c_{\tau}$ is some integer depends only on $\tau$. 
\end{proof}

\section{Convergence of the local integral in the Archimedean case} \label{7}

In this section we give bounds for the local integrals in the Archimedean case. Let $F$ be an Archimedean local field. Let $K_G$ be the maximal compact subgroup of $G(F)$, and let $K_H$ be the maximal compact subgroup of $H(F)$. 

The local integral in the Archimedean case is
\[
I_s(y) :=I(f,W_{\rho_{\tau,s}})= \int_{U_2(F)\backslash G(F)} \rho(g)f(y,1)\int_{N^{\circ}(F)} W_{\rho_{\tau,s}}(w u\iota(g),a_y)\psi_1(u)dudg.
\]

As in the non-Archimedean case, in \cref{7.1} and \cref{7.2} we first use techniques similar to those employed in the proof of absolute convergence of the Archimedean local integral for the Rankin-Selberg integral for $\SO_{2\ell+1} \times \GL_n$ in \cite[Section 5]{Sou} to bound our inner integral.  

For  $t \in \GL_n(F)$ let
\begin{align}
t':=\begin{psmatrix} t & & \\ & 1 & \\ & & w_0t^{-1}w_0\end{psmatrix},
\end{align}
where $w_0 \in \GL_n(F)$ is the antidiagonal matrix.  

Arguing analogously as in \cite[Lemma 5.2]{Sou}, we have 
\begin{lemma}\label{7.1}
Let $(n,t',k) \in N_n(F) \times T_H(F) \times K_H$, there is a positive integer $N$ and $c_s \in \RR_{>0}$ depending on $W_{\rho_{\tau,s}}$ such that 
\[
|W_{\rho_{\tau,s}}(nt'k,1)|\leq c_s|\det t|^{\Re(s)+\frac{n-1}{2}}\lVert t \rVert^N
\]
for $t=\mathrm{diag}(t_1, t_2, \dots, t_{n-1}, 1)$. Here
\[
\lVert t \rVert^2 = 1 + \sum_{i=1}^{n-1}|t_i|^2+\sum_{i=1}^{n-1}|t_i|^{-2}.
\]

\qed

\end{lemma}

\begin{lemma} \label{7.2}
For $v \in H(F)$ as defined in \cref{v} and $\mathrm{Re}(s)$ large,
\[
\int_{w N^{\circ}(F)w^{-1}} |W_{\rho_{\tau,s}}(uv,1)|du
\]
converges.
\end{lemma}

\begin{proof}
As in \cref{5.2}, for $u \in w N^{\circ}(F) w^{-1}$ and $v$ a unipotent element of the form \cref{v}, we denote the Iwasawa decomposition of $uv$ as $uv=nt'k$ where $(n,t',k) \in N_n(F) \times T_H(F) \times K_H$, and we denote the $i$-th line of $uv$ as $(uv)_i$.

By \cref{7.1}, there is a positive integer $N$ such that for $W_{\rho_{\tau,s}}$ there is $c_s \in \RR_{>0}$ such that
\begin{align}\label{Arc}
|W_{\rho_{\tau,s}}(nt'k,1)|\leq c_s|\det t|^{\Re(s)+\frac{n-1}{2}}|w_{\tau}(t_n)|\lVert t \rVert^{N}.
\end{align}
Here $t=\mathrm{diag} (t_1t_2\cdot \ldots \cdot t_n, t_2\cdot \ldots \cdot t_n, \ldots, t_{n-1}t_n, t_n)$, $\lVert
t\rVert=\lVert t_{n}^{-1}t \rVert$, and $w_{\tau}$ is the central character of $\tau$. We assume $N$ is even, then $\lVert a \rVert^{N}$ is a sum of positive quasicharacters.

As in the non-Archimedean case, we denote $\mathcal{L}(uv)=\begin{psmatrix}
(uv)_{n+2}\\ \vdots \\ (uv)_{2n+1}
\end{psmatrix}$. 

Using the technique as in \cite[Section 7.3, Lemma 3]{Sou} we get
\[
(1+\lVert \mathcal{L}(uv) \rVert^2)^{-\frac{n}{2}} \leq \det(t) \leq (1+\lVert \mathcal{L}(uv) \rVert^2)^{-\frac{1}{2}}.
\]
Here $\lVert \mathcal{L}(uv) \rVert$ denotes the standard norm on $M_{n\times (2n+1)}(F)$, and
\begin{align} \label{character}
\max \{\tfrac{t_j}{t_{j+1}}, \tfrac{t_{j+1}}{t_j}\}\leq (1+\lVert \mathcal{L}(uv) \rVert^2)^n, j=1,\ldots,n-1.
\end{align}

Similar to the non-Archimedean case, we have 
\begin{align*}
(1+\lVert \mathcal{L}(uv)\rVert)^{-\frac{1}{2}} &= (1+\sup \{\lVert v_1 \rVert, \lVert v_2 \rVert, \lVert v_3 \rVert, \lVert T \rVert,  \lVert -\frac{1}{2}v_2c^2 \rVert, \lVert v_3c \rVert\} )^{-\frac{1}{2}}\\
& \leq   (1+\sup \{\lVert v_1 \rVert, \lVert v_2 \rVert, \lVert v_3 \rVert, \lVert T \rVert\} )^{-\frac{1}{2}} \\
&= (1+\lVert u\rVert)^{-\frac{1}{2}},
\end{align*}
where $\lVert \cdot \rVert$ denotes the standard matrix norms. 

By \cref{Arc} we have
\begin{align}\label{B1}
|W_{\rho_{\tau,s}}|(uv,1) \leq \sum_j c_s (1+\lVert u\rVert)^{-\frac{\Re(s)}{2}-\frac{n-2}{4}}|\chi_j(t)|.
\end{align}
Here the sum is finite and the $\chi_j$ are positive quasi-characters which depend on $\tau$. 

By \cref{character}, we have
\[
|\chi_j(t)| \leq (1+\lVert \mathcal{L}(uv)\rVert^2)^C \leq (1+\lVert u \rVert^2)^C(1+\lVert v \rVert^4)^C
\]
for some positive constant $C$ which depends on $\tau$. 

Thus we have 
\begin{align}\label{B2}
\int_{w N^{\circ}(F)w^{-1}} |W_{\rho_{\tau,s}}(uv,1)|du \leq \sum_j c_s(1+\lVert v \rVert^4)^C \int_{w N^{\circ}(F)w^{-1}} (1+\lVert u \rVert)^{-\frac{\Re(s)}{2}-\frac{n-2}{4}+C}du.
\end{align}
The sum over $j$ is finite. Since by definition $\lVert \cdot \rVert$ is positive, the integral converges for $\Re(s)$ large.

\end{proof}
We now proceed to bound our local integral $I_s(y)$. 

\begin{lemma} \label{7.3}
For $y=(y_1,y_2)\in \mathbb{P}Y'(F)$, we have that
\[
I_s(y) \ll   \int_{F^{\times}} \int_{F^{\times}} f'(a_1y_1, a_2y_2) |a_1|^{\Re(s)+\frac{d_1-n-1}{2}-c_0}|a_2|^{\Re(s)+\frac{d_2-n-1}{2}-c_0} d^{\times}a_1d^{\times}a_2
\]
for $\Re(s)$ large enough where $f' \in \S(V(F))$ is nonnegative and $c_0 \in \RR_{ \geq 0}$ depends only on $\tau$ as in \cref{5.5}.
\end{lemma}

\begin{proof}
Applying the Iwasawa decomposition of $G(F)$ with respect to the standard Borel subgroup, we have 
\[
I_s(y) =  \int_{T_G(F)} \int_{K_G} \rho(ak) f(y,1) \delta_{B_G}^{-1}(a) \int_{N^{\circ}(F)} W_{\rho_{\tau,s}}(w u\iota(ak),a_y)\psi_1(u)dud^{\times}a_1d^{\times}a_2dk.
\]

By the action of Weil representation it suffices to bound 
\begin{multline*}
 \int_{F^{\times}}  \int_{F^{\times}} \int_{F^{\times}} \int_{K_G} |\rho(k)f(a_1y_1,a_2y_2,b)| |a_1|^{\frac{d_1}{2}}|a_2|^{\frac{d_2}{2}}|a_1a_2|^{-2}  \\
\times \int_{N^{\circ}(F)} |W_{\rho_{\tau,s}}|(w u\iota\left(\begin{psmatrix}
a_1& \\ &a_1^{-1}b^{-1}
\end{psmatrix},\begin{psmatrix}
a_2&\\&a_2^{-1}b\end{psmatrix}\right)\iota(k),a_y)dud^{\times}bd^{\times}a_1d^{\times}a_2dk.
\end{multline*}
Here $\iota\left(\begin{psmatrix}
a_1& \\ &a_1^{-1}b^{-1}
\end{psmatrix},\begin{psmatrix}
a_2&\\&a_2^{-1}b
\end{psmatrix}\right)$ is 
\[
\begin{psmatrix}
I_{n-1}&&&&&&\\&a_1a_2&&&&& \\ &&\frac{1}{2}+\frac{1}{4}(b^{-1}a_1a_2^{-1}+ba_1^{-1}a_2) & \frac{1}{2}(b^{-1}a_1a_2^{-1}-ba_1^{-1}a_2) & 2(b^{-1}a_1a_2^{-1}-ba_1^{-1}a_2)&& \\ && (b^{-1}a_1a_2^{-1}-ba_1^{-1}a_2)& \frac{1}{2}(b^{-1}a_1a_2^{-1}+ba_1^{-1}a_2) & -\frac{1}{2}(b^{-1}a_1a_2^{-1}-ba_1^{-1}a_2)&& \\ &&\frac{1}{2}(\frac{1}{2}-\frac{1}{4}(b^{-1}a_1a_2^{-1}+ba_1^{-1}a_2)) & -\frac{1}{4}(b^{-1}a_1a_2^{-1}-ba_1^{-1}a_2) & \frac{1}{2}+\frac{1}{4}(b^{-1}a_1a_2^{-1}+ba_1^{-1}a_2)&&\\  &&&&&(a_1a_2)^{-1}& \\&&&&&&I_{n-1}
\end{psmatrix}.
\]

As in the non-Archimedean case, applying the Iwasawa decomposition of $\SO_3(F)$ we get for $b\in F^{\times}$
\[
\begin{psmatrix} \frac{1}{2}+\frac{1}{4}(b+b^{-1}) & \frac{1}{2}(b-b^{-1}) & 2(b-b^{-1}) \\ (b-b^{-1})& \frac{1}{2}(b+b^{-1}) & -\frac{1}{2}(b-b^{-1}) \\ \frac{1}{2}(\frac{1}{2}-\frac{1}{4}(b+b^{-1})) & -\frac{1}{4}(b-b^{-1}) & \frac{1}{2}+\frac{1}{4}(b+b^{-1}) \end{psmatrix}=\begin{psmatrix}
1 & c & -\frac{1}{2}c^2  \\                                                                                                                      
0 & 1 & -c \\
0 & 0 & 1 
\end{psmatrix}\begin{psmatrix}
a & 0 & 0  \\
0 & 1 & 0 \\
0 & 0 & a^{-1} 
\end{psmatrix}k'.
\]
Here $k' \in K_{\SO_3}\subset K_H$.

For $F=\mathbb{R}$, let $K'=\mathrm{diag}(\SO(2,\mathbb{R}),1)$ and S be such that $\,^{t}SJ'S=J_3$, where $J'=\mathrm{diag}(-I_2,1)$. We have $K_{\SO_3}=\,^{t}SK'S$. Let $\lVert\cdot \rVert$ denote the Euclidean vector norm. Using the above decomposition, we have 
\begin{align*}
&\lVert (\tfrac{1}{2}(\tfrac{1}{2}-\tfrac{1}{4}(b+b^{-1})),  -\tfrac{1}{4}(b-b^{-1}), \tfrac{1}{2}+\tfrac{1}{4}(b+b^{-1})) \rVert = \lVert  (0,0,a^{-1})k \rVert,\\
&\lVert ((b-b^{-1}), \tfrac{1}{2}(b+b^{-1}),  -\tfrac{1}{2}(b-b^{-1}))  \rVert = \lVert (0,1,-ca^{-1})k' \rVert.
\end{align*}

Since the action of $K_{\SO_3}$ preserve the Euclidean norm, we get 
\begin{align*}
&|a^{-1}| = \lVert (0, 0, a^{-1}) \rVert =  \lVert (\tfrac{1}{2}(\tfrac{1}{2}-\tfrac{1}{4}(b+b^{-1})),  -\tfrac{1}{4}(b-b^{-1}), \tfrac{1}{2}+\tfrac{1}{4}(b+b^{-1})) \rVert, \\
&|ca^{-1}| \leq \lVert (0, 1, -ca^{-1}) \rVert = \lVert  ((b-b^{-1}), \tfrac{1}{2}(b+b^{-1}),  -\tfrac{1}{2}(b-b^{-1})) \rVert.
\end{align*}

For $F=\mathbb{C}$, $K_{\SO_3}\cong \SO(3,\mathbb{R})$. Similar as in the real case, the action of $K_{\SO_3}$ preserve $\lVert \cdot \rVert$. We have
\begin{align*}
&|a^{-1}| = \lVert (0, 0, a^{-1}) \rVert^2 = \lVert (\tfrac{1}{2}(\tfrac{1}{2}-\tfrac{1}{4}(b+b^{-1})),  -\tfrac{1}{4}(b-b^{-1}), \tfrac{1}{2}+\tfrac{1}{4}(b+b^{-1})) \rVert^2, \\
&|ca^{-1}| \leq \lVert  ((b-b^{-1}), \tfrac{1}{2}(b+b^{-1}),  -\tfrac{1}{2}(b-b^{-1})) \rVert^2.
\end{align*}

Therefore in both cases, we have 
\begin{align*}
|a| &\ll \max(|b|, |b^{-1}|)^{-1} = \min(|b|, |b^{-1}|),\\
|ca^{-1}| &\ll  \max(|b|, |b^{-1}|).
\end{align*}

We denote $\lfloor b \rfloor=b$ if $|b|\leq 1$ and $\lfloor b \rfloor=b^{-1}$ otherwise. Since 
\[
\begin{psmatrix}
1 & c & -\frac{1}{2}c^2  \\
0 & 1 & -c \\
0 & 0 & 1 
\end{psmatrix} \begin{psmatrix}
a  & 0 & 0  \\
0 & 1 & 0 \\
0 & 0 & a^{-1}
\end{psmatrix} = \begin{psmatrix}
a  & 0 & 0  \\
0 & 1 & 0 \\
0 & 0 & a^{-1} 
\end{psmatrix} \begin{psmatrix}
1 & ca ^{-1} & -\frac{1}{2}c^2a^{-2}  \\
0 & 1 & -ca^{-1} \\
0 & 0 & 1 
\end{psmatrix}, 
\]
the local integral is majorized by
\begin{multline*}
\int_{F^{\times}} \int_{F^{\times}} \int_{F^{\times}} \int_{K_G} |\rho(k)f(a_1y_1,a_2y_2,b)| |a_1|^{\frac{d_1}{2}-2}|a_2|^{\frac{d_2}{2}-2} |a_1a_2|^{\Re(s)-\frac{n+1}{2}} |\lfloor b^{-1}a_1a_2^{-1} \rfloor|^{\Re(s)-\frac{n-3}{2}} \\ \times \int_{N^{\circ}(F)} |W_{\rho_{\tau,s}}|(w unk''\iota(k), a_y\mathrm{diag}(a_1a_2, \lfloor b^{-1}a_1a_2^{-1} \rfloor, I_{n-2}))dud^{\times}bd^{\times}a_1d^{\times}a_2dk.
\end{multline*}
This is 
\begin{multline*}
\int_{F^{\times}}\int_{F^{\times}}  \int_{F^{\times}} \int_{K_G} |\rho(k)f(a_1y_1,a_2y_2,b)|  |a_1|^{\frac{d_1}{2}-2}|a_2|^{\frac{d_2}{2}-2}|a_1a_2|^{\Re(s)-\frac{n+1}{2}}|\lfloor b^{-1}a_1a_2^{-1} \rfloor|^{\Re(s)-\frac{n-3}{2}} \\
\times\int_{w N^{\circ}(F)w^{-1}} |W_{\rho_{\tau,s}}|(u(w nw^{-1})w k''\iota(k),a_y\mathrm{diag}(a_1a_2, \lfloor b^{-1}a_1a_2^{-1} \rfloor, I_{n-2}))dud^{\times}bd^{\times}a_1d^{\times}a_2dk,
\end{multline*}
where 
\[
n=\begin{psmatrix}
I_{n-1}&&&& \\
&1 & ca^{-1} & -\frac{1}{2}c^2a^{-2}&  \\
&& 1 & -ca^{-1}& \\
&&&1& \\
&&&&I_{n-1}\\
\end{psmatrix}, \quad 
 k''= \begin{psmatrix}
I_{n-1} && \\ 
&k'& \\
&&  I_{n-1}
\end{psmatrix}.
\]

Since 
\[
w nw^{-1} = \begin{psmatrix}
1&&&&&& \\ &1&&-\frac{1}{2}ca^{-1}&&-\frac{1}{4}c^2a^{-2}& \\ &&I_{n-2}&&&& \\ &&&1&&\frac{1}{2}ca^{-1}& \\ &&&&I_{n-2}&& \\ &&&&&1& \\ &&&&&&1 
\end{psmatrix},
\]
by \cref{B1} and \cref{B2} in \cref{7.2}, the local integral is majorized by
\begin{align*}
&\sum_j c_s \int_{F^{\times}} \int_{F^{\times}}  \int_{F^{\times}} \int_{K_G} |\rho(k)f(a_1y_1,a_2y_2,b)| |a_1|^{\frac{d_1}{2}-2}|a_2|^{\frac{d_2}{2}-2}|a_1a_2|^{\Re(s)-\frac{n+1}{2}} \\
&\times |\lfloor b^{-1}a_1a_2^{-1} \rfloor|^{\Re(s)-\frac{n-3}{2}}|\chi_j|(a_y\mathrm{diag}(a_1a_2, \lfloor b^{-1}a_1a_2^{-1} \rfloor, I_{n-2}))(1+\lVert -\frac{1}{2}ca^{-1} \rVert^4)^C \\
& \times \left(\int_{w N^{\circ}(F)w^{-1}}(1+\lVert u\rVert)^{-\frac{\Re(s)}{2}-\frac{n-2}{4}+C}du\right)d^{\times}b d^{\times}a_1d^{\times}a_2dk.
\end{align*}

Since $|ca^{-1}|\ll |\lfloor b \rfloor|^{-1}$, the above sum is majorized by 
\begin{align*}
&\sum_j c_s \int_{F^{\times}} \int_{F^{\times}} \int_{F^{\times}} \int_{K_G} |\rho(k)f(a_1y_1,a_2y_2,b)| |a_1|^{\frac{d_1}{2}-2}|a_2|^{\frac{d_2}{2}-2}|a_1a_2|^{\Re(s)-\frac{n+1}{2}}\\
& \times |\lfloor b^{-1}a_1a_2^{-1} \rfloor|^{\Re(s)-\frac{n-3}{2}-C'}|\chi_1^j(-4Q'(y_2)a_1a_2)||\chi_2^j(\lfloor b^{-1}a_1a_2^{-1} \rfloor)|\\
& \times \left(\int_{w N^{\circ}(F)w^{-1}} (1+\lVert u\rVert)^{-\frac{\Re(s)}{2}-\frac{n-2}{4}+C}du\right) d^{\times}bd^{\times}a_1d^{\times}a_2dk.
\end{align*}
Here $\chi_1^j$, $\chi_2^j$ are positive quasi-characters, and $C'$ is some positive integer depends only on $\tau$. 

Thus for $\Re(s)$ large, our sum is majorized by
\begin{align*}
&\sum_j c_s \int_{F^{\times}} \int_{F^{\times}} \int_{F^{\times}} \int_{K_G} |\rho(k)f(a_1y_1,a_2y_2,b)| |a_1|^{\frac{d_1}{2}-2}|a_2|^{\frac{d_2}{2}-2}\\
\times & |a_1a_2|^{\Re(s)-\frac{n+1}{2}}|-4Q'(y_2)|^{c_j} |a_1a_2|^{c_j}d^{\times}bd^{\times}a_1d^{\times}a_2dk,
\end{align*}
where $c_j$ is some integer depends on $\tau$. This is majorized by
\begin{align*}
&\sum_j c_s  \int_{F^{\times}}\int_{F^{\times}} \int_{F^{\times}} \int_{K_G} |\rho(k)f(a_1y_1,a_2y_2,b)| |a_1|^{\frac{d_1}{2}-2}|a_2|^{\frac{d_2}{2}-2}\\
\times & |a_1a_2|^{\Re(s)-\frac{n+1}{2}}|a_1y_1|^{c_j} |a_2y_2|^{c_j}d^{\times}bd^{\times}a_1d^{\times}a_2dk.
\end{align*}

Then for $\Re(s)$ large, the local integral is majorized by a constant times a finite sum of integrals of the form 
\begin{align*}
&\int_{F^{\times}} \int_{F^{\times}} \int_{K_G} |\rho(k)\tilde{f}(a_1y_1,a_2y_2)| |a_1|^{\frac{d_1}{2}-2}|a_2|^{\frac{d_2}{2}-2}|a_1a_2|^{\Re(s)-\frac{n+1}{2}}d^{\times}a_1d^{\times}a_2dk\\
=&\int_{F^{\times}} \int_{F^{\times}} |f'(a_1y_1, a_2y_2)| |a_1|^{\Re(s)+\frac{d_1-n-1}{2}}|a_2|^{\Re(s)+\frac{d_2-n-1}{2}} d^{\times}a_1d^{\times}a_2.
\end{align*}
Here $\tilde{f} \in \S(V(F))$ and $f'(v)=\int_{K_G}\rho(k)\tilde{f}(v)$ for $v\in V(F)$. Then we deduce the lemma. 
\end{proof}

\section{Absolute Convergence} \label{8}
In this section, we handle the absolute convergence of the sum of the global integral in the main theorem (\cref{thm}), which will make the proof of \cref{thm} rigorous. The proof follows from the absolute convergence of the local integrals in the Archimedean case and the non-Archimedean case. 

\begin{lemma}
The sum of the global integral 
\[
\sum_{y\in \mathbb{P}Y'(F)}I(f,W_{\xi_s})(y)=\sum_{y\in \mathbb{P}Y'(F)}\int_{U_2(\A_F)\backslash G(\A_F)} \rho(g)f(y,1)\int_{N^{\circ}(\A_F)} W_{\xi_s}(w u\iota(g),a_y)\psi(u)dudg 
\]
converges absolutely for $\Re(s)$ large enough. 
\end{lemma}

\begin{proof}
Let $y=(y_1, y_2)\in \mathbb{P}Y'(F)$. Let $S$ be a finite set of places of $F$ that includes the infinite places and all the finite places such that $q_v < n$, $f^{S}=\one_{V(\hat{\OO}^S)\times \prod_{v\not\in S}\OO_v^{\times}}$, $\tau_v$ unramified, $|\mathcal{Q}'(y_2)|_v=1$  and $\rho_v(k)f_v=f_v$ for $k \in G(\OO_v)$ for $v \not \in S$.

Using the results and notations of \cref{5.3}, \cref{5.5} and \cref{7.3}, for $\Re(s)$ large we have 
\begin{align*}
I(f,W_{\xi_s})(y) &= \prod_{v|\infty}I_v(f_v,W_{\rho_{\tau,s}})(y)\prod_{v\in \mathrm{S-\infty}}I_v(f_v, W^{\circ}_{\rho_{\tau,s}})(y)\prod_{v\notin \mathrm{S}}I_v(f_v, W_{\rho_{\tau,s}})(y) \\
&\ll \prod_{v|\infty} \int_{F^{\times}_v} \int_{F^{\times}_v} |f'_v(a_1y_1, a_2y_2)| |a_1|_v^{\Re(s)+\frac{d_1-n-1}{2}-c_0}|a_2|_v^{\Re(s)+\frac{d_2-n-1}{2}-c_0} d^{\times}a_1d^{\times}a_2 \\
&\times \prod_{v\in \mathrm{S}-\infty}\int_{F^{\times}}\int_{F^{\times}} |f'_v(a_1y_1, a_2y_2)| |a_1|_v^{\Re(s)+\frac{d_1-n-1}{2}-c_0}|a_2|_v^{\Re(s)+\frac{d_2-n-1}{2}-c_0} d^{\times}a_1d^{\times}a_2\\
&\times \prod_{v\notin \mathrm{S}} \zeta_v(\Re(s)+c_{\tau}) \int_{F^{\times}} \int_{F^{\times}} |\one_{V(\OO_F)}(a_1y_1, a_2y_2)| |a_1|_v^{\Re(s)+\frac{d_1-n-1}{2}-c_0}\\
&\times |a_2|_v^{\Re(s)+\frac{d_2-n-1}{2}-c_0} d^{\times}a_1d^{\times}a_2\\
& = \zeta_F(\Re(s)+c_{\tau}) \int_{\A_{F}^{\times}} \int_{\A_{F}^{\times}}   |f'(a_1y_1,a_2y_2)| |a_1|^{\Re(s)+\frac{d_1-n-1}{2}-c_0}\\
& \times |a_2|^{\Re(s)+\frac{d_2-n-1}{2}-c_0} d^{\times}a_1d^{\times}a_2.
\end{align*}
Here $f' \in \S(V(\A_F))$ depends on $f$, and $c_0>0$ is a constant depends only on $\tau$. 

Thus the sum of the global integral is majorized by a finite sum of a sum of integrals of the form
\[
\sum_{y\in \mathbb{P}Y'(F)}  \zeta_F(\Re(s)+c_{\tau}) \int_{\A_{F}^{\times}} \int_{\A_{F}^{\times}}  |f'(a_1y_1,a_2y_2)| |a_1|^{\Re(s)+\frac{d_1-n-1}{2}-c_0}|a_2|^{\Re(s)+\frac{d_2-n-1}{2}-c_0} d^{\times}a_1d^{\times}a_2,
\]
which, when $\Re(s)$ large, is majorized by 
\begin{align*}
&\sum_{y\in \mathbb{P}V(F)}  \zeta_F(\Re(s)+c_{\tau}) \int_{\A_{F}^{\times}} \int_{\A_{F}^{\times}}  |f'(a_1y_1,a_2y_2)| |a_1|^{\Re(s)+\frac{d_1-n-1}{2}-c_0}|a_2|^{\Re(s)+\frac{d_2-n-1}{2}-c_0} d^{\times}a_1d^{\times}a_2 \\
= & \sum_{y\in V(F)}  \zeta_F(\Re(s)+c_{\tau}) \int_{F^{\times} \backslash \A_{F}^{\times}} \int_{F^{\times} \backslash \A_{F}^{\times}}  |f'(a_1y_1,a_2y_2)| |a_1|^{\Re(s)+\frac{d_1-n-1}{2}-c_0}|a_2|^{\Re(s)+\frac{d_2-n-1}{2}-c_0} \\
\times & d^{\times}a_1d^{\times}a_2.
\end{align*}
This is a product of mirabolic Eisenstein series which converge absolutely for $\Re(s)$ large enough (see \cite{JS}). 
\end{proof}

\section{Unramified computation} \label{4}

In this section, we give the computation for the unramified local factor of the global integral $I(f, W_{\rho_{\tau,s}})$. The results of this section shed light on the nature of the local integrals $I(f,W_{\rho_{\tau,s}}).$  

We assume all data are unramified, i.e. the local field $F$ is absolutely unramified, and the character $\psi$ is unramified. Let $\tau$ be an irreducible unramified generic representation of $\GL_n(F)$. We denote $K_G=G(\OO)$.  We assume that the matrices of $\mathcal{Q}$ and $\mathcal{Q'}$ are invertible and that the residual characteristic is not $2$. We also assume that $d_2 > d_1$.

Let $f\in \S(V(F)\times F^{\times})$ be $f(v,u)=\one_{V(\mathcal{O})\times \OO^{\times}}(v,u)$ for $v\in V(F), u\in F^{\times}$. Let $\rho$ denote the local Weil representation of $G(F)$. Then $\rho(k)f=f$ for $k \in K_G$. 

Let 
\begin{align} \label{local vector}
W^{\circ}_{\rho_{\tau,s}} \in \rho_{\tau,s}=\mathrm{Ind}_{Q_n(F)}^{H(F)}(\mathcal{W}(\tau, \psi_0) \otimes |\det |^{s-\tfrac{1}{2}})
\end{align}
be the unique spherical vector satisfying $W^{\circ}_{\rho_{\tau,s}}(1,1)=1$. 
  
The Satake parameter for $\rho_{\tau,s}$ is defined (up to a permutation) by
\[
t_{\rho_{\tau,s}}=\mathrm{diag}(\chi_{1,s}(\varpi), \dots, \chi_{n,s}(\varpi),\chi_{n,s}(\varpi)^{-1},\dots,\chi_{1,s}(\varpi)^{-1}) \in \mathrm{Sp}_{2n}(\C).
\]
Here each $\chi_i:F^\times \to \CC^\times$ is an unramified character and
$$
\chi_{i,s}:=\chi_{i}|\cdot|^{s-\tfrac{1}{2}}.
$$
We set
$$
\chi_s:=\chi_{1,s} \otimes \dots \otimes \chi_{n,s}:(F^\times)^n \lto \CC^\times.
$$
For an unramified character $\mu$ of split $\SO_2(F)$, the Satake parameter is 
\[
t_{\mu}=\mathrm{diag}(\mu(\varpi), \mu(\varpi)^{-1}) \in \SO_2(\C).
\]

We then have the local integral
\begin{align} \label{basicfunction}
I_s(y):=I(f, W^{\circ}_{\rho_{\tau,s}})(y)=\int_{U_2(F)\backslash G(F)} \rho(g)f(y,1)\int_{N^{\circ}(F)} W^{\circ}_{\rho_{\tau,s}}(w u\iota(g),a_y)\psi_1(u)dudg
\end{align}
as the unramified local component of the global integral $I(f,W_{\rho_{\tau,s}})(y)$ by construction. We assume $y$ is integral.
To compute the local integral $I_s$ we will adapt a procedure in \cite{Kap}.  This requires us to write the function $\rho(g)f(y,1)$ in terms of functions lying in an appropriate Whittaker model.  

Let 
\begin{align} \label{c_k}
\alpha(y_1,y_2) = \prod_{i=1}^2(2+\sum_{k'=1}^{\mathrm{val}(y_i)}q^{k'}(1-(q-1)\sum_{k=1}^{\mathrm{val}(y_i)}q^{(\frac{d_i}{2}-2)k}))^{-1}.
\end{align}
Let $H_{1,y}, H_{2,y} \in C^{\infty}(T_G(F))$ be such that
\begin{align*}
& H_{1,y}\left( \begin{psmatrix} a_1 & \\ & a_1^{-1} \end{psmatrix},\begin{psmatrix} a_2 & \\ & a_2^{-1} \end{psmatrix}\right) =|a_1|^{\frac{d_1}{2}-2}|a_2|^{\frac{d_2}{2}-2}, \\
& H_{1,y}\left( \begin{psmatrix} 1 & \\ & d \end{psmatrix},\begin{psmatrix} 1 & \\ & d^{-1} \end{psmatrix}\right) = \alpha(y_1, y_2), 
\end{align*}
\begin{align*}
H_{2,y} \left( \begin{psmatrix} a_1 & \\ & a_1^{-1} \end{psmatrix},\begin{psmatrix} a_2 & \\ & a_2^{-1} \end{psmatrix}\right) &= (\one_{\varpi^{-\mathrm{val}(y_1)}\mathcal{O}_{F}}(a_1) - \sum_{k=1}^{\infty} q^{(\frac{d_1}{2}-2)k}(q-1)\one_{\varpi^{k-\mathrm{val}(y_1)}  \mathcal{O}_{F}}(a_1))\\
&\times (\one_{\varpi^{-\mathrm{val}(y_2)}\mathcal{O}_{F}}(a_2) - \sum_{k=1}^{\infty} q^{(\frac{d_2}{2}-2)k}(q-1)\one_{\varpi^{k-\mathrm{val}(y_2)}  \mathcal{O}_{F}}(a_2)), \\
H_{2,y}\left( \begin{psmatrix} 1 & \\ & d \end{psmatrix},\begin{psmatrix} 1 & \\ & d^{-1} \end{psmatrix}\right) &=  \one_{\OO^{\times}} (d).
\end{align*}
Here $a_i, d\in F^{\times}$.

Using the Iwasawa decomposition with respect to the Borel subgroup of $G(F)$ consisting of lower triangular matrices in $G(F)$, we define functions $\Phi_{1,y}, \Phi_{2,y} \in C^{\infty}(\overline{U_2}(F) \backslash G(F) / K_G )$ (here $\overline{U_2}(F)$ is the unipotent radical of the lower Borel subgroup of $G(F)$) by 
\[
\Phi_{1,y}(g): G(F) \to \C
\]
\[
\left( \begin{psmatrix}
1 & 0 \\
t_1 & 1
\end{psmatrix}
\begin{psmatrix} a & \\ & d \end{psmatrix} \kappa_1
, \begin{psmatrix}
1 & 0 \\
t_2 & 1
\end{psmatrix}
\begin{psmatrix}
a' & 0 \\
0 & (ada')^{-1}
\end{psmatrix}\kappa_2  \right) 
\mapsto H_{1,y}\left( \begin{psmatrix} a & \\ & d \end{psmatrix} , \begin{psmatrix}
a' & 0 \\
0 & (ada')^{-1}
\end{psmatrix} \right), 
\]
and
\[
\Phi_{2,y}(g): G(F) \to \C
\]
\[
\left(  \begin{psmatrix}
1 & 0 \\
t_1 & 1
\end{psmatrix} 
\begin{psmatrix} a & \\ & d \end{psmatrix}\kappa_1
, \begin{psmatrix}
1 & 0 \\
t_2 & 1
\end{psmatrix}
\begin{psmatrix}
a' & 0 \\
0 & (ada')^{-1}
\end{psmatrix} \kappa_2 \right)
\mapsto H_{2,y}\left( \begin{psmatrix} a & \\ & d \end{psmatrix} , \begin{psmatrix}
a' & 0 \\
0 & (ada')^{-1}
\end{psmatrix} \right).
\]
Here $a,a',d \in F^\times$, $t_i \in F$, and $(\kappa_1,\kappa_2) \in K_G$. 

We let $\Phi_{y}=\Phi_{1,y}\Phi_{2,y}$, and $H_{y}=H_{1,y}H_{2,y}$. We first show that an appropriate integral of $\Phi_{y}$ is $\rho(g)f(y,1)$ in the following lemma:
\begin{lemma}{\label{4.1}}
We have 
\begin{align} \label{lem4.1}
\rho(g)f(y,1)=\int_{U_2(F)} \Phi_{y}(ng) \psi_{U_2,Q}(n)^{-1}dn,
\end{align}
where $\psi_{U_2, Q}(n)=\psi(n_1\mathcal{Q}(y_1)+n_2\mathcal{Q'}(y_2))$ for $n=\left( \begin{psmatrix}
1 & n_1 \\
0 & 1
\end{psmatrix} \right), \left( \begin{psmatrix}
1 & n_2 \\
0 & 1
\end{psmatrix} \right)\in U_2(F)$. 

\end{lemma}

\begin{proof}
As a function of $g$, both sides of the equality are invariant under $K_G$ on the right and both transform via the same character under $U_2(F)$ on the left.  Thus it suffices to verify the equality \cref{lem4.1} for 
\begin{align*}
g_1 &=\left( \begin{psmatrix} a_1 & \\ & a_1^{-1} \end{psmatrix},\begin{psmatrix} a_2 & \\ & a_2^{-1} \end{psmatrix}\right), \\
g_2 &= \left( \begin{psmatrix} 1 & \\ & d \end{psmatrix},\begin{psmatrix} 1 & \\ & d^{-1} \end{psmatrix}\right). 
\end{align*}

We have 
\begin{align*}
\rho\left( \begin{psmatrix} a_1 & \\ & a_1^{-1} \end{psmatrix},\begin{psmatrix} a_2 & \\ & a_2^{-1} \end{psmatrix}\right)f(y,1) &=|a_1|^{\frac{d_1}{2}}|a_2|^{\frac{d_2}{2}}\one_{V(\OO)}(a_1y_1,a_2y_2)\\&=|a_1|^{\frac{d_1}{2}}|a_2|^{\frac{d_2}{2}}\one_{\varpi^{-\mathrm{val}(y_1)}\mathcal{O}}(a_1)\one_{\varpi^{-\mathrm{val}(y_2)}\OO}(a_2), \\
\rho\left( \begin{psmatrix} 1 & \\ & d \end{psmatrix},\begin{psmatrix} 1 & \\ & d^{-1} \end{psmatrix}\right)f(y,1) &= \one_{\OO^{\times}} (d)\one_{\OO^{\times}} (d^{-1})|d|^{\frac{d_2-d_1}{4}}\\&=\one_{\OO^{\times}} (d).
\end{align*}

On the other hand, the right hand side of \cref{lem4.1} in the two cases are
\begin{align*}
&\int_{F^2} \Phi_{y}\left(\begin{psmatrix}1 & n_1 \\ & 1\end{psmatrix}\begin{psmatrix} a_1 & \\ & a_1^{-1} \end{psmatrix}, \begin{psmatrix}1 & n_2 \\ & 1\end{psmatrix}\begin{psmatrix} a_1 & \\ & a_1^{-1} \end{psmatrix} \right)\psi(n_1\mathcal{Q}(y_1)+n_2\mathcal{Q'}(y_2)) dn_1dn_2, \\
&\int_{F^2} \Phi_{y}\left( \begin{psmatrix}1 & n_1 \\ & 1\end{psmatrix} \begin{psmatrix} 1 & \\ & d \end{psmatrix}, \begin{psmatrix}1 & n_2 \\ & 1\end{psmatrix} \begin{psmatrix} 1 & \\ & d^{-1} \end{psmatrix}\right) \psi(n_1\mathcal{Q}(y_1)+n_2\mathcal{Q'}(y_2)) dn_1dn_2.  
\end{align*}
Let $\Phi_{y}(g_1,g_2)=\Phi_{y}'(g_1)\Phi_{y}''(g_2)$ where $(g_1,g_2)\in G(F)$, and 
$$
\Phi_{y}', \Phi_{y}'' \in C^{\infty}(\overline{U_2}(F) \backslash \SL_2(F) / \SL_2(\OO_F)).
$$
We denote accordingly for $H_y(t_1,t_2)=H'_y(t_1)H''_y(t_2)$ where $(t_1,t_2) \in T_G(F)$.

By symmetry, it suffices to verify the equalities

\begin{align*}
&\one_{\varpi^{-\mathrm{val}(y_1)}\mathcal{O}_{F}}(a)|a|^{\frac{d_1}{2}} = \int_{F} \Phi_{y}'\left(\begin{psmatrix}1 & n \\ & 1\end{psmatrix}\begin{psmatrix} a & \\ & a^{-1} \end{psmatrix}\right)\psi(n\mathcal{Q}(y_1)) dt, \\
&\one_{\OO_F^{\times}} (d) = \int_{F} \Phi_{y}'\left(\begin{psmatrix}1 & n \\ & 1\end{psmatrix}\begin{psmatrix} 1 & \\ & d \end{psmatrix}\right)\psi(n\mathcal{Q}(y_1)) dt
\end{align*}
for $a,d\in F^{\times}$. 

Applying the Iwasawa decompostion with respect to the lower Borel subgroup of $\SL_2(F)$ to $\begin{psmatrix}1 & n \\ & 1\end{psmatrix}\begin{psmatrix} a & \\ & a^{-1} \end{psmatrix}$ and $\begin{psmatrix}1 & n \\ & 1\end{psmatrix}\begin{psmatrix} 1 & \\ & d \end{psmatrix}$, we have 
\[
\begin{psmatrix}1 & n \\ & 1\end{psmatrix}\begin{psmatrix} a & \\ & a^{-1} \end{psmatrix} = \begin{cases} \begin{psmatrix}
a & \\ & a^{-1} 
\end{psmatrix} \begin{psmatrix}
1 & a^{-2}n & \\ 0 & 1
\end{psmatrix} &\textrm{for} \quad |a^{-2}n|\leq 1
\\
\begin{psmatrix}
1 &  \\ n^{-1} & 1 
\end{psmatrix} \begin{psmatrix}
a^{-1}n & \\ & an^{-1}
\end{psmatrix}
\begin{psmatrix}
a^2n^{-1} & 1 \\ -1 & 0
\end{psmatrix} & \textrm{for} \quad |a^{-2}n| > 1
\end{cases},
\]
\[
\begin{psmatrix}1 & n \\ & 1\end{psmatrix}\begin{psmatrix} 1 & \\ & d \end{psmatrix} = \begin{cases} \begin{psmatrix}
1 & \\ & d 
\end{psmatrix} \begin{psmatrix}
1 & dn & \\ 0 & 1
\end{psmatrix}  &\textrm{for} \quad |dn|\leq 1
\\
\begin{psmatrix}
1 &  \\ n^{-1} & 1 
\end{psmatrix} \begin{psmatrix}
dn & \\ & n^{-1}
\end{psmatrix}
\begin{psmatrix}
(dn)^{-1} & 1 \\ -1 & 0
\end{psmatrix} & \textrm{for} \quad |dn| > 1
\end{cases}.
\]

Then it suffices to verify the equalities
\begin{align} \begin{split}
\one_{\varpi^{-\mathrm{val}(y_1)}\mathcal{O}_{F}}(a)|a|^{\frac{d_1}{2}} &= \int_{|n| \leq |a|^{2}} H'_{y}\begin{psmatrix}
a & \\ & a^{-1} 
\end{psmatrix} \psi(n\mathcal{Q}(y_1))^{-1} dn \\
& + \int_{|a|^2 < |n| } H'_{y}\begin{psmatrix}
a^{-1}n & \\ & an^{-1}
\end{psmatrix} \psi(n\mathcal{Q}(y_1))^{-1} dn \label{eqa1} 
\end{split},
\end{align}
\begin{align}
\begin{split}
\one_{\OO_F^{\times}} (d) & = \int_{|n| \leq |d|^{-1}} H'_{y}\begin{psmatrix}
1 & \\ & d 
\end{psmatrix} \psi(n\mathcal{Q}(y_1))^{-1} dn \\
& + \int_{|d|^{-1} < |n| } H'_{y}\begin{psmatrix}
dn & \\ & n^{-1}
\end{psmatrix} \psi(n\mathcal{Q}(y_1))^{-1} dn  \label{eqa2}
\end{split}.
\end{align}

We first verify \cref{eqa1}.

For $|a|>q^{\mathrm{val}(y_1)}$, $H'_{y}\begin{psmatrix}
a & \\ & a^{-1} 
\end{psmatrix}=0$ and $H'_{y}\begin{psmatrix}
a^{-1}n & \\ & an^{-1}
\end{psmatrix}=0$ when $|a|^2<n$, thus the right hand side of \cref{eqa1} is $0$, while the left hand side is also $0$. In particular, we have the equality for $|a|>q^{\mathrm{val}(y_1)}$.

For $|a|=q^{\mathrm{val}(y_1)}$, $H'_{y}\begin{psmatrix}
a^{-1}n & \\ & an^{-1}
\end{psmatrix}=0$ when $|a|^2<n$, thus we have that the right hand side of \cref{eqa1}
is
\begin{align*}
\int_{|n|\leq |a|^2} H'_{y}\begin{psmatrix}
a & \\ & a^{-1} 
\end{psmatrix} \psi(n\mathcal{Q}(y_1))^{-1} dn &= H'_{y}\begin{psmatrix}
a & \\ & a^{-1} 
\end{psmatrix}\int_{|n|\leq q^{2\mathrm{val}(y_1)}}\psi(n\mathcal{Q}(y_1))^{-1} dn \\
&= q^{2\mathrm{val}(y_1)}H'_{y}\begin{psmatrix}
a & \\ & a^{-1} 
\end{psmatrix}\int_{\OO_F} \psi(n\varpi^{-2\mathrm{val}(y_1)}\mathcal{Q}(y_1))^{-1} dn\\
&= q^{2\mathrm{val}(y_1)}H'_{y}\begin{psmatrix}
a & \\ & a^{-1} 
\end{psmatrix} \\
&= q^{2\mathrm{val}(y_1)}|a|^{\frac{d_1}{2}-2}\one_{\varpi^{-\mathrm{val}(y_1)}\mathcal{O}_{F}}(a) \\
&= \one_{\varpi^{-\mathrm{val}(y_1)}\mathcal{O}_{F}}(a)|a|^{\frac{d_1}{2}} = q^{\frac{d_1}{2}\mathrm{val}(y_1)},
\end{align*}
so \cref{eqa1} is valid in this case.  

For $|a|<q^{\mathrm{val}(y_1)}$, we have that the left hand side of \cref{eqa1} is 
\begin{align*}
&  H'_{y}\begin{psmatrix}
a & \\ & a^{-1} 
\end{psmatrix} \int_{|n| \leq |a|^2} \psi(n\mathcal{Q}(y_1))^{-1} dn + \int_{|a|^2 < |n| } H'_{y}\begin{psmatrix}
a^{-1}n & \\ & an^{-1}
\end{psmatrix} \psi(n\mathcal{Q}(y_1))^{-1} dn\\ &= q^{-2\mathrm{val}(a)}H'_{y}\begin{psmatrix}
a & \\ & a^{-1} 
\end{psmatrix}+ \int_{q^{-\mathrm{val}(a)} < |n| \leq q^{\mathrm{val}(y_1)}} H'_{y}\begin{psmatrix}
n & \\ & n^{-1}
\end{psmatrix} \psi(mn\mathcal{Q}(y_1))^{-1} q^{-\mathrm{val}(a)}dn \\
&= q^{-2\mathrm{val}(a)}H'_{y}\begin{psmatrix}
a & \\ & a^{-1} 
\end{psmatrix} + \sum_{k=0}^{\mathrm{val}(y_1)+\mathrm{val}(a)-1}q^{-2\mathrm{val}(a)+k}(q-1)H'_{y}\begin{psmatrix}
\varpi^{\mathrm{val}(a)-1+k} & \\ & \varpi^{-\mathrm{val}(a)+1-k}
\end{psmatrix}.
\end{align*}
On the other hand the right hand side is $ q^{-\frac{d_1}{2}\mathrm{val}(a)}$. One can show by induction on $\mathrm{val}(a)$, for $\mathrm{val}(a)>\mathrm{val} (y_1)$, that these are equal.  

For \cref{eqa2}, the right hand side is 
\begin{align*}
&\int_{|n| \leq |d|^{-1}} H'_{y}\begin{psmatrix}
1 & \\ & d 
\end{psmatrix} \psi(n\mathcal{Q}(y_1))^{-1} dn + \int_{|d|^{-1} < |n| } H'_{y}\begin{psmatrix}
dn & \\ & n^{-1}
\end{psmatrix} \psi(n\mathcal{Q}(y_1))^{-1} dn \\
=& H'_{y}\begin{psmatrix}
1 & \\ & d 
\end{psmatrix} \left(\int_{|n| \leq |d|^{-1}} \psi(n\mathcal{Q}(y_1))^{-1} dn + \int_{|d|^{-1} < |n| } H'_{y}\begin{psmatrix}
dn & \\ & (dn)^{-1}
\end{psmatrix} \psi(n\mathcal{Q}(y_1))^{-1}dn\right).
\end{align*}

When $d\in \OO^{\times}$, the above expression is
\begin{align*}
&\frac{1+\int_{|dn|>1} H'_{y}\begin{psmatrix}
dn & \\ & (dn)^{-1}
\end{psmatrix} \psi(n\mathcal{Q}(y_1))^{-1}dn)}{2+\sum_{k'=1}^{\mathrm{val}(y_1)}q^{k'}(1-(q-1)\sum_{k=1}^{\mathrm{val}(y_1)}q^{(\frac{d_1}{2}-2)k})} = 1 
\end{align*}
This is equal to the left-hand side of \cref{eqa2}.

Finally, when $d\not\in \OO^{\times}$, both the left-hand side and right-hand side of \cref{eqa2} are $0$. This exhausts the cases and proves the lemma.
\end{proof}

Inserting the result of Lemma \ref{4.1} to the local integral $I$, we obtain
\[
I_s(y)=\int_{U_2(F)\backslash G(F)} \int_{U_2(F)} \Phi_{y}(ng)\psi_{U_2, Q}(n)^{-1}dn  \int_{N^{\circ}(F)} W^{\circ}_{\rho_{\tau,s}}(w u\iota(g),a_y)\psi_1(u)dudg.
\]

Then by collapsing the integrals we get
\begin{align} \label{eq1}
I_s(y)=\int_{G(F)}  \Phi_{y}(g) \int_{N^{\circ}(F)} W^{\circ}_{\rho_{\tau,s}}(w u\iota(g),a_y)\psi_1(u)dudg. 
\end{align}
We justify this step by showing the above integral converges absolutely for $\Re(s)$ large in \cref{6.2} below.

Since $\Phi_{y} \in C^{\infty}(\overline{U_2}(F) \backslash G(F) / K_G )$, applying the Iwasawa decomposition with respect to the lower Borel subgroup $\overline{B}_G(F)$ of $G(F)$ we get 
\begin{align*} 
I_s(y) &=  \int_{T_G(F)} \Phi_{y}(g) \delta_{\overline{B}_G}^{-1}(g) \int_{\overline{U_2}(F)} \int_{N^{\circ}(F)}   W^{\circ}_{\rho_{\tau,s}}(w u\iota(\bar{n}g),a_y) \psi_1(u) dud\bar{n}dg\\
&= \int_{A_1(F)} \int_{M_1(F)} \int_{G_1(F)}  \Phi_{y}(amx) \delta_{\overline{B}_G}^{-1}(amx) \\
& \times \int_{\overline{U_2}(F)}  \int_{N^{\circ}(F)} W^{\circ}_{\rho_{\tau,s}}(w u\iota(\bar{n}amx),a_y) \psi_1(u) dud\bar{n}dadmdx \\
&= \int_{F^{\times}} \int_{F^{\times}} \int_{F^{\times}} \Phi_{y}(\begin{psmatrix}
a &  \\  & a^{-1}(mb)^{-1}
\end{psmatrix},\begin{psmatrix}
m &\\&b
\end{psmatrix}) |am|^{2}\\
& \times \int_{\overline{U_2}(F)}  \int_{N^{\circ}(F)} W^{\circ}_{\rho_{\tau,s}} \left(w u\iota(\bar{n}) \iota_2\begin{psmatrix}
am&&&\\&ab&&\\&&(ab)^{-1}&\\&&&(am)^{-1}
\end{psmatrix},a_y\right) \psi_1(u) dud\bar{n}d^{\times}ad^{\times}md^{\times}b\\
&= \int_{F^{\times}} H'_y(\begin{psmatrix}
a_1 &\\&a_1^{-1}
\end{psmatrix})|a_1|^{2} \int_{F^{\times}} \int_{F^{\times}} H_{1,y}(\begin{psmatrix}
1& \\&(mb)^{-1}
\end{psmatrix},\begin{psmatrix}
m& \\ & b
\end{psmatrix}) H_{2,y}(\begin{psmatrix}
1& \\&(mb)^{-1}
\end{psmatrix},\begin{psmatrix}
m& \\ & b
\end{psmatrix})|m|^{2}\\
& \times \int_{\overline{U_2}(F)}  \int_{N^{\circ}(F)} W^{\circ}_{\rho_{\tau,s}} \left(w u\iota(\bar{n}) \iota_2\begin{psmatrix}
a_1m&&&\\&a_1b&&\\&&(a_1b)^{-1}&\\&&&(a_1m)^{-1}
\end{psmatrix},a_y\right) \psi_1(u) dud\bar{n}d^{\times}a_1d^{\times}md^{\times}b.
\end{align*}

Let $I_F=[-\frac{2\pi i}{\mathrm{log}q}, \frac{2\pi i}{\mathrm{log}q}]$, $\sigma_1,\sigma_2 \in \mathbb{R}$, and
\begin{align} \label{c_q}
c_q=(\frac{\log q}{2\pi i})^2. 
\end{align}
Let 
\[
\chi_{s_1,s_2}\left(\begin{psmatrix} 1 & \\ & a_1 \end{psmatrix},\begin{psmatrix} a_2 & \\ & (a_1a_2)^{-1} \end{psmatrix} \right) = |a_1|^{s_1}|a_2|^{s_2}
\]
for $a_i\in F^\times $,
and
\[
H_{y,s_2}(1) = \int_{F^{\times2}} H_{2,y}\left( \begin{psmatrix} 1 & \\ & a_1 \end{psmatrix},\begin{psmatrix} a_2 & \\ & (a_1a_2)^{-1} \end{psmatrix}\right)\chi_{s_1,s_2}^{-1}\left( \begin{psmatrix} 1 & \\ & a_1 \end{psmatrix},\begin{psmatrix} a_2 & \\ & (a_1a_2)^{-1} \end{psmatrix}\right)d^{\times}a_1d^{\times}a_2.
\]

By Mellin inversion, for $\left(\begin{psmatrix}
1& \\& (mb)^{-1}
\end{psmatrix},\begin{psmatrix}
m& \\ & b
\end{psmatrix}\right) \in M_1(F)G_1(F)$, we have 
\[
H_{2,y}\left(\begin{psmatrix}
1& \\& (mb)^{-1}
\end{psmatrix},\begin{psmatrix}
m& \\ & b
\end{psmatrix}\right) = c_q \int_{iI_F +\sigma_1}\int_{iI_F+\sigma_2} H_{y,s_2}(1)\chi_{s_1,s_2}\left(\begin{psmatrix}
1& \\&m^{-1}b
\end{psmatrix},\begin{psmatrix}
m& \\ & b^{-1}
\end{psmatrix}\right) ds_1ds_2.
\]

Let 
\begin{align} \label{c_s_2}
c(q^{s_2},y_i) = 1-q^{s_2}+(q-1)q^{-\mathrm{val}(y_i)s_2}
\end{align}
for $i=1,2$.

We have 
\begin{align*}
& H_{y,s_2}(1) \\
&= \int_{F^{\times2}} H_{2,y}\left( \begin{psmatrix} 1 & \\ & a_1 \end{psmatrix},\begin{psmatrix} a_2 & \\ & (a_1a_2)^{-1} \end{psmatrix}\right)|a_1|^{-s_1}|a_2|^{-s_2} d^{\times}a_1d^{\times}a_2\\
&= \int_{F^{\times2}} \one_{\OO_F^{\times}}(a_1) \left(\one_{\varpi^{-\mathrm{val}(y_2)}\mathcal{O}_{F}}(a_2) - \sum_{k=1}^{\infty} q^{(\frac{d_2}{2}-2)k}(q-1) \one_{\varpi^{k-\mathrm{val}(y_2)}  \mathcal{O}_{F}}(a_2)\right) \\
& \times |a_1|^{-s_1} |a_2|^{-s_2} d^{\times}a_1d^{\times}a_2 \\
&= \int_{F^{\times}} \one_{\varpi^{-\mathrm{val}(y_2)}\mathcal{O}_{F}}(a_2)|a_2|^{-s_2} d^{\times}a_2 + (q-1) \sum_{k=1}^{\infty} q^{(\frac{d_2}{2}-2)k}\int_{F^{\times}} \one_{\varpi^{k-\mathrm{val}(y_2)}  \mathcal{O}_{F}}(a_2)|a_2|^{-s_2} d^{\times}a_2 \\
&= \sum_{i=-\mathrm{val}(y_2)}^{\infty}q^{is_2}+(q-1)\sum_{k=1}^{\infty}(\sum_{i=k-\mathrm{val}(y_2)}^{\infty}q^{is_2}) \\
&= c(q^{s_2},y_2)\zeta_v(-s_2)^2.
\end{align*}

We denote $\chi_{s_1,s_2}'= \chi_{s_1,s_2} H_{1,y} \delta_{\overline{B}_G}^{-1}$. 
Let
\begin{align*}
I_{s, s_1,s_2}(y) &:= \int_{F^{\times}} \int_{F^{\times}}  |m|^{-s_1+s_2+\frac{d_2}{2}}|b|^{-s_1} \\
&\times \int_{\overline{U_2}(F)}  \int_{N^{\circ}(F)} W^{\circ}_{\rho_{\tau,s}} \left(w u\iota(\bar{n}) \iota_2\begin{psmatrix}
m&&&\\&b&&\\&&b^{-1}&\\&&&m^{-1}
\end{psmatrix},a_y\right) \psi_1(u) dud\bar{n}d^{\times}md^{\times}b.
\end{align*}
Then 
\begin{align*}
I_s(y) &= c_q \int_{F^{\times}} H'_y\begin{psmatrix}
a_1 &\\&a_1^{-1}
\end{psmatrix}|a_1|^{-2}  \int_{F^{\times}} \int_{F^{\times}} \int_{iI_F +\sigma_1}\int_{iI_F+\sigma_2}H_{y,s_2}(1)\chi_{s_1,s_2}'\left(\begin{psmatrix}
1& \\&(mb)^{-1}
\end{psmatrix},\begin{psmatrix}
m& \\ & b
\end{psmatrix}\right)\\
& \times \int_{\overline{U_2}(F)}  \int_{N^{\circ}(F)} W^{\circ}_{\rho_{\tau,s}}\left(w u\iota(\bar{n}) \iota_2\begin{psmatrix}
a_1m&&&\\&a_1b&&\\&&(a_1b)^{-1}&\\&&&(a_1m)^{-1}
\end{psmatrix},a_y\right) \\
& \times \psi_1(u) dud\bar{n}d^{\times}a_1d^{\times}md^{\times}bds_1ds_2 \\
&= c_q\alpha(y_1, y_2) \int_{iI_F +\sigma_1}\int_{iI_F+\sigma_2} H_{y,s_2}(1) \int_{F^{\times}} H'_y\begin{psmatrix}
a_1 &\\&a_1^{-1}
\end{psmatrix}|a_1|^{-s_2} d^{\times}a_1 \int_{F^{\times}} \int_{F^{\times}}  |m|^{-s_1+s_2+\frac{d_2}{2}} \\
& \times  |b|^{-s_1} \int_{\overline{U_2}(F)}  \int_{N^{\circ}(F)}W^{\circ}_{\rho_{\tau,s}} \left(w u\iota(\bar{n}) \iota_2\begin{psmatrix}
m&&&\\&b&&\\&&b^{-1}&\\&&&m^{-1}
\end{psmatrix},a_y\right) \psi_1(u) dud\bar{n}d^{\times}md^{\times}bds_1ds_2\\
&= c_q\alpha(y_1, y_2) \int_{iI_F +\sigma_1}\int_{iI_F+\sigma_2} H_{y,s_2}(1) \int_{F^{\times}} H'_y\begin{psmatrix}
a_1 &\\&a_1^{-1}
\end{psmatrix}|a_1|^{-s_2} d^{\times}a_1 I_{s, s_1,s_2}(y)ds_1ds_2.
\end{align*}
We show in \cref{6.3} that $I_{s, s_1,s_2}$ converges when $\Re(s_1), \Re(-s_2)$ large and $s$ in some vertical strip in the right-half plane depends on $\Re(s_1),\Re(-s_2)$. 

We have now written $I_s(y)$ in terms of Whittaker functions as mentioned below \cref{basicfunction}. We now make use of the local functional equation for the local Rankin-Selberg integral on $\GL_1 \times \GL_n$ (see \cite[Theorem 2.7]{JPS}) to simplify $I_s(y).$  This requires the introdution of certain dual local integrals. Let
\[
r(a) = \begin{psmatrix}
0&1&0\\ 0&0&I_{n-2} \\ 1&0&a
\end{psmatrix} \in \GL_n(F). 
\]
We construct the dual integral 
\begin{align*}
I_s'(y) = \int_{T_G(F)} \Phi_{y}(g) \delta_{\overline{B}_G}^{-1}(g) \int_{\overline{U_2}(F)}  \int_{N^{\circ}(F)} \int_{M_{1\times (n-2)}(F)} W^{\circ}_{\rho_{\tau,s}}(w u\iota(\bar{n}g), r(a)a_y) \psi_1(u)dadud\bar{n}dg.
\end{align*}

Applying the same process as for $I_s$, we have that 
\begin{align*}
I_s'(y) &=  c_q\alpha(y_1, y_2) \int_{iI_F +\sigma_1}\int_{iI_F+\sigma_2} \int_{F^{\times}} H'_y\begin{psmatrix}
a_1 &\\&a_1^{-1}
\end{psmatrix}|a_1|^{-s_2} d^{\times}a_1 \int_{F^{\times}} \int_{F^{\times}}  |m|^{-s_1+s_2+\frac{d_2}{2}}|b|^{-s_1}   \\
& \times  \int_{\overline{U_2}(F)}  \int_{N^{\circ}(F)} \int_{M_{1\times (n-2)}(F)} W^{\circ}_{\rho_{\tau,s}} \left(w u\iota(\bar{n}) \iota_2\begin{psmatrix}
m&&&\\&b&&\\&&b^{-1}&\\&&&m^{-1}
\end{psmatrix},r(a)a_y\right) \\
& \times \psi_1(u) dadud\bar{n}d^{\times}md^{\times}bds_1ds_2.
\end{align*}

We denote
\begin{align*}
I'_{s, s_1,s_2}(y) & =\int_{F^{\times}} \int_{F^{\times}}  |m|^{-s_1+s_2+\frac{d_2}{2}}|b|^{-s_1}   \\
& \times \int_{\overline{U_2}(F)}  \int_{N^{\circ}(F)}  \int_{M_{1\times (n-2)}(F)} W^{\circ}_{\rho_{\tau,s}} \left(w u\iota(\bar{n}) \iota_2\begin{psmatrix}
m&&&\\&b&&\\&&b^{-1}&\\&&&m^{-1}
\end{psmatrix},r(a)a_y\right)\\
& \times \psi_1(u) dadud\bar{n}d^{\times}md^{\times}b.
\end{align*}

\begin{remark}
We show in \cref{6.4} that $I'_{s, s_1,s_2}$ converges when 
\begin{align}
\Re(s_1)+\Re(-s_2) \gg 1 \textrm{ and }
A(s_1,s_2) \geq \Re(s) \geq B(s_1,s_2)
\end{align}
for some $A(s_1,s_2)>B(s_1,s_2).$ The convergence region of $I'_{s, s_1,s_2}$ intersects the convergence region of $I_{s, s_1,s_2}$ when $\Re(-s_2)$ is in some vertical strip. 
\end{remark}

We now relate $I'_{s,s_1,s_2}$ to our integral $I_{s, s_1,s_2}$.

\begin{lemma}\label{4.2}
In the region of convergence of $I_{s, s_1,s_2}$ (see \cref{6.3}), we have 
\[
I_{s, s_1,s_2}(y) = \gamma(s-s_1+s_2+\frac{d_2+1}{2}, \tau)^{-1}I'_{s, s_1,s_2}(y).
\]
Here $\gamma(s-s_1+s_2+\frac{d_2+1}{2}, \tau)$ is the $\GL_1\times \GL_n$ gamma factor defined in \cite[Theorem 2.7]{JPS}.
\end{lemma}

\begin{proof}
We let $\chi_1, \chi_2$ be characters on $F^{\times}$ such that $\chi_1=|\cdot|^{-s_1+s_2+\frac{d_2}{2}}$, $\chi_2=|\cdot|^{-s_1}$. 

Let $\mu$ be a character of $\SO_2(F)$, and let $\pi_{\zeta}=\mathrm{Ind}_{\overline{B}_{G'}(F)}^{G'(F)}(\chi \otimes \mu)$. Let $\varphi_{\zeta}(g,m,I_2) \in V_{\pi_{\zeta}}$.
Let 
\begin{multline*}
I_1 = \int_{\GL_1(F) \backslash G'(F)} \int_{N^{\circ}(F)} \int_{\GL_1(F)} \varphi_{\zeta}(g,m,I_2)\\
\times W^{\circ}_{\rho_{\tau,s}}(w u\iota_2(g),\mathrm{diag}(m,I_{n-1}))|\mathrm{det}m|^{s-\zeta-\frac{n-1}{2}}\psi_1(u)dmdudg,
\end{multline*}
and
\begin{multline*}
I_2 = \int_{\GL_1(F) \backslash G'(F)} \int_{N^{\circ}(F)} \int_{\GL_1(F)} \varphi_{\zeta}(g,m,I_2) \\
\times \int_{M_{1\times (n-2)}(F)} W^{\circ}_{\rho_{\tau,s}}(w u\iota_2(g), r(a)) |\mathrm{det}m|^{s-\zeta-\frac{n-1}{2}}\psi_1(u)dadmdudg. 
\end{multline*}

By \cite[Claim 4.1]{Kap}, one has
\[
I_1 = \gamma(s-\zeta, \chi_1 \otimes \tau) I_2.
\]

Let 
\[
I_1' = \int_{G'(F)} \varphi_{\zeta}(g,I_1,I_2)   \int_{N^{\circ}(F)}W^{\circ}_{\rho_{\tau,s}}(w u\iota(g),1) \psi_1(u) dudydg,
\]
and
\[
I_2' = \int_{G'(F)} \varphi_{\zeta}(g,I_1,I_2)  \int_{N^{\circ}(F)} \int_{M_{1\times (n-2)}(F)} W^{\circ}_{\rho_{\tau,s}}(w u\iota(g), r(a)) \psi_1(u)dadudydg.
\]

By \cite[Proof of Lemma 4.1]{Kap}, one has 
\[
I_1 = I_1', I_2=I_2'.
\]

By the definition of $\varphi_{\xi}$ as element of an induced representation
 and since we are in the unramified case, we have 
\begin{align*}
I_1' &= \int_{T_{G'}(F)} \chi(g)  \int_{\overline{U_2}(F)}  \int_{N^{\circ}(F)}W^{\circ}_{\rho_{\tau,s}}(w u\iota(\bar{n}g),1) \psi_1(u) dudydg\\
&= \int_{F^{\times}} \int_{F^{\times}} \chi_1(m)\chi_2(b)|m|^{-\frac{1}{2}} \int_{\overline{U_2}(F)}  \int_{N^{\circ}(F)} W^{\circ}_{\rho_{\tau,s}} \left(w u\iota(\bar{n}) \iota_2\begin{psmatrix}
m&&&\\&b&&\\&&b^{-1}&\\&&&m^{-1}
\end{psmatrix},1\right) \\
&\times \psi_1(u) dud\bar{n}d^{\times}md^{\times}b
\end{align*}
and
\begin{align*}
I_2' &= \int_{T_{G'}(F)} \chi(g)  \int_{\overline{U_2}(F)}  \int_{N^{\circ}} \int_{M_{1\times (n-2)}(F)} W^{\circ}_{\rho_{\tau,s}}(w u\iota(\bar{n}g), r(a)) \psi_1(u)dadud\bar{n}dg \\
&=\int_{F^{\times}} \int_{F^{\times}} \chi_1(m)\chi_2(b)|m|^{-\frac{1}{2}}   \\
& \times \int_{\overline{U_2}(F)}  \int_{N^{\circ}(F)}  \int_{M_{1\times (n-2)}(F)}  W^{\circ}_{\rho_{\tau,s}} \left(w u\iota(\bar{n}) \iota_2\begin{psmatrix}
m&&&\\&b&&\\&&b^{-1}&\\&&&m^{-1}
\end{psmatrix},r(a)\right) \psi_1(u) dadud\bar{n}d^{\times}md^{\times}b.
\end{align*}
Here
\[
\chi \begin{psmatrix}
m&&& \\ &b&& \\ &&b^{-1}& \\ &&&m^{-1}
\end{psmatrix} = \chi_1(m)\chi_2(b)|m|^{-\frac{1}{2}},
\]
where $\chi_1, \chi_2$ are characters on $F^{\times}$.

Then we have
\[
I_1' = \gamma(s-\zeta, \chi_1 \otimes \tau)I_2'.
\]
\end{proof}

Inserting the result to our local integral $I_s$ we get
\begin{align*}
I_s(y) &=  c_q\alpha(y_1, y_2) \int_{iI_F +\sigma_1}\int_{iI_F+\sigma_2} H_{y,s_2}(1) \int_{F^{\times}} H'_y\begin{psmatrix}
a_1 &\\&a_1^{-1}
\end{psmatrix}|a_1|^{-s_2} d^{\times}a_1 \\
& \times \gamma(s-s_1+s_2, \chi'\otimes \tau)^{-1} I'_{s, s_1,s_2}ds_1ds_2\\
&= c_q\alpha(y_1, y_2)\int_{iI_F +\sigma_1}\int_{iI_F+\sigma_2}\frac{ c(q^{s_2},y_1)\zeta_v(-s_2-\frac{d_1}{2}+2)^2H_{y,s_2}(1)}{ \gamma(s-s_1+s_2, \chi'\otimes \tau)} I'_{s, s_1,s_2}ds_1ds_2.
\end{align*}

\begin{remark}
Since $c(q^{s_2},y_1)\zeta_v(-s_2-\frac{d_1}{2}+2)^2H_{y,s_2}(1)$ converges when
\[
\Re(-s_2) > \frac{d_1}{2}+3,
\]
\[
c(q^{s_2},y_1) \zeta_v(-s_2-\frac{d_1}{2}+2)^2H_{y,s_2}(1) \gamma(s-s_1+s_2, \chi'\otimes \tau)^{-1} I'_{s, s_1,s_2}
\]
converges when 
\[
\max(\frac{d_1}{2}+3, -n+\frac{d_2}{2}+8+c_1) < \Re(-s_2) < n+\frac{d_2}{2}+2+c_1+c_2,
\]
$\Re(s_1)+\Re(-s_2)$ large, and $\Re(s)$ bounded in some vertical strip depends on $\Re(s_1)+\Re(-s_2)$ as in \cref{6.3}.
Since we set $d_2>d_1$ the region for $\Re(-s_2)$ is non-empty. 
\end{remark}

To compute our final result, we remain to compute $I'_{s,s_1,s_2}(y)$. By the properties of $W^{\circ}_{\rho_{\tau,s}}$, we have
\begin{align*}
I'_{s, s_1,s_2}(y) & = |4\mathcal{Q}'(y_2)|^{-\Re(s)+\frac{1}{2}-\frac{n}{2}} \int_{F^{\times}} \int_{F^{\times}}  |m|^{-s_1+s_2+\frac{d_2}{2}}|b|^{-s_1}   \\
& \times \int_{\overline{U_2}(F)}  \int_{N^{\circ}(F)}  \int_{M_{1\times (n-2)}(F)} W^{\circ}_{\rho_{\tau,s}} \left(r(a)w u\iota(\bar{n}) \iota_2\begin{psmatrix}
-4\mathcal{Q}'(y_2)m&&&\\&b&&\\&&b^{-1}&\\&&&(-4\mathcal{Q}'(y_2)m)^{-1}
\end{psmatrix},1\right)\\
& \times \psi_1(u) dadud\bar{n}d^{\times}md^{\times}b.
\end{align*}

Since $\overline{U_2}(F)$ normalizes $N^{\circ}(F)$ and $\psi_1$, we can interchange $u\iota(\bar{n})$ to $\iota(\bar{n})u$ in the integral. Denoting $\,^{(r(a)w_0)^{-1}}(\iota(\bar{n})u)=(r(a)w_0)(\iota(\bar{n})u)(r(a)w_0)^{-1}$, we have that
\begin{align*}
& I'_{s, s_1,s_2}(y) \\
&=|4\mathcal{Q}'(y_2)|^{-\Re(s)+\frac{1}{2}-\frac{n}{2}} \int_{F^{\times}} \int_{F^{\times}}  |m|^{-s_1+s_2+\frac{d_2}{2}}|b|^{-s_1}  \\
& \times  \int_{\overline{U_2}(F)}  \int_{N^{\circ}(F)}  \int_{M_{1\times (n-2)}(F)} W^{\circ}_{\rho_{\tau,s}} \left(\,^{(r(a)w_0)^{-1}}(\iota(\bar{n})u) r(a)w  \iota_2\begin{psmatrix}
-4\mathcal{Q}'(y_2)m&&&\\&b&&\\&&b^{-1}&\\&&&(-4\mathcal{Q}'(y_2)m)^{-1}
\end{psmatrix},1\right) \\
& \times \psi_1(u) dadud\bar{n}d^{\times}md^{\times}b.
\end{align*}

Let $\tilde{w_0} = \left( \begin{array}{ccc}
 &  & I_n\\
& (-1)^n & \\
I_n &  & 
\end{array} \right) $ and let $V$ be the subgroup  
\[
V(R) = \left\{ \left( \begin{array}{ccccc}
I_{n-1} & 0& y_2& y_3&s \\
&1&0&0&y_3' \\
&&1&0&y_2' \\
&&&1&0\\
&&&&I_{n-1}
\end{array} \right) \in H(R) \right\}.
\]

By \cite[Claim 4.2]{Kap}, denoting $\,^{\tilde{w_0}}V=(\tilde{w_0})^{-1}V\tilde{w_0}$ we have
\begin{align*} 
& I'_{s, s_1,s_2}(y) \\
&= |4\mathcal{Q}'(y_2)|^{-\Re(s)+\frac{1}{2}-\frac{n}{2}} \int_{F^{\times}} \int_{F^{\times}} |m|^{-s_1+s_2+\frac{d_2}{2}}|b|^{-s_1} \\
& \times\int_{M_{1\times (n-2)}(F)}  \int_{\,^{\tilde{w_0}}V(F)} W^{\circ}_{\rho_{\tau,s}}\left(vu(a)w \iota_2\begin{psmatrix}
-4\mathcal{Q}'(y_2)m&&&\\&b&&\\&&b^{-1}&\\&&&(-4\mathcal{Q}'(y_2)m)^{-1}
\end{psmatrix}, 1\right)\psi_1(u) \\ 
& \times dvdad^{\times}md^{\times}b \\
&= |4\mathcal{Q}'(y_2)|^{-\Re(s)+\frac{1}{2}-\frac{n}{2}} \sum_{k\in \Z}  |\varpi^{k}|^{-s_1+s_2+\frac{d_2}{2}}  \\
&\times  \int_{M_{1\times (n-2)}(F)}\int_{\SO_2(F)} \int_{\,^{\tilde{w_0}}V(F)} W^{\circ}_{\rho_{\tau,s}}\left(vu(a)w x \iota_2\begin{psmatrix}
-4\mathcal{Q}'(y_2)\varpi^{k}&&&\\&1&&\\&&1&\\&&&(-4\mathcal{Q}'(y_2)\varpi^{k})^{-1}
\end{psmatrix}, 1\right)\\
& \times \mu_{s_1}(x)\psi_1(u)dvdxda.
\end{align*}
Here $\mu_{s_1}(x)=|b|^{-s_1}$ is a character for $x\in \iota(G_1(F)) \cong \SO_2(F)$ (see \cref{lem2}). 

As in \cite[Page 161]{Kap}, by the structure of $x\in \iota(G_1(F))$, we have $\,^{(u(a)w_0)^{-1}}x=\,^{w' \tilde{w_0}}x$, where $w'=\mathrm{diag}(I_n, -1, I_n)$. Thus we have
\begin{align*}
& I'_{s, s_1,s_2}(y) \\
& = |4\mathcal{Q}'(y_2)|^{-\Re(s)+\frac{1}{2}-\frac{n}{2}} \sum_{k\in \Z}|\varpi^{k}|^{-s_1+s_2+\frac{d_2}{2}} \\
& \times \int_{M_{1\times (n-2)}(F)}\int_{\,^{w' \tilde{w_0}}\SO_2(F)} \int_{\,^{\tilde{w_0}}V(F)}W^{\circ}_{\rho_{\tau,s}}\left(vxu(a)w \iota_2\begin{psmatrix}
-4\mathcal{Q}'(y_2)\varpi^{k}&&&\\&1&&\\&&1&\\&&&(-4\mathcal{Q}'(y_2)\varpi^{k})^{-1}
\end{psmatrix}, 1\right)\\
& \times \mu_{s_1}(x)\psi_1(u)dvdxda.
\end{align*}

We regard the $dvdx$ integral over $\,^{\tilde{w_0}}V(F) \times \,^{w' \tilde{w_0}}\SO_2(F)$ as a function on $H(F)$:
\begin{align} \label{Bessel}
B_{W_{\rho_{\tau,s}}}(h):=\int_{\,^{w' \tilde{w_0}}\SO_2(F)} \int_{\,^{\tilde{w_0}}V(F)} W^{\circ}_{\rho_{\tau,s}}(vxh, 1)\psi_1(v)\mu_{s_1}(x)dvdx.
\end{align}

Let 
\[
\psi'_1=\psi\left(\sum_{i=1}^{n-1}v_{i,i+1}+\frac{1}{2}v_{2n,1}\right) \quad (v\in \,^{\tilde{w_0}}Z_n(F) \ltimes \,^{\tilde{w_0}}V(F))
\]
be a character. 

Similar to the Bessel function in \cite[Page 162]{Kap}, our function $B_{W^{\circ}_{\rho_{\tau,s}}}$ is an unramified Bessel function which corresponds to the Bessel functional defined for the subgroup 
\[
(\,^{\tilde{w_0}}Z_n(F) \ltimes \,^{\tilde{w_0}}V(F)) \rtimes \,^{w' \tilde{w_0}}\SO_2(F) = \,^{\tilde{w_0}}R^{\circ}(F)
\]
(for $\,^{w' \tilde{w_0}}\SO_2(F)$ split) and representations $\rho_{\tau,s}$, $\psi'_1$ and $\mu_s$.

For $\delta=(\delta_1, \dots, \delta_n)\in \Z^n$, we denote
\[
\varpi_{H}^{\delta} = \mathrm{diag}(\varpi^{\delta_1}, \dots, \varpi^{\delta_n},1,\varpi^{-\delta_1}, \dots, \varpi^{-\delta_n})\in T_H(F).
\]

Substituting the function to our local integral we get
\[
I_{s, s_1,s_2}'(y) = \sum_{k\in \Z}  |\varpi^{k}|^{-s_1+s_2+\frac{d_2}{2}} \int_{M_{1\times (n-2)}(F)} B_{W^{\circ}_{\rho_{\tau,s}}} (\tilde{w_0} u(a)w \varpi_H^{\tilde{\delta}_{k+\mathrm{val}(4\mathcal{Q}'(y_2))}})da, 
\]
where $\tilde{\delta}_k=(0_{n-2},k,0)\in \Z^n$. 

We have the following equality using the same argument as that proving \cite[Claim 4.3]{Kap}:

\begin{lemma}
\[
\int_{M_{1\times (n-2)}(F)}B_{W^{\circ}_{\rho_{\tau,s}}}(\tilde{w_0} u(a)w \varpi_H^{\tilde{\delta}_{k}})da =B_{W^{\circ}_{\rho_{\tau,s}}}(\varpi_H^{\delta_{-k}})|\varpi^k|^{n-2},
\]
where $\delta_{k}=(k,0_{n-1})\in \Z^n$. 
\end{lemma}\qed

Thus we get
\[
I_{s, s_1,s_2}'(y) = |4\mathcal{Q}'(y_2)|^{-\Re(s)+\frac{1}{2}-\frac{n}{2}} \sum_{k\in \Z}  q^{-(-s_1+s_2+n-2+\frac{d_2}{2})k}  B_{W^{\circ}_{\rho_{\tau,s}}}(\varpi_H^{\delta_{-k-\mathrm{val}(4\mathcal{Q'}(y_2))}}).
\]
Since $B_{\psi'_1, s_1}$ is an unramified Bessel function, by the vanishing condition of unramified Bessel function, $B_{\psi'_1, s_1}(\varpi_H^{\delta_{-k-\mathrm{val}(4\mathcal{Q'}(y_2)}})=0$ unless $k \leq -\mathrm{val}(4\mathcal{Q'}(y_2))$. We have 
\begin{align*}
I_{s, s_1,s_2}'(y) = |4\mathcal{Q}'(y_2)|^{-\Re(s)+\frac{1}{2}} \sum_{k\leq -\mathrm{val}(4\mathcal{Q'}(y_2))}  q^{-(-s_1+s_2+n-2+\frac{d_2}{2})k}  B_{W^{\circ}_{\rho_{\tau,s}}}(\varpi_H^{\delta_{-k-\mathrm{val}(4\mathcal{Q'}(y_2))}}).
\end{align*}

We proceed to state the main theorem of this section. Consider
\begin{align} \label{Laurent}
\frac{\prod_{i=1}^2c(q^{s_2},y_i)\zeta_v(-s_2-\frac{d_1}{2}+2)^2\zeta_v(-s_2)^2  }{\gamma(s-s_1+s_2, \chi'\otimes \tau)}B_{W^{\circ}_{\rho_{\tau,s}}}(\varpi_H^{\delta_{k-\mathrm{val}(4\mathcal{Q}'(y_2))}})
\end{align}
as a product of Laurent series in $q^{s_1}$ and $q^{s_2}$, where $c(q^{s_2},y_i)$ are those defined in \cref{c_s_2}.  Let
\begin{align}\begin{split} \label{q-coefficient}
&C_{k,s}(y):= \\
&c_q\int_{iI_F +\sigma_1}\int_{iI_F+\sigma_2}q^{(-s_1+s_2)k}\frac{\prod_{i=1}^2c(q^{s_2},y_i)\zeta_v(-s_2-\frac{d_1}{2}+2)^2\zeta_v(-s_2)^2}{\gamma(s-s_1+s_2, \chi'\otimes \tau)}B_{W^{\circ}_{\rho_{\tau,s}}}(\varpi_H^{\delta_{k-\mathrm{val}(4\mathcal{Q}'(y_2))}}),
\end{split}
\end{align}
where $c_q$ is as defined in \cref{c_q}. This is nothing but the product of the $-k$-th coefficient in $q^{s_1}$ and the $k$-th coefficient in  $q^{s_2}$ of \eqref{Laurent}.

\begin{theorem}\label{thm2}
For all the data unramified and $\Re(s)$ large, we have
\begin{align*}
I_s(y)  &= \alpha(y_1, y_2)|4Q'(y_2)|^{-\Re(s)+\frac{1}{2}-\frac{n}{2}} \sum_{k=\mathrm{val}(4\mathcal{Q'}(y_2))}^{\infty}  q^{(n-2+\frac{d_2}{2})k} C_{k,s}(y),
\end{align*}
where $\alpha(y_1, y_2)$ is as defined in \cref{c_k}. 
\end{theorem}

\begin{proof}
Combining all the above results we get 
\begin{align*}
& I_s(y) \\
&= c_q\alpha(y_1, y_2)\int_{iI_F +\sigma_1}\int_{iI_F+\sigma_2} \frac{c(q^{s_2},y_1)\zeta_v(-s_2-\frac{d_1}{2}+2)^2H_{y,s_2}(1)}{\gamma(s-s_1+s_2, \chi'\otimes \tau)}  I'_{s,s_1,s_2}ds_1ds_2 \\ 
&=c_q \alpha(y_1, y_2)\int_{iI_F +\sigma_1}\int_{iI_F+\sigma_2} \frac{\prod_{i=1}^2c(q^{s_2},y_i)\zeta_v(-s_2-\frac{d_1}{2}+2)^2\zeta_v(-s_2)^2}{\gamma(s-s_1+s_2, \chi'\otimes \tau)}  \\
& \times  \sum_{k=\mathrm{val}(4\mathcal{Q'}(y_2))}^{\infty}  q^{(-s_1+s_2+n-2+\frac{d_2}{2})k} B_{W^{\circ}_{\rho_{\tau,s}}}(\varpi_H^{\delta_{k-\mathrm{val}(4\mathcal{Q'}(y_2))}})ds_1ds_2  \\
&=c_q\alpha(y_1, y_2) \sum_{k=\mathrm{val}(4\mathcal{Q'}(y_2))}^{\infty}  q^{(n-2+\frac{d_2}{2})k} \int_{iI_F +\sigma_1}\int_{iI_F+\sigma_2} \frac{\prod_{i=1}^2c(q^{s_2},y_i)\zeta_v(-s_2-\frac{d_1}{2}+2)^2\zeta_v(-s_2)^2}{\gamma(s-s_1+s_2, \chi'\otimes \tau)}   \\
& \times q^{(-s_1+s_2)k}B_{W^{\circ}_{\rho_{\tau,s}}}(\varpi_H^{\delta_{k-\mathrm{val}(4\mathcal{Q'}(y_2))}})ds_1ds_2.
\end{align*}
 
We will make the above manipulations rigorous in \cref{6.5} by showing that the infinite sum
\[
\frac{\prod_{i=1}^2c(q^{s_2},y_i)\zeta_v(-s_2-\frac{d_1}{2}+2)^2\zeta_v(-s_2)^2} {\gamma(s-s_1+s_2, \chi'\otimes \tau)}  \sum_{k=\mathrm{val}(4\mathcal{Q'}(y_2))}^{\infty} q^{(-s_1+s_2+n-2+\frac{d_2}{2})k} B_{W^{\circ}_{\rho_{\tau,s}}}(\varpi_H^{\delta_{k-\mathrm{val}(4\mathcal{Q'}(y_2))}})
\]
converges absolutely for  $\Re(s_1), \Re(-s_2)$ large and $\Re(s)$ lies in some region in the right plane that depends on $\Re(s_1), \Re(-s_2)$.
\end{proof}

We give an explicit expression for the unramified Bessel function $B_{W^{\circ}_{\rho_{\tau,s}}}$ following \cite{Kap} below.  Let 
\begin{align}
    B^\circ
\end{align}
be the unique Bessel function corresponding to the data $\rho_{\tau,s}$, $\psi'_1$ and $\mu_s$ such that $B^\circ(1)=1.$  Let us recall the formula for $B^{\circ}$ from \cite[Theorem 1.6]{BFF}.
We first define three functions depending on the characters $\chi_{i,s}$ and the character $\mu_{s_1}$:
\[
\Delta(\chi_s) = \prod_{i=1}^n\chi_{i,s}(\varpi)^{-1+i-n}(1-\chi_{i,s}(\varpi)^2)\prod_{1\leq i\leq j\leq n}(1-\chi_{i,s}(\varpi)\chi_{j,s}(\varpi))(1-\chi_{i,s}(\varpi)\chi_{j,s}(\varpi)^{-1}),
\]
\[
D(\chi_s, \mu_{s_1}) = \prod_{i=1}^n \chi_{i,s}(\varpi)^{-(n+1-i)} \prod_{i=1}^{n} (1-\chi_{i,s}(\varpi)\mu_{s_1}(\varpi)q^{-\frac{1}{2}})(1-\chi_{i,s}(\varpi)\mu_{s_1}(\varpi)^{-1}q^{-\frac{1}{2}}),
\]
and for $t \in T_H(F)$,
\begin{align} \label{B:eval}
B^\circ(t) = \frac{
\delta_H(t)^{1/2}}{(1-q^{-1})\Delta(\chi_s)}\sum_{w \in W}\mathrm{sgn}(w)D(\,^{w}\chi_s, \mu_{s_1})\,^{w}\chi_s(t)^{-1}.
\end{align}

The function $\Delta(\chi)$ is (up to a sign) the denominator in the Weyl character formula for the symplectic group $\Sp_{2n}$.  The equality \eqref{B:eval} is the analog for unramified Bessel functions of the Casselman-Shalika formula for Whittaker functions (see \cite[Theorem 1.6]{BFF}).

\begin{lemma} \label{4.5}
Let $W$ be the Weyl group of $H$. Let $t_{\tau}$ be the Satake parameter for $\tau$. We denote the local L-function for $\mu_{s_1} \times \tau$ and symmetric square L-function for $\tau$ as 
\begin{align*}
L(s, \mu_{s_1} \times \tau) &= \mathrm{det}(1-(t_{\mu} \otimes t_{\tau})q^{-s})^{-1} \\ 
&= \prod_{i=1}^n (1-\chi_{i}(\varpi)\mu_{s_1}(\varpi)q^{-s})^{-1}(1-\chi_{i}(\varpi)\mu_{s_1}(\varpi)^{-1}q^{-s})^{-1},\\
L(2s, \tau, \mathrm{Sym}^2) &= \prod_{1\leq i\leq j \leq n}(1-\chi_{i}(\varpi)\chi_{i}(\varpi)q^{-2s})^{-1}\prod_{i=1}^n(1-\chi_{i}(\varpi)q^{-2s})^{-1}.
\end{align*}
For $t=\varpi_H^{\delta_{k-\mathrm{val}(4\mathcal{Q'}(y_2)}} \in T_H(F)$ such that $k \geq \mathrm{val}(4\mathcal{Q}'(y_2))$, we have
\[
B_{W^{\circ}_{\rho_{\tau,s}}}(t)= \frac{L(s, \mu_{s_1} \times \tau)}{L(2s, \tau, \mathrm{Sym}^2)}B^\circ(t).
\]
\end{lemma}

\begin{proof}
This follows from the same argument in \cite[Proof of Lemma 4.2, Page 169-170]{Kap}.
\end{proof}

\quash{\begin{proof}
Let $f_0 \in \rho_{\tau,s}$ be such that $f_0(K)=1$. Let $B_{f_0}$ be the unramified Bessel function associated with $f_0$:
\[
B_{f_0}(h)=\int_{\,^{w' \tilde{w_0}}\SO_2(F)} \int_{\,^{\tilde{w_0}}V(F)} f_0(vxh, 1)\psi_1(v)\mu_{s_1}(x)dvdx.
\]
This integral is absolutely convergent for $\mathrm{Re}(s_1)$....by ??
Now Bessel functions are unique up to scalar multiple \textcolor{red}{REFERENCE?}.
Since we normalize $W_{\rho_{\tau,s}}$ by $W_{\rho_{\tau,s}}(1,1)=1$, we have

$B_{W_{\rho_{\tau,s}}}=W_{\tau_s}(1)^{-1}B_{f_0}(1)B^\circ$, where
\[
W_{\tau_s}(1)=\sum_{1\leq i\leq j\leq n}(1-\chi_{i,s}(\varpi)\chi_{i,s}(\varpi)^{-1}q^{-1})
\]
(see \cite{CS}), and by \cite[Theorem 1]{Kap}
\[
B_{f_0}(1) = \frac{\prod_{1\leq i\leq j \leq n}(1-\chi_{i,s}(\varpi)\chi_{j,s}(\varpi)q^{-1})(1-\chi_{i,s}(\varpi)\chi_{j,s}(\varpi)^{-1}q^{-1})\prod_{i=1}^n(1-\chi_{i,s}(\varpi)q^{-1})}{ \prod_{i=1}^n (1-\chi_{i,s}(\varpi)\mu_{s_1}(\varpi)q^{-\frac{1}{2}})(1-\chi_{i,s}(\varpi)\mu_{s_1}(\varpi)^{-1}q^{-\frac{1}{2}})}.
\]

Using the explicit formula for the unramified Bessel functional (see \cite{BFF}, \cite[Section 2.5]{Kap}), we deduce the result using similar notations as in \cite[Section 2.5]{Kap}.
\end{proof}}

\section{Convergence of local integrals in unramified computation} \label{6}

Let $F$ be a non-Archimedean local field. In this section we prove absolute convergence for various local integrals that appear in the unramified computations in \cref{4}. We point out that the bounds we prove in this section are not used in our global considerations.

The points of the group $w N^{\circ}\overline{U_2}' w^{-1}$ in an $F$-algebra $R$ are 
\begin{align*}
& w N^{\circ}\overline{U_2}'(R) w^{-1} \\
& = \left\{  \begin{psmatrix}
1 & &  &&&& \\ c_1&1& &&&& \\ &&I_{n-2}&&&&& \\ -\alpha c_2&0&v_3'&1&&& \\ v_1 & v_2 & T & v_3 & I_{n-2} & & \\  -\alpha c_1 &0&v_2'&0&&1& \\ \alpha c_1^2-\frac{1}{2}c_2^2&\alpha c_1&v_1'&\alpha c_2&&-c_1&1
\end{psmatrix}: c_1,c_2\in R, v_1,v_2,v_3\in R^{n-2}, T\in M_{(n-2)}(R)  \right\}.
\end{align*}
Here $\overline{U_2}'=\iota(\overline{U_2})$ and $\alpha=\frac{1}{2}$.
\begin{lemma}\label{6.1}
For $\Re(s)$ large, the integral 
\begin{align} \label{integral 6.1}
\int_{w N^{\circ}(F)\overline{U_2}'(F) w^{-1}} W^{\circ}_{\rho_{\tau,s}}(uv,1)du
\end{align}
converges. Here $v$ is a unipotent element of $H(F)$ defined as in \cref{v}. 
\end{lemma}

\begin{remark}
We remark that this is different from \cref{5.1}. The integral is over the group $w N^{\circ}(F)\overline{U_2}'(F) w^{-1}$ whereas the integral in \cref{5.1} is over $w N^{\circ}(F) w^{-1}$.
\end{remark}

\begin{proof}
For $u\in w N^{\circ}(F)\overline{U_2}'(F) w^{-1}$, we have
\[
uv=\begin{psmatrix}
1 & &  &&&& \\ c_1&1&&c&&-\frac{1}{2}c^2& \\ &&I_{n-2}&&&&& \\ -\alpha c_2&0&v_3'&1&&-c& \\ v_1 & v_2 & T & v_2c+v_3 & I_{n-2} &-\frac{1}{2}v_2c^2-v_3c & \\  -\alpha c_1 &0&v_2'&0&&1& \\ \alpha c_1^2-\frac{1}{2}c_2^2&\alpha c_1&v_1'&\alpha c_1c+\alpha c_2&&-c_1&1
\end{psmatrix},
\]
where $c\in F$. 

Considering the Iwasawa decomposition of $u\in w N^{\circ}(F)\overline{U_2}'(F) w^{-1}$ in $H$ (using the notation as in \cref{5.1}), we have
\[
uv=na'k.
\]
Here $(n,a',k) \in N_n(F) \times T_H(F) \times K_H$ and $a'=\mathrm{diag}(a,1,w_0 a^{-1}w_0)$ for $a\in \GL_n(F)$. We denote the $i$-th line of $uv$ as $(uv)_i$. By a similar argument as in \cref{5.2}, we have 
\[
[\mathcal{L}(uv)]^{-2j} \leq |\frac{a_j}{a_{j+1}}| \leq [\mathcal{L}(uv)]^{2j}.
\]
Here $j=1,\dots,n$ and $[\mathcal{B}(uv)]=\max \{1, \lVert \mathcal{B}(uv) \rVert \}$ and $\lVert \cdot \rVert$ is the sup-norm, and
\[
[\mathcal{L}(uv)]^{-n} \leq D(uv)=|\mathrm{det}(a)| \leq [\mathcal{L}(uv)]^{-1}.
\]

Note that by arguing as in \cref{5.1} and \cref{5.2}, the integral \cref{integral 6.1} is majorized by
\begin{align}\label{6.1 bound}
\sum_{j=1}^{\nu} c_{j,s}[v]^C \int_{w N^{\circ}\overline{U_2}'(F) w^{-1}}[u]^{-\Re(s)-\frac{n-1}{2}+C}du,
\end{align}
where $C$ is a positive constant that depends only on $\tau$. 
\end{proof}

\begin{lemma}\label{6.2}
The integral 
\begin{align} 
I_s(y) =\int_{G(F)}  \Phi_{y}(g) \int_{N^{\circ}(F)} W^{\circ}_{\rho_{\tau,s}}(w u\iota(g),a_y)\psi_1(u)dudg
\end{align}
converges absolutely for $\Re(s)$ large enough.
\end{lemma}

\begin{proof}
By the Iwasawa decomposition with respect to the lower Borel subgroup of $G(F)$, we have
\[
I_s(y)= \int_{\overline{U_2}(F)}\int_{T_G(F)} \Phi_{y}(\bar{n}g)\delta_{\overline{B}_G}^{-1}(g) \int_{N^{\circ}(F)} W^{\circ}_{\rho_{\tau,s}}(w u\iota(\bar{n}g),a_y)\psi_1(u)dud\bar{n}dg. 
\]

Since $\Phi_{y}$ is invariant under $\overline{U_2}(F)$, we have 
\begin{align*}
I_s(y) &= \int_{T_G(F)} \Phi_{y}(g) \delta_{\overline{B}_G}^{-1}(g)\int_{\overline{U_2}(F)} \int_{N^{\circ}(F)} W^{\circ}_{\rho_{\tau,s}}(w u\iota(\bar{n}g),a_y)\psi_1(u)dud\bar{n}dg \\
&= \int_{G_1(F)} \int_{A_1(F)} \int_{A_2(F)} \Phi_{y}(xa_1a_2) \delta_{\overline{B}_G}^{-1}(xa_1a_2) \\
&\times \int_{\overline{U_2}(F)} \int_{N^{\circ}(F)} W^{\circ}_{\rho_{\tau,s}}(w u\iota(\bar{n}xa_1a_2),a_y)\psi_1(u)dud\bar{n}dxda_1da_2.
\end{align*}

Using the Iwasawa decomposition of $\SO_3(F)$ as in \cref{5.3}, we have
\begin{align*}
|I_s(y)| &\leq    \int_{F^{\times}} \int_{F^{\times}}\int_{F^{\times}} |\Phi_{y}|\left(\begin{psmatrix}
a_1&\\&ba_1^{-1}
\end{psmatrix},\begin{psmatrix}
a_2& \\ &b^{-1}a_2^{-1}
\end{psmatrix}\right)|a_1a_2|^{\Re(s)-n+5}|\lfloor ba_1a_2^{-1} \rfloor|^{\Re(s)-n+3}  \\
&\times \int_{\overline{U_2}(F)} \int_{N^{\circ}(F)} |W^{\circ}_{\rho_{\tau,s}}|((w u\iota(\bar{n})n', \mathrm{diag}(-4\mathcal{Q}'(y_2)a_1a_2, \lfloor b^{-1}a_1a_2^{-1} \rfloor , I_{n-2})) \\
& \times dud\bar{n}d^{\times}bd^{\times}a_1d^{\times}a_2\\
&=  \int_{F^{\times}} \int_{F^{\times}}\int_{F^{\times}} |\Phi_{y}|\left(\begin{psmatrix}
a_1&\\&ba_1^{-1}
\end{psmatrix},\begin{psmatrix}
a_2& \\ &b^{-1}a_2^{-1}
\end{psmatrix}\right)|a_1a_2|^{\Re(s)-n+5}|\lfloor ba_1a_2^{-1} \rfloor|^{\Re(s)-n+3} \\
&\times\int_{\overline{U_2}(F)} \int_{N^{\circ}(F)} |W^{\circ}_{\rho_{\tau,s}}|((w u\iota(\bar{n})w^{-1})(w n'w^{-1}),\mathrm{diag}(-4\mathcal{Q'}(y_2)a_1a_2, \lfloor b^{-1}a_1a_2^{-1} \rfloor , I_{n-2}))\\
& \times dud\bar{n}d^{\times}bd^{\times}a_1d^{\times}a_2 \\
&= \int_{F^{\times}} \int_{F^{\times}}\int_{F^{\times}} |\Phi_{y}|\left(\begin{psmatrix}
a_1&\\&ba_1^{-1}
\end{psmatrix},\begin{psmatrix}
a_2& \\ &b^{-1}a_2^{-1}
\end{psmatrix}\right)|a_1a_2|^{\Re(s)-n+5}|\lfloor ba_1a_2^{-1} \rfloor|^{\Re(s)-n+3}\\
&\times \int_{w N^{\circ}\overline{U_2}'(F)w^{-1}} |W^{\circ}_{\rho_{\tau,s}}|(u(w n'w^{-1}), \mathrm{diag}(-4\mathcal{Q}'(y_2)a_1a_2, \lfloor b^{-1}a_1a_2^{-1} \rfloor , I_{n-2})) \\
& \times dud^{\times}bd^{\times}a_1d^{\times}a_2,
\end{align*}
where $n'$ is as defined in \cref{n'}.

By \cref{6.1 bound} in \cref{6.1} and a similar argument as in \cref{5.3}, the above integral is majorized by a finite sum of integrals of the form 
\begin{multline*}
\int_{F^{\times}} \int_{F^{\times}}\int_{F^{\times}} |\Phi_{y}|\left(\begin{psmatrix}
a_1&\\&ba_1^{-1}
\end{psmatrix},\begin{psmatrix}
a_2& \\ &b^{-1}a_2^{-1}
\end{psmatrix}\right)|a_1a_2|^{\Re(s)-n+5}|\lfloor ba_1a_2^{-1} \rfloor|^{\Re(s)-n+3-c_1} \\
\times \left(\int_{w N^{\circ}\overline{U_2}'(F)w^{-1}}[u]^{-\Re(s)-\frac{n-1}{2}+c_2}du\right)d^{\times}bd^{\times}a_1d^{\times}a_2,
\end{multline*}
where $c_1,c_2>0$ are constants depend only on $\tau$. 

Substituting the above result into our local integral while using the explicit formula for $\Phi_y$, we have that for $\Re(s)$ large the above integral is majorized by
\begin{align*}
&\int_{F^{\times}} \int_{F^{\times}}\int_{F^{\times}}  |H_{1,y}|\left(\begin{psmatrix}
a_1&\\&ba_1^{-1}
\end{psmatrix},\begin{psmatrix}
a_2& \\ &b^{-1}a_2^{-1}
\end{psmatrix}\right) |H_{2,y}|(\begin{psmatrix}
a_1&\\&ba_1^{-1}
\end{psmatrix},\begin{psmatrix}
a_2& \\ &b^{-1}a_2^{-1}
\end{psmatrix})\\
& \times |a_1a_2|^{\Re(s)-n+5} d^{\times}bd^{\times}a_1d^{\times}a_2\\
=& \int_{F^{\times}}\int_{F^{\times}} |a_1|^{\Re(s)+\frac{d_1-n+3}{2}}|a_2|^{\Re(s)+\frac{d_2-n+3}{2}}(\one_{\varpi^{-k_{y_1}}\mathcal{O}_{F}}(a_1) - \sum_{k=1}^{\infty} q^{(\frac{d_1}{2}-2)k}(q-1)\one_{\varpi^{k-k_{y_1}}  \mathcal{O}_{F}}(a_1))\\
&\times (\one_{\varpi^{-k_{y_2}}\mathcal{O}_{F}}(a_2) - \sum_{k=1}^{\infty} q^{(\frac{d_1}{2}-2)k}(q-1)\one_{\varpi^{k-k_{y_2}}  \mathcal{O}_{F}}(a_2))da_1^{\times}da_2^{\times},
\end{align*}
which converges absolutely for $\Re(s)$ large by a similar computation as for $H_{y,s_1,s_2}(1)$ in \cref{4}.
\end{proof}

\begin{lemma} \label{6.3}
There exists positive integers $C_1<C_2, C_3, C_4$ which depend on $(\tau, d_1, d_2, n)$ such that 
\begin{multline*}
I_{s,s_1,s_2}(y)= \int_{F^{\times}} \int_{F^{\times}} \chi'_{s_1,s_2}(\begin{psmatrix}
1& \\ &(mb)^{-1}
\end{psmatrix},\begin{psmatrix}
m& \\ &b
\end{psmatrix}) \int_{\overline{U_2}(F)}  \int_{N^{\circ}(F)}\\
\times W^{\circ}_{\rho_{\tau,s}}\left(w u\iota(\bar{n})\iota_2\left(\begin{psmatrix}
m&&&\\&b&&\\&&b^{-1}&\\&&&m^{-1}
\end{psmatrix}\right),a_y\right) \psi_1(u) dud\bar{n}d^{\times}md^{\times}b
\end{multline*}
converges absolutely for 
\[
\Re(s_1)+\Re(-s_2)+C_1 < \Re(s) < \Re(s_1)+2\Re(-s_2)+C_2,
\]
\[
C_3 < \Re(s_1)+\Re(-s_2), C_4 < \Re(-s_2).
\]
\end{lemma}

\begin{proof}
As in \cref{6.2}, $I_{s,s_1,s_2}$ is majorized by a finite sum of integrals of the form
\[
\int_{F^{\times}} \int_{F^{\times}}|m|^{-s_1+s_2+\frac{d_2}{2}}|b|^{-s_1}|m|^{\Re(s)-n+5}|\lfloor b \rfloor|^{\Re(s)-n+3-c_1} \int_{w N^{\circ}\overline{U_2}'(F)w^{-1}}[u]^{-\Re(s)-\frac{n-1}{2}+c_2}dud^{\times}md^{\times}b,
\]
where $c_1,c_2>0$ are constants depend only on $\tau$.

Also, as in \cref{5.3}, we have
\[
|m| \leq c'[u]|\lfloor b \rfloor|^{-2}.
\]
Here $c'$ is a conatant depends only on $\tau$.

Thus the integral is majorized by a finite sum of integrals of the form
\begin{align*}
& \int_{F^{\times}} |b|^{-s_1} |\lfloor b \rfloor|^{\Re(s)-n+3-c_1} d^{\times}b \int_{w N^{\circ}\overline{U_2}'(F)w^{-1}} \int_{|m| \leq c'[u]|\lfloor b \rfloor|^{-2}} |m|^{-s_1+s_2+\frac{d_2}{2}+\Re(s)-n+5} \\
& \times [u]^{-\Re(s)-\frac{n-1}{2}+c_2}d^{\times}mdu.
\end{align*}

The integral
\[
\int_{|m| \leq c'[u]|\lfloor b \rfloor|^{-2}} |m|^{-s_1+s_2+\frac{d_2}{2}+\Re(s)-n+5}d^{\times}m
\]
converges absolutely when 
\[
-\Re(s_1)+\Re(s_2)+\frac{d_2}{2}+\Re(s)-n+5 > 0.
\]

Thus when $\Re(s) > \Re(s_1)-\Re(s_2)-\frac{d_2}{2}+n-5$, the integral $I'_{s,s_1,s_2}$ is majorized by a finite sum of integrals of the form 
\[
\int_{F^{\times}} |b|^{-s_1}\lfloor b^2 \rfloor ^{-\Re(s)+2s_1-2s_2+d_2+n-7-c_1} d^{\times}b \int_{w N^{\circ}\overline{U_2}'(F)w^{-1}} [u]^{-s_1+s_2+\frac{d_2-3n+11}{2}+c_2}du.
\]
This integral converges absolutely when
\begin{align*}
&-\Re(s_1)+\Re(s_2)+\frac{d_2-3n+11}{2}+c_2 < -C, \\
& -\Re(s)+2\Re(s_1)-2\Re(s_2)+d_2+n-7-c_1-\Re(s_1) > 0,\\
& \Re(s)-2\Re(s_1)+2\Re(s_2)-d_2-n+7+c_1-\Re(s_1) < 0.
\end{align*}
Here $C$ is a positive integers which depends on $\tau$. 

In summary, the convergence region is equivalent to 
\begin{align*}
& \Re(s_1)+\Re(-s_2)-\frac{d_2}{2}+n-5 < \Re(s) <  \min(\Re(s_1),3\Re(s_1))+2\Re(-s_2)+d_2+n-7-c_1, \\
& \Re(s_1)+\Re(-s_2) > \frac{d_2-3n+11}{2}+c_2+C.
\end{align*}
This is non-empty when 
\[
\Re(s_1)+\Re(-s_2)-\frac{d_2}{2}+n-5 < \min(\Re(s_1),3\Re(s_1))+2\Re(-s_2)+d_2+n-7-c_1, 
\]
which is equivalent to
\[\Re(-s_2) > -\frac{3d_2}{2}+2+c_1. \]
Thus we get the non-empty region in the statement of the lemma.
\end{proof}

Let $X_1 \subset H$ be the unipotent subgroup of $H$ whose points in an $F$-algebra $R$ are 
\[
X_1(R) = \left\{  \begin{psmatrix}
1 & &  &&&& \\ &I_{n-2}& &&&& \\ &&1&&&&& \\ 0&v_3'&-\alpha c_2&1&&& \\ -\alpha c_1 & v_1' & \frac{1}{2}c_2^2 & \alpha c_2 & 1 & & \\  v_2 &T&v_1&v_3&&I_{n-2}& \\ 0&v_2'&-\alpha c_1&0&&&1
\end{psmatrix}: c_1,c_2\in F, v_1,v_2,v_3\in F^{n-2}, T\in M_{n-2}(F)\right\},
\]
where $\alpha=\frac{1}{2}$.

\begin{lemma}\label{6.4}
There are constants $C_0,C'',c_2>0$, $c_1$ which depend on $(\tau, d_1, d_2, n)$ such that 
\begin{align*}
I'_{s, s_1,s_2}(y) &= \int_{F^{\times}} \int_{F^{\times}} \chi'_{s_1,s_2}\left(\begin{psmatrix}
1&\\&(mb)^{-1} 
\end{psmatrix},\begin{psmatrix}
m&\\&b
\end{psmatrix}\right)\\ &\times \int_{\overline{U_2}(F)}  \int_{N^{\circ}(F)} \int_{M_{1\times (n-2)}(F)} W^{\circ}_{\rho_{\tau,s}}(w u\iota(\bar{n}g), r(e)a_y) \psi_1(u)dedud\bar{n}d^{\times}md^{\times}b
\end{align*}
converges absolutely for 
\begin{align*}
  \Re(s) < & \Re(s_1) + \Re(-s_2) -\frac{d_1}{2} -5 -c_1    \\ 
\Re(s) > & \max\bigg(\frac{C_0}{C_0+1}(\Re(s_1)+\Re(-s_2)-\frac{d_2}{2}-5-c_1)-\frac{n-1-C''}{C_0+1},\\
& \Re(s_1)+\frac{2C_0}{2C_0+1}(\Re(-s_2)-\frac{d_2}{2}-5-c_1)-\frac{n-3+c_2}{2C_0+1}, \\
& \Re(s_1)+\frac{2C_0}{2C_0-1}(\Re(-s_2)-\frac{d_2}{2}-5-c_1)-\frac{n-3+c_2}{2C_0-1}\bigg).
\end{align*}
\begin{align*}
& -n+\frac{d_2}{2}+8+c_1 < \Re(-s_2) < n+\frac{d_2}{2}+2+c_1+c_2, \\
& c_0(-n+3-c_2)+C'' < \Re(s_1)+\Re(-s_2).
\end{align*}
\end{lemma}

\begin{proof}
We use the approach outlined in the proof of \cite[Proposition 11.16]{Sou}.

As in \cref{6.3}, we have 
\begin{align*}
|I'_{s,s_1,s_2}| & \leq \int_{F^{\times}} \int_{F^{\times}}   |m|^{-s_1+s_2+\frac{d_2}{2}+\Re(s)+5}|b|^{-s_1}|\lfloor b \rfloor|^{\Re(s)-n+3}\\ 
& \times \int_{\overline{U_2}(F)}  \int_{N^{\circ}(F)} \int_{M_{1\times (n-2)}(F)}  |W^{\circ}_{\rho_{\tau,s}}|(w u\iota(\bar{n})n', r(e)\mathrm{diag}(-4\mathcal{Q'}(y_2)m, \lfloor b  \rfloor  , I_{n-2})) \\
& \times dedud\bar{n}d^{\times}md^{\times}b,
\end{align*}
where 
\[
r(e) = \begin{pmatrix}
0&1&0\\ 0&0&I_{n-1} \\ 1&0&e
\end{pmatrix}.
\]

Since 
\[
\begin{pmatrix}
0&1&0\\ 0&0&I_{n-1} \\ 1&0&e
\end{pmatrix} \begin{pmatrix}
m&& \\ &\lfloor b  \rfloor  & \\ &&I_{n-2}
\end{pmatrix} = \begin{pmatrix}
\lfloor b  \rfloor  && \\ &I_{n-2}& \\ &&m
\end{pmatrix} \begin{pmatrix}
0&1&0 \\ 0&0&I_{n-2} \\ 1&0&m^{-1}e
\end{pmatrix},
\]
we have 
\begin{align*}
|I'_{s,s_1,s_2}(y)| & \leq |-4\mathcal{Q}'(y_2)|^{s_1-s_2-\frac{d_2}{2}-\Re(s)-5} \int_{F^{\times}} \int_{F^{\times}}    |m|^{-s_1+s_2+\frac{d_2}{2}+\Re(s)+5}|b|^{-s_1}|\lfloor b \rfloor|^{\Re(s)-n+3}   \\
&  \times \int_{\overline{U_2}(F)} \int_{N^{\circ}(F)} \int_{M_{1\times (n-2)}(F)} |W^{\circ}_{\rho_{\tau,s}}|(v(r(m^{-2}e))w uyn',\mathrm{diag}(\lfloor b \rfloor , I_{n-2},m))\\
& \times dedudyd^{\times}md^{\times}b.
\end{align*}
Here the symbol $v(r(m^{-2}e))\in M_n(F)$ is as defined in \cref{vx}. 

By a change of variable $e \mapsto m^2e$, we have that 
\begin{align*}
& |I'_{s,s_1,s_2}(y)| \\
&\leq |-4\mathcal{Q}'(y_2)|^{s_1-s_2-\frac{d_2}{2}-\Re(s)-5} \int_{F^{\times}} \int_{F^{\times}}  |m|^{-s_1+s_2+\frac{d_2}{2}+\Re(s)+5}|b|^{-s_1}|\lfloor b \rfloor|^{\Re(s)-n+3}|m|^{n} \\
& \times \int_{\overline{U_2}(F)}  \int_{N^{\circ}(F)} \int_{M_{1\times (n-2)}(F)} |W^{\circ}_{\rho_{\tau,s}}|(v(r(e))w u\iota(\bar{n})n',\mathrm{diag}(\lfloor b \rfloor , I_{n-2},m))dedud\bar{n}d^{\times}md^{\times}b \\
&= |-4\mathcal{Q}'(y_2)|^{s_1-s_2-\frac{d_2}{2}-\Re(s)-5} \int_{F^{\times}} \int_{F^{\times}} |m|^{-s_1+s_2+\frac{d_2}{2}+\Re(s)+5}|b|^{-s_1}|\lfloor b \rfloor|^{\Re(s)-n+3}
 \\
&\times \int_{\overline{U_2}(F)} \int_{N^{\circ}(F)} \int_{M_{1\times (n-2)}(F)}  |W^{\circ}_{\rho_{\tau,s}}|(v(r(e))(w uw^{-1})(w \iota(\bar{n})w^{-1})(w n'w^{-1}),\mathrm{diag}(\lfloor b \rfloor , I_{n-2},m))\\
& \times dedud\bar{n}d^{\times}md^{\times}b.
\end{align*}
Since for $u\in N^{\circ}(F)$, 
\[
w uw^{-1} \in \begin{psmatrix}
1 & &  &&&& \\ &1& &&&& \\ &&I_{n-2}&&&&& \\ &&v_3'&1&&& \\ v_1 & v_2 & T & v_3 & I_{n-2} & & \\  0 &0&v_2'&&&1& \\ 0&0&v_1'&&&&1
\end{psmatrix}
\]
and
\[
v(r(e))\begin{psmatrix}
1 & &  &&&& \\ &1& &&&& \\ &&I_{n-2}&&&&& \\ &&v_3'&1&&& \\ v_1 & v_2 & T & v_3 & I_{n-2} & & \\  0 &0&v_2'&&&1& \\ 0&0&v_1'&&&&1
\end{psmatrix} =  \begin{psmatrix}
1 & &  &&&& \\ &I_{n-2}& &&&& \\ &&1&&&&& \\
0&v_3'&0&1&&& \\ 0&v_1'&0&0&1&&\\
v_2&T+e'v_1'-ev_1&v_1&v_3&&I_{n-2}&\\
0&v_2'&0&0&&&1
\end{psmatrix}v(r(e)),
\]
we have
\begin{align*}
&|W^{\circ}_{\rho_{\tau,s}}|(v(r(e))(w uw^{-1})(w \iota(\bar{n})w^{-1})(w n'w^{-1}),\mathrm{diag}(\lfloor b \rfloor , I_{n-2},m)) \\
=& |W^{\circ}_{\rho_{\tau,s}}|(u'v(r(e))(w \iota(\bar{n})w^{-1})(w n'w^{-1}),\mathrm{diag}(\lfloor b \rfloor , I_{n-2},m)).
\end{align*}
Here
\[
u' =  \begin{psmatrix}
1 & &  &&&& \\ &I_{n-2}& &&&& \\ &&1&&&&& \\
0&v_3'&0&1&&& \\ 0&v_1'&0&0&1&&\\
v_2&T+e'v_1'-ev_1&v_1&v_3&&I_{n-2}&\\
0&v_2'&0&0&&&1
\end{psmatrix}.
\]

We observe that 
\[
\begin{pmatrix}
0&1&0 \\ 0&0&I_{n-2} \\ 1&0&e
\end{pmatrix} \begin{pmatrix}
1&0&0 \\ c_1&1&0 \\ 0&0&I_{n-2}
\end{pmatrix} = \begin{pmatrix}
1&-c_1e&c_1 \\ 0&I_{n-2}&0 \\ 0&0&1
\end{pmatrix}\begin{pmatrix}
0&1&0 \\ 0&0&I_{n-2} \\ 1&0&e
\end{pmatrix}. 
\]
Also, for $\iota(\bar{n})\in \overline{U_2}(F)$, 
\begin{align*}
w \iota(\bar{n})w^{-1}  &= \begin{psmatrix}
1 & &  &&&& \\ c_1&1& &&&& \\ &&I_{n-2}&&&&& \\ -\alpha c_2&0&0&1&&& \\  0& 0 &  0& 0 & 1 & & \\  -\alpha c_1 &0&0&0&&1& \\ \alpha c_1^2-\frac{1}{2}c_2^2&\alpha c_1&0&\alpha c_2&&-c_1&I_{n-2}
\end{psmatrix} \\
&= v\left(\begin{psmatrix}
1&&\\c_1&1&\\&&I_{n-2}
\end{psmatrix}\right)\begin{psmatrix}
1 & &  &&&& \\ &1& &&&& \\ &&I_{n-2}&&&&& \\ -\alpha c_2&0&0&1&&& \\  0& 0 &  0& 0 & 1 & & \\  -\alpha c_1 &0&0&0&&1& \\ -\frac{1}{2}c_2^2&\alpha c_1&0&\alpha c_2&&&I_{n-2}
\end{psmatrix},
\end{align*}
where $c_1,c_2\in F$, $\alpha=\frac{1}{2}$. This implies that 
\begin{align*}
&|W^{\circ}_{\rho_{\tau,s}}|(u'v(r(e))(w \iota(\bar{n})w^{-1})(w n'w^{-1}),\mathrm{diag}(\lfloor b \rfloor , I_{n-2},m)) \\
=&  |W^{\circ}_{\rho_{\tau,s}}|\left(u'v\left(\begin{psmatrix}
1&-c_1e&c_1 \\ 0&I_{n-2}&0 \\ 0&0&1
\end{psmatrix}\right)\iota(\bar{n})'(w n'w^{-1})v(r(e)),\mathrm{diag}(\lfloor b \rfloor , I_{n-2},m)\right).
\end{align*}
Here
\[
\iota(\bar{n})' = \begin{psmatrix}
1&-c_1e&c_1&&&&\\&I_{n-2}&&&&&\\&&1&&&&\\0&-\alpha c_2e&\alpha c_2&1&&& \\
\alpha c_1&\frac{1}{2}c_2^2e&-\frac{1}{2}c_2^2&-\alpha c_2&1&&-c_1\\
-\alpha c_1e'&e'\frac{1}{2}c_2^2e&\frac{1}{2}c_2^2e'&-\alpha c_2e'&&I_{n-2}&c_1e'\\
0&\alpha c_1e&\alpha c_1&0&&&1
\end{psmatrix},
\]
where $e'=-J_{n-2}\,^te$. 

Since
\begin{align*}
&\begin{psmatrix}
1 & &  &&&& \\ &I_{n-2}& &&&& \\ &&1&&&&& \\
0&v_3'&0&1&&& \\ 0&v_1'&0&0&1&&\\
v_2&T+e'v_1'-ev_1&v_1&v_3&&I_{n-2}&\\
0&v_2'&0&0&&&1
\end{psmatrix} v\left(\begin{psmatrix}
1&-c_1e&c_1 \\ 0&I_{n-2}&0 \\ 0&0&1
\end{psmatrix}\right)  \\ =&  v\left(\begin{psmatrix}
1&-c_1e&c_1 \\ 0&I_{n-2}&0 \\ 0&0&1
\end{psmatrix}\right)\begin{psmatrix}
1 & &  &&&& \\ &I_{n-2}& &&&& \\ &&1&&&&& \\
0&v_3'&0&1&&& \\ 0&v_1'-c_1v_2'&0&0&1&&\\
v_2&T+e'v_1'-ev_1-c_1ev_2&v_1+c_1v_2&v_3&&I_{n-2}&\\
0&v_2'&0&0&&&1
\end{psmatrix},
\end{align*}
we have 
\begin{align*}
&|W^{\circ}_{\rho_{\tau,s}}|\left(u'v\left(\begin{psmatrix}
1&-c_1e&c_1 \\ 0&I_{n-2}&0 \\ 0&0&1
\end{psmatrix}\right)\iota(\bar{n})'(w n'w^{-1})v(r(e)),\mathrm{diag}(\lfloor b \rfloor , I_{n-2},m)\right)\\
=& |W^{\circ}_{\rho_{\tau,s}}|\left(v\left(\begin{psmatrix}
1&-c_1e&c_1 \\ 0&I_{n-2}&0 \\ 0&0&1
\end{psmatrix}\right)u''\iota(\bar{n})'(w n'w^{-1})v(r(e)),\mathrm{diag}(\lfloor b \rfloor , I_{n-2},m)\right)\\
=& |W^{\circ}_{\rho_{\tau,s}}|(u''\iota(\bar{n})'(w n'w^{-1})v(r(e)),\mathrm{diag}(\lfloor b \rfloor , I_{n-2},m)).
\end{align*}
Here
\[
u'' = \begin{psmatrix}
1 & &  &&&& \\ &I_{n-2}& &&&& \\ &&1&&&&& \\
0&v_3'&0&1&&& \\ 0&v_1'-c_1v_2'&0&0&1&&\\
v_2&T+e'v_1'-ev_1-c_1ev_2&v_1+c_1v_2&v_3&&I_{n-2}&\\
0&v_2'&0&0&&&1
\end{psmatrix}.
\]

Thus, using the following change of variables
\begin{align*}
&v_1 \mapsto v_1-c_1v_2-\frac{1}{2}c_2^2e'\\
&v_2 \mapsto v_2-\alpha c_1e \\
&v_3 \mapsto v_3-\alpha c_2e'\\
&T \mapsto T-e'v_1'+ev_1-\frac{1}{2}c_2^2e'e,
\end{align*}
we obtain the integral $I_{s,s_1,s_2}$ is bounded by
\begin{multline*}
\int_{F^{\times}} \int_{F^{\times}} |m|^{-s_1+s_2+\frac{d_2}{2}+\Re(s)+5}|b|^{-s_1}|\lfloor b \rfloor|^{\Re(s)-n+3}
 \\
\times \int_{X_1(F)} \int_{M_{1\times (n-2)}(F)} |W^{\circ}_{\rho_{\tau,s}}|(x_1(w nw^{-1})v(r(e)),\mathrm{diag}(\lfloor b \rfloor , I_{n-2},m))dedx_1d^{\times}md^{\times}b. 
\end{multline*}

Now we proceed to decompose $v(r(e))$. 

We have
\[
\begin{pmatrix}
0&1&0 \\ 0&0&I_{n-2} \\ 1&0&e
\end{pmatrix}  = \begin{pmatrix}
1&0&0 \\ 0&I_{n-2}&0 \\ 0&e&1
\end{pmatrix} \begin{pmatrix}
0&1&0 \\ 0&0&I_{n-2} \\ 1&0&0
\end{pmatrix}.
\]
Then we apply the Iwasawa decomposition with respect to the standard Borel subgroup of $\GL_n(F)$ 
\[
\begin{pmatrix}
1&0&0 \\ 0&I_{n-2}&0 \\ 0&e&1
\end{pmatrix} = n_et_ek_e.
\]
Here $n_e=\mathrm{diag}(I_2,n'_e)$ where $n'_e$ lies in the unipotent radical of the standard Borel subgroup of $\GL_{n-2}(F)$, $t_e=(t_1,\dots,t_n)$, where $t_1=t_2=1$, $k_e\in \GL_n(\OO)$. 

By the structure of this decomposition, we have 
\[
[e] \leq |t_n|=|t_3\cdots t_{n-1}|^{-1}
\]
since $\mathrm{det}(t_e)=1$.

The integral $I_{s_1,s_2}$ is majorized by
\begin{multline*}
\int_{F^{\times}} \int_{F^{\times}} |m|^{-s_1+s_2+\frac{d_2}{2}+\Re(s)+5}|b|^{-s_1}|\lfloor b \rfloor|^{\Re(s)-n+3}\\
\times \int_{X_1(F)} \int_{M_{1\times (n-2)}(F)}  |W^{\circ}_{\rho_{\tau,s}}|(x_1(w nw^{-1})v(n_e), \mathrm{diag}(\lfloor b \rfloor , I_{n-2},m)t_e)dedx_1d^{\times}md^{\times}b.  
\end{multline*}

By change of variables $x_1 \mapsto v(n_e)(x_1(w nw^{-1}))v(n_e)^{-1}$ the above integral is 
\begin{multline*}
\int_{F^{\times}} \int_{F^{\times}}  |m|^{-s_1+s_2+\frac{d_2}{2}+\Re(s)+5}|b|^{-s_1}|\lfloor b \rfloor|^{\Re(s)-n+3}
 \\
\times \int_{X_1(F)} \int_{M_{1\times (n-2)}(F)} |W^{\circ}_{\rho_{\tau,s}}|(x_1(w nw^{-1}), \mathrm{diag}(\lfloor b^2 \rfloor , I_{n-2},m^2)t_e)dedx_1d^{\times}md^{\times}b.  
\end{multline*}

Thus by similar arguments as in \cref{6.1} the integral is majorized by a finite sum of integrals of the form
\begin{multline*}
\int_{F^{\times}} \int_{F^{\times}} |m|^{-s_1+s_2+\frac{d_2}{2}+\Re(s)+5}|b|^{-s_1}|\lfloor b \rfloor|^{\Re(s)-n+3-2C}
 \\
\times \int_{X_1(F)} \int_{M_{1\times (n-2)}(F)} [x_1]^{-\Re(s)-\frac{n-1}{2}+C}\eta_j(\mathrm{diag}(\lfloor b^2 \rfloor , I_{n-2},m^2)t_e)dedudyd^{\times}md^{\times}b.  
\end{multline*}
Here $\eta_j$ is some positive quasi-character depends on $\tau$ and $C$ is some positive integer which depends on $\tau$.

Using the notation as in \cref{6.1}, we denote the Iwasawa decomposition of $x_1 \in X_1(F)$ as $x_1=na'k$. Then we have that $\mathrm{diag}(\lfloor b \rfloor , I_{n-2},m)at_e$ lies in the support of a gauge on $\GL_n(F)$, we have
\[
\left|\frac{a_2}{a_3t_3}\right| \leq 1, \quad \left|\frac{a_it_i}{a_{i+1}t_{i+1}}\right| \leq 1, \quad \left|\frac{t_{n-1}}{mt_n}\right| \leq 1,
\]
where $i=3,\dots,n-2$.

Thus, by similar arguments as in \cref{6.1} we have
\[
[e] \leq |t_3\cdots t_{n-1}|^{-1} \leq [x_1]^{C'}|\lfloor b \rfloor|^{-2C'},
\]
where $C'$ is some positive integer. By \cite[Proposition 11.15, Lemma 2]{Sou} we have
\[
\max\left\{\left|\frac{t_i}{t_{i+1}}\right|,\left|\frac{t_{i+1}}{t_1}\right|\right\} \leq [e]^{2n} \leq [x_1]^{2nC'}|\lfloor b \rfloor|^{-4nC'}.
\]
Thus we have
\[
|m|\geq [x_1]^{-C_0}|\lfloor b \rfloor|^{2C_0},
\]
where $C_0$ is a positive integer depends on $\tau$. 

Then, the integral $I'_{s_1,s_2}$ is majorized by a finite sum of integrals of the form
\begin{multline*}
\int_{F^{\times}}  |b|^{-s_1}|\lfloor b \rfloor|^{\Re(s)-n+3-c_2}\\
\times \int_{X_1(F)} \int_{|m|\geq [x_1]^{-C_0}|\lfloor b \rfloor|^{2C_0}}
 |m|^{-s_1+s_2+\frac{d_2}{2}+\Re(s)+5+c_1} [x_1]^{-\Re(s)-\frac{n-1}{2}+c_3}dx_1d^{\times}md^{\times}b,
\end{multline*}
where $c_2,c_3>0$, $c_1$ are constants depend on $\tau$.

Similar as in \cref{6.3}, the above integral converges absolutely when 
\begin{align} \label{region} \begin{split}
&-\Re(s_1)+\Re(s_2)+\frac{d_2}{2}+\Re(s)+5+c_1 < 0,\\
&-C_0(-\Re(s_1)+\Re(s_2)+\frac{d_2}{2}+\Re(s)+5+c_1)-\Re(s)-\frac{n-1}{2}+c_3 < -C'',\\
& 2C_0(-\Re(s_1)+\Re(s_2)+\frac{d_2}{2}+\Re(s)+5+c_1)+\Re(s)-n+3-c_2-\Re(s_1) > 0,\\
& -2C_0(-\Re(s_1)+\Re(s_2)+\frac{d_2}{2}+\Re(s)+5+c_1)+\Re(s)-n+3-c_2-\Re(s_1) < 0, \end{split}
\end{align}
where $C''$ is a positive integer which depends on $\tau$. Thus we deduce the lemma.

Then above region (\ref{region}) is simplified to  
\begin{align}\label{region3} \begin{split}
\Re(s) <& \Re(s_1) + \Re(-s_2) -\frac{d_1}{2} -5 -c_1 \\
\Re(s) >& \max\bigg(\frac{C_0}{C_0+1}(\Re(s_1)+\Re(-s_2)-\frac{d_2}{2}-5-c_1)-\frac{n-1-C''}{C_0+1},\\
& \Re(s_1)+\frac{2C_0}{2C_0+1}(\Re(-s_2)-\frac{d_2}{2}-5-c_1)-\frac{n-3+c_2}{2C_0+1}, \\
& \Re(s_1)+\frac{2C_0}{2C_0-1}(\Re(-s_2)-\frac{d_2}{2}-5-c_1)-\frac{n-3+c_2}{2C_0-1}\bigg).
\end{split}
\end{align}
For this to be non-empty we need 
\begin{align*}
\Re(s_1) - \Re(s_2) -\frac{d_1}{2} -5 -c_1 >& \max\bigg(\frac{C_0}{C_0+1}(\Re(s_1)+\Re(-s_2)-\frac{d_2}{2}-5-c_1)-\frac{n-1-C''}{C_0+1},\\
& \Re(s_1)+\frac{2C_0}{2C_0+1}(\Re(-s_2)-\frac{d_2}{2}-5-c_1)-\frac{n-3+c_2}{2C_0+1}, \\
& \Re(s_1)+\frac{2C_0}{2C_0-1}(\Re(-s_2)-\frac{d_2}{2}-5-c_1)-\frac{n-3+c_2}{2C_0-1}\bigg).
\end{align*}
For the above to be valid we need 
\begin{align*}
& -n+\frac{d_2}{2}+8+c_1 < \Re(-s_2) < n+\frac{d_2}{2}+2+c_1+c_2, \\
& c_0(-n+3-c_2)+C'' < \Re(s_1)+\Re(-s_2).
\end{align*}
Since $n\geq 3$, $-n+\frac{d_2}{2}+8+c_1 < n+\frac{d_2}{2}+2+c_1+c_2$, the above inequalities are valid. Thus region (\ref{region3}) is non-empty. \end{proof}

\begin{lemma} \label{6.5}
The infinite sum 
\[
\sum_{k=\mathrm{val}(4\mathcal{Q'}(y_2))}^{\infty} \frac{\prod_{i=1}^2c(q^{s_2},y_i)\zeta_v(-s_2-\frac{d_1}{2}+2)^2 \zeta_v(-s_2)^2 q^{(-s_1+s_2+n-2+\frac{d_1}{2})k}}  
{\gamma(s-s_1+s_2, \chi'\otimes \tau)} B_{W^{\circ}_{\rho_{\tau,s}}}(\varpi_H^{\delta_{k-\mathrm{val}(4\mathcal{Q'}(y_2)}})
\]
converges absolutely when
\begin{align} \label{region1}
\Re(s_1)-\Re(s_2)-2-\frac{d_1}{2} \leq \Re(s) \leq \Re(s_1)-\Re(s_2)+n+1-\frac{d_1}{2},
\end{align}
\begin{align} \label{region2}
\Re(s_1) > C_1, \Re(-s_2) > C_2.
\end{align}
Here $C_1,C_2$ are constants depends on $(n,d_1,d_2)$.
\end{lemma} 

\begin{proof}
By the formula of $c_{s_2}$ and $B_{\psi'_1, s_1, s_2}$, it suffices to show
\[
 \sum_{k=\mathrm{val}(4\mathcal{Q'}(y_2))}^{\infty}  q^{(-s_1+s_2-n-1+\frac{d_2}{2})k} \zeta_v(-s_2-\frac{d_1}{2}+2)^2 \zeta_v(-s_2)^2 \gamma(s-s_1+s_2, \chi'\otimes \tau)^{-1}\sum_{w \in W} \,^{w}\chi_s(\varpi^{\delta_{k}})^{-1}
\]
converges absolutely in the region given by \ref{region1} and \ref{region2}.

Note that $\zeta_v(-s_2-\frac{d_1}{2}+2)^2 \zeta_v(-s_2)^2$ converges when
\[
\Re(-s_2)>0, \Re(-s_2-\frac{d_1}{2}+2)>0,
\]
and $\gamma(s-s_1+s_2, \chi'\otimes \tau)^{-1}$ converges when
\[
\Re(s-s_1+s_2+\frac{d_1+5}{2}) \geq \frac{1}{2}
\]
which lies in the given region. It remains to show the sum
\[
\sum_{k =0}^{\infty}  q^{-(s_1-s_2+n+1-\frac{d_1}{2})k} \sum_{w \in W} \,^{w}\chi_s(\varpi^{\delta_{k}})^{-1}
\]
converges absolutely in the region. 

We have
\[
|\chi_s(\varpi^{\delta_{k}})^{-1}| = q^{(s-\frac{1}{2})k}|\chi_{1}(\varpi^{-k})|.
\]
By \cite[Corollary 2.5]{JS1} we have 
\[
|\,^{w}\chi_{1}(\varpi^{-k})| < q^{\frac{k}{2}}
\]
for any $w \in W$ and $k > 0$. Then it suffices to observe that
\[
\sum_{k =0}^{\infty}   q^{-(s_1-s_2+n+1-\frac{d_1}{2})k} q^{sk}
\]
and
\[
\sum_{k =0}^{\infty}  q^{-(s_1-s_2+n+1-\frac{d_1}{2})k} q^{-(s-1)k}
\]
converges in the given region because they are convergent geometric series. 
The convergence region is obviously non-empty since $n$ is positive. 
\end{proof}

\newpage
\section*{List of symbols} \label{symbols}

\begin{tabular*}{\textwidth}{@{\extracolsep{\fill}} c  c  c}
$A_1$ & subgroup of $T_G$ & \eqref{A_1}\\
$A_2$ & subgroup of $T_G$ & \eqref{A_2}\\
$a_y$ & $\mathrm{diag}(-4\mathcal{Q}'(y_2),I_{n-1}) \in \GL_n$ & \eqref{a_y}\\
$B_{\psi'_1, s_1}$ & unramified Bessel function & \ref{Bessel} \\
$G$ & $\{g=(g_1,g_2) \in \GL_2^2(R): \det g_1=\det g_2^{-1}\}$  & \eqref{G,H} \\
$G'$ & $\SO_4$ & \eqref{G'}\\ 
$G_1$ & subgroup of $T_G$ & \eqref{G_1}\\
$H$ & $\SO_{2n+1}$ & \eqref{G,H}\\
$\iota$ & embedding map from G to H & \eqref{iota}\\
$I(f, W_{\xi_s})$ & global integral & \ref{Ifxis}\\
$M_1$ & subgroup of $T_G$ & \ref{M_1} \\
$M_{\SL_2^2}$ & subgroup of maximal torus of $\SL_2 \times \SL_2$ & \eqref{M_SL_2^2}\\
$M_n$ & Levi subgroup of $Q_n$ & \eqref{M_n} \\
$\mu$ & irreducible unramified character of $\SO_2$ & \ref{4}\\
$N_1$ & subgroup of $U_2$ & \ref{N1} \\
$N_2$ & subgroup of $U_2$ & \ref{N2} \\ 
$N^{\circ}$ & unipotent subgroup of $H$ & \eqref{Ncirc}\\
$N_n$ & unipotent radical of $Q_n$  & \eqref{N_n}\\
$\overline{N_n}$ & opposite unipotent radical of $\overline{Q_n}$  & \ref{2.1.2}\\
$\mathbb{P}Y'$ & quasi-projective subscheme of $Y'$ & \ref{1}\\
$\mathcal{Q}$ & quadratic form on $V_1$ & \ref{3}\\
$\mathcal{Q'}$ & quadratic form on $V_2$ & \ref{3}\\
$Q_n$ & standard parabolic subgroup of $H$ & \ref{2.1.2}\\
$\overline{Q_n}$ & opposite parabolic subgroup of $H$ &\ref{2.1.2} \\
$w$ & Weyl group element of $H$ & \eqref{omega}\\
$\rho$ & Weil representation & \ref{3}\\
$\tau$ & irreducible cuspidal representation of $\GL_n$& \ref{3}\\
$T_G$ & maximal torus of $G$ & \ref{2.1}\\
$T_H$ & maximal torus of $H$ & \ref{2.1}\\
$\Theta_f$ & Theta function & \ref{3}\\
$U_2$ & maximal unipotent subgroup of $G$ & \eqref{U_2}\\
$\overline{U_2}$ & opposite of $U_2$ & \ref{4}\\
$V$ & $V_1 \times V_2$ & \ref{1}\\
$V_i$ & quadratic space of even dimension& \ref{1}\\
$W_{\xi_s}$ & Whittaker function on $\GL_n$ when restricted to a Levi subgroup of $H$ & \ref{Whittaker} \\
$W_{\rho_{\tau,s}}$ & local vector in $\mathrm{Ind}_{Q_n}^{H}(\mathcal{W}(\tau, \psi_0) \otimes |\det |^{s-\tfrac{1}{2}})$ & \ref{local vector}\\
$\xi_s$ & global smooth holomorphic section in the space $\mathrm{Ind}_{Q_n}^{H}(\tau \otimes |\det|^{s-\frac{1}{2}})$ & \ref{3}\\
$Y$ & $\{v\in V(R): Q(v_1)=2 Q'(V_2)\}$ & \eqref{Y}\\
$Y'$ & subscheme of $Y$ such that no $y_i=0$ & \ref{1}\\
\end{tabular*}

\newpage
\bibliography{ATSFPQS}{}
\bibliographystyle{alpha}

\end{document}